\documentclass[a4paper,11pt,bibtotoc,tocindent,idxtotoc,makeidx]{article}

\usepackage[english]{babel}
\usepackage{amsmath}
\usepackage{amsthm}
\usepackage{amssymb}
\usepackage{amsfonts}
\usepackage{amsxtra}
\usepackage{color}
\usepackage[breaklinks]{hyperref} 
\usepackage{enumitem}
\usepackage{tocloft}
\usepackage{makeidx}
\usepackage{textcomp}
\usepackage{rotating}
\usepackage{multirow}
\usepackage{booktabs,longtable}
\usepackage[all]{xy}
\usepackage{float}

\setenumerate{itemsep=0pt} 
\setitemize{itemsep=0pt} 
\setdescription{itemsep=0pt} 
\setlist{noitemsep}

\makeindex

\newcommand{\td}{\,\mathrm{d}}

\newcommand{\Stab}{\textup{Stab}}
\newcommand{\const}{\textup{const}}

\newcommand{\Ad}{\textup{Ad}}
\newcommand{\ad}{\textup{ad}}

\renewcommand{\det}{\textup{det}}
\newcommand{\id}{\textup{id}}
\newcommand{\tr}{\textup{tr}}
\newcommand{\rk}{\textup{rk}}
\newcommand{\Tr}{\textup{Tr}}
\newcommand{\Sym}{\textup{Sym}}
\newcommand{\Skew}{\textup{Skew}}
\newcommand{\Herm}{\textup{Herm}}
\newcommand{\Str}{\textup{Str}}
\newcommand{\str}{\mathfrak{str}}

\newcommand{\Co}{\textup{Co}}
\newcommand{\co}{\mathfrak{co}}
\newcommand{\GL}{\textup{GL}}
\newcommand{\gl}{\mathfrak{gl}}
\newcommand{\SL}{\textup{SL}}
\renewcommand{\sl}{\mathfrak{sl}}

\newcommand{\Sp}{\textup{Sp}}
\renewcommand{\sp}{\mathfrak{sp}}
\newcommand{\Mp}{\textup{Mp}}
\newcommand{\upO}{\textup{O}}
\newcommand{\SO}{\textup{SO}}
\newcommand{\so}{\mathfrak{so}}
\newcommand{\SU}{\textup{SU}}
\newcommand{\su}{\mathfrak{su}}
\newcommand{\End}{\textup{End}}
\newcommand{\RR}{\mathbb{R}}
\newcommand{\KK}{\mathbb{K}}
\newcommand{\CC}{\mathbb{C}}
\newcommand{\ZZ}{\mathbb{Z}}
\newcommand{\NN}{\mathbb{N}}
\newcommand{\QQ}{\mathbb{Q}}

\newcommand{\HH}{\mathbb{H}}
\newcommand{\OO}{\mathbb{O}}
\newcommand{\DD}{\mathbb{D}}
\newcommand{\XX}{\mathbb{X}}

\newcommand{\lf}{\mathfrak{l}}
\newcommand{\gf}{\mathfrak{g}}
\newcommand{\nf}{\mathfrak{n}}
\newcommand{\nfo}{\overline{\mathfrak{n}}}
\newcommand{\kf}{\mathfrak{k}}

\renewcommand{\1}{\mathbf{1}}
\newcommand{\Ind}{\textup{Ind}}

\newcommand{\calF}{\mathcal{F}}
\newcommand{\calO}{\mathcal{O}}
\newcommand{\calS}{\mathcal{S}}
\newcommand{\calW}{\mathcal{W}}
\newcommand{\calB}{\mathcal{B}}
\newcommand{\calH}{\mathcal{H}}
\newcommand{\calP}{\mathcal{P}}
\newcommand{\calU}{\mathcal{U}}
\newcommand{\calL}{\mathcal{L}}

\newcommand{\calD}{\mathcal{D}}

\newcommand{\calV}{\mathcal{V}}

\newcommand{\calJ}{\mathcal{J}}

\newcommand{\frakg}{\mathfrak{g}}
\newcommand{\fraku}{\mathfrak{u}}
\newcommand{\frakk}{\mathfrak{k}}
\newcommand{\frakp}{\mathfrak{p}}
\newcommand{\fraks}{\mathfrak{s}}
\newcommand{\frakn}{\mathfrak{n}}

\newcommand{\frakq}{\mathfrak{q}}
\newcommand{\frakm}{\mathfrak{m}}
\newcommand{\frakl}{\mathfrak{l}}
\newcommand{\frakt}{\mathfrak{t}}

\newcommand{\frakh}{\mathfrak{h}}
\newcommand{\Det}{\textup{Det}}
\newcommand{\rank}{\textup{rank}}
\renewcommand{\mod}{\textup{mod}}
\newcommand{\rad}{\textup{rad}}
\newcommand{\res}{\textup{res}}

\newcommand{\spec}{\textup{spec}}
\newcommand{\Ann}{\textup{Ann}}

\newcommand{\gr}{\textup{gr}}
\newcommand{\Hom}{\textup{Hom}}
\newcommand{\diag}{\textup{diag}}

\theoremstyle{plain}
\newtheorem{theorem}{Theorem}[section]
\newtheorem*{theoremA}{Theorem A}
\newtheorem*{theoremB}{Theorem B}
\newtheorem*{theoremC}{Theorem C}
\newtheorem{proposition}[theorem]{Proposition}
\newtheorem{lemma}[theorem]{Lemma}
\newtheorem{corollary}[theorem]{Corollary}

\theoremstyle{definition}
\newtheorem{definition}[theorem]{Definition}
\newtheorem{example}[theorem]{Example}

\newtheorem{remark}[theorem]{Remark}

\theoremstyle{remark}
\renewenvironment{pf}{\begin{proof}[\sc Proof]}{\end{proof}}

\numberwithin{equation}{section}

\date{}

\begin{document}

\nocite{howe}


\title{Minimal representations via Bessel operators}
\author{Joachim Hilgert\footnote{Part of this research was done at the Hausdorff Research Institute for Mathematics in the context of the trimester program ``Interaction of Representation Theory with Geometry and Combinatorics''}, Toshiyuki Kobayashi\footnote{Partially supported by Grant-in-Aid for Scientific Research (B) (22340026), Japan Society for the Promotion of Science, and the Alexander Humboldt Foundation.}\ \footnotemark[1], Jan M\"ollers\footnote{Partially supported by the International Research Training Group 1133 ``Geometry and Analysis of Symmetries'', and the GCOE program of the University of Tokyo.}}
\maketitle 

\begin{abstract}

We construct an $L^2$-model of \lq\lq very small" irreducible unitary representations of simple Lie groups $G$ which, up to finite covering, occur as conformal groups $\Co(V)$ of simple Jordan algebras $V$. 
If $V$ is split and $G$ is not of type $A_n$, then the representations are minimal in the sense that the annihilators are the Joseph ideals. Our construction allows the case where $G$ does not admit minimal representations.
In particular, applying to Jordan algebras of split rank one we obtain the entire complementary series representations of $\SO(n,1)_0$.
A distinguished feature of these representations in all cases is that they attain the minimum of the Gelfand--Kirillov dimensions among irreducible unitary representations.
Our construction provides a unified way to realize the irreducible unitary representations of the Lie groups in question as Schr\"odinger models in $L^2$-spaces on Lagrangian submanifolds of the minimal real nilpotent coadjoint orbits.
In this realization the Lie algebra representations are given explicitly by differential operators of order at most two, and the key new ingredient is a systematic use of specific second-order differential operators (\textit{Bessel operators}) which are naturally defined in terms of the Jordan structure. 
\\



\end{abstract} 

\newpage

\tableofcontents

\newpage


\addcontentsline{toc}{section}{Introduction}
\section*{Introduction}

Minimal representations are building blocks of unitary representations. Every unitary representation can be built up from irreducible unitary representations by means of direct integrals. Further, the Kirillov--Kostant--Duflo--Vogan orbit philosophy suggests that a large part of irreducible unitary representations should be constructed from ``unipotent representations'' by using classical or cohomological induction functors. Minimal representations show up as the smallest kind of unipotent representations, and cannot be constructed by
the existing induction functors in general.\\

The subject of this paper is a unified construction of $L^2$-models for a family of ``smallest'' irreducible unitary representations including minimal representations. A key feature of those representations is that they attain the minimum of the Gelfand--Kirillov dimensions among all irreducible infinite dimensional unitary representations. This is reflected by the fact that in $L^2$-models of these representations we cannot expect geometric actions, and consequently the Lie algebra does not act by vector fields. For a general program of $L^2$-models and conformal models of minimal representations of real reductive groups, we refer to \cite[Chapter 1]{KM07b}. It should be noted that there is no known straightforward way to construct minimal representations.

Our construction is effected in the framework of Jordan algebras. To each simple real Jordan algebra $V$ with simple maximal euclidean subalgebra $V^+$ we associate its conformal group $\Co(V)$ which is a simple real Lie group. Let $G$ denote its identity component and $\frakg$ the Lie algebra of $G$. The structure of $V$ provides a Lagrangian submanifold $\calO$ of a real nilpotent coadjoint orbit $\calO_{\min}^G$ of minimal dimension (see Theorem \ref{lem:Lagrange}).

On $C^\infty(\calO)$ there is a natural representation of the maximal parabolic subalgebra $\frakq^{\max}=\str(V)\ltimes V$ of $\frakg$ by differential operators up to order $1$. The non-trivial part is to extend it to the semisimple Lie algebra $\frakg=\frakq^{\max}\oplus\overline{\frakn}$, and then to lift it to a Lie group with Lie algebra $\frakg$. The novelty here is a systematic use of a differential operator of order two, which we refer to as the ``Bessel operator''. The Bessel operator was originally studied for euclidean Jordan algebras (see e.g.\ \cite{FK94}) in a different context and for $V=\RR^{p,q}$ in \cite{KM07b}. Using the quadratic representation $P$ of the Jordan algebra, we define in \eqref{eq:BesselOp} the Bessel operator $\calB_\lambda:C^\infty(V)\rightarrow C^\infty(V)\otimes V$ by
\begin{align*}
 \calB_\lambda := P\left(\frac{\partial}{\partial x}\right)x+\lambda\frac{\partial}{\partial x}.
\end{align*}
For the special value $\lambda=\lambda_1\in\QQ$ (see Section \ref{sec:BesselMinO}) the operator $\calB:=\calB_\lambda$ is tangential to the submanifold $\calO$. The resulting $V$-valued differential operator may be interpreted as a family of second order differential operators $_\phi\calB:=\langle\phi,\calB\rangle$ on $\calO$ parameterized by $\phi\in V^*\cong\overline{\frakn}$. This family of operators complements the action of $\frakq^{\max}$ to define a Lie algebra representation $\td\pi$ of the semisimple Lie algebra $\frakg$ on $C^\infty(\calO)$. We note that this in fact defines a Lie algebra representation of $\frakg$ on the space of sections of any flat vector bundle over $\calO$.

To integrate the representation $\td\pi$ of the Lie algebra $\frakg$ to a representation of a Lie group, we use the specific generator
\begin{align*}
 \psi_0(x) &:= \widetilde{K}_{\frac{\nu}{2}}(|x|), & x\in\calO,
\end{align*}
which traces back to \cite{DS99,KO03c,Sah92}. Here we have renormalized the $K$-Bessel function as $\widetilde{K}_\alpha(z):=(\frac{z}{2})^{-\alpha}K_\alpha(z)$ following \cite{KM07b, KM07a} and the parameter $\nu\in\ZZ$ is defined in terms of the structure constants of $V$ (see Section~\ref{sec:MuNu} or Table~\ref{tb:MuNu}). In the case of $V=\Sym(k,\RR)$ the isotropy subgroup of the structure group on $\calO$ is disconnected, and we also use the specific generator
\begin{align*}
 \psi_0^-(x) &:= (x|c_1)^{\frac{1}{2}}\widetilde{K}_{\frac{\nu}{2}}(|x|) = \frac{\sqrt{\pi}}{2}\sqrt{x_{11}}e^{-|x|}, & x\in\calO,
\end{align*}
for the line bundle $\calL\to\calO$ associated to the sign representation (see \eqref{eq:DefPsi-}).

\begin{theoremA}[Theorems \ref{prop:Kfinite} and \ref{thm:IntgkModule}]\label{thm:A}
Assume that the split rank $r_0$ of the real simple Jordan algebra $V$ is larger than one.
\begin{enumerate}
\item[\textup{(1)}] The following conditions on $\frakg$ are equivalent:
\begin{enumerate}
\item[\textup{(i)}] $\psi_0$ is $\frakk$-finite.
\item[\textup{(ii)}] $\frakg\ncong\so(p,q)$ with $p+q$ odd, $p,q\geq3$.
\end{enumerate}
\end{enumerate}
Let $W$ be the subrepresentation of $(\td\pi,C^\infty(\calO))$ generated by $\psi_0$.
\begin{enumerate}
\item[\textup{(2)}] If the equivalent conditions in (1) are satisfied then $W$ is a dense subspace of the Hilbert space $L^2(\calO)$ and $\td\pi$ integrates to an irreducible unitary representation $\pi$ of a finite covering group of $G$ on $L^2(\calO)$.
\end{enumerate}
For $V=\Sym(k,\RR)$ let $W^-$ be the subrepresentation of $(\td\pi,C^\infty(\calO,\calL))$ generated by $\psi^-_0$.
\begin{enumerate}
\item[\textup{(3)}] $W^-$ is a dense subspace of the Hilbert space $L^2(\calO,\calL)$ of square integrable sections and $\td\pi$ integrates to an irreducible unitary representation $\pi^-$ of a finite covering group of $G$ on $L^2(\calO,\calL)$.
\end{enumerate}
\end{theoremA}

The minimal covering groups for Theorem~A\,(2) \&\ (3), to be denoted by $G^\vee$ and $G^\vee_-$, are given in Definition \ref{def:minG}.

We shall also write $\pi^+$ for $\pi$, $W^+$ for $W$ and $\psi_0^+$ for $\psi_0$ to state results that include the representation $\pi^-$. Under the equivalent conditions in Theorem~A\,(1), the $\gf$-module $W^\pm$ is multiplicity-free as a $\frakk$-module. In light of the $\frakk$-type formula in Theorem \ref{thm: K-type formula}, $\psi_0^\pm$ belongs to the minimal $\frakk$-type of $\pi^\pm$.\\

For a real simple Lie algebra $\frakg$ there exists a unique minimal complex nilpotent orbit in $\frakg_\CC$, to be denoted by $\calO_{\min,\frakg}^{G_{\CC}}$, having real points (see Proposition~\ref{prop:nilcpx}). Then the nilpotent coadjoint orbit $\calO_{\min}^G$ is a connected component of $\calO_{\min,\frakg}^{G_\CC}\cap\frakg^*$, and our construction provides the smallest irreducible unitary representations on the Hilbert space $L^2(\calO)$ consisting of square integrable functions on a Lagrangian submanifold $\calO$ of the symplectic manifold $\calO_{\min}^G$ in the following sense.

\begin{theoremB}[Theorem \ref{thm:Minimality} and Corollary \ref{cor:UniqueMinRep}]\label{thm:A} We assume that the equivalent conditions of Theorem~A\,(1) hold. Let $\pi$ be the representation $\pi^\pm$ constructed in Theorem A, $\td\pi^\pm$ its differential representation on the space of smooth vectors, and $\mathcal J_\pi$ the annihilator ideal of the representation $\td\pi^\pm$ in the enveloping algebra of $\gf_{\CC}$.
\begin{enumerate}
\item[\textup{(1)}]
  $\mathcal J_\pi$ is completely prime and its associated variety $\mathcal V(\mathcal J_\pi)$ is equal to the closure of $\calO_{\min,\frakg}^{G_\CC}$ in $\frakg_\CC^*$.
\item[\textup{(2)}] If $V$ is a split Jordan algebra or a complex Jordan algebra then $\calO_{\min,\frakg}^{G_\CC}$ is a minimal nilpotent orbit. If in addition $\frakg_\CC$ is not of type $A_n$ then $\pi^\pm$ is a minimal representation in the sense that $\mathcal J_\pi$ is the Joseph ideal (cf. Definition~\ref{def:minrep}). Conversely, every minimal representation of any covering group of $G^\vee$ is equivalent to one of the representations constructed in Theorem~A or its dual.
\end{enumerate}
\end{theoremB}

Concerning the equivalent conditions of Theorem~A\,(1), it is noteworthy that
 there is no minimal representation for any group $G$ with Lie algebra $\frakg\cong\so(p,q)$ with $p+q$ odd, $p,q\geq4$ (see \cite[Theorem 2.13]{Vog81}).
We also remark that contrary to what was stated in \cite[page 206]{DS99}
the $L^2$-model of the minimal representations exist for the group $O(p,q)$
with $p+q$ even, $p,q \geq 2$ (see \cite{KO03c}).\\

Our construction also applies to the case of split rank one. However, in contrast to the cases of higher split rank, in that case there exists a one-parameter family of  measures $\td\mu_\lambda$ on the Lagrangian manifold $\calO$ which are equivariant under the structure group of the Jordan algebra. Correspondingly, we obtain a one-parameter family of  irreducible unitary representations of $G$ on $L^2(\calO,\td\mu_\lambda)$ for a bounded interval of parameters (Theorem~\ref{thm:IntgkModule}). As is well-known, the Lorentz group $G=SO(n,1)_0$ has a `long' complementary series representation, leading to a failure of Kazhdan's property $(T)$. On the other hand, it is notorious by experience that the orbit philosophy does not work well for complementary series representations. Remarkably, our construction provides the entire complementary series of this group in a way that fits with the orbit philosophy. \\

For all minimal representations appearing in Theorem~B\,(2), one can find $L^2$-models in each specific case in the literature
 (\cite{BSZ06,DS99,KM07b,KO03c,LV80,Sah92}). Other important papers on the construction of minimal representations are  Brylinski--Kostant \cite{BK94} and  Torasso \cite{Tor97}, which, however, do not contain simple and explicit formulas for the $\frakg$-action.

Whereas the model in \cite{BK94} is built on the $K_\CC$-minimal nilpotent orbit $\calO_{\min}^{K_\CC}$ in $\frakk_\CC^\perp$, our model may be thought of as a geometric quantization of the $G$-minimal nilpotent orbit $\calO_{\min}^G$ in $\frakg^*$ which is the counterpart of $\calO_{\min}^{K_\CC}$ via the Kostant--Sekiguchi correspondence. The advantage of our model is that not only the Hilbert structure, but also the Lie algebra action is simple and explicit by means of the Bessel operators.\\

The Bessel operators are needed already in the construction of the Lie algebra representation. In the course of the proof we show the following properties:
\begin{enumerate}
\item[\textup{(1)}] The operators $_\phi\calB$, $\phi\in V^*$, commute.
\item[\textup{(2)}] For each $\phi\in V^*$ the operator $_\phi\calB$ is symmetric on $L^2(\calO)$.
\end{enumerate}
In the cases where the Lie algebra representation integrates to a unitary representation of $G$, the operator $_\phi\calB$ has a self-adjoint extension for every $\phi\in V^*$. This brings us to the study of a new family of special functions associated with an explicit fourth order ordinary differential operator $\calD_{\alpha,\beta}$ corresponding to the Casimir operator of $\frakk$
 (see \cite{HKMM09a,HKMM09b,KM10}).\\

A further remarkable feature of the Bessel operators is the following refinement of the property (1):

\begin{theoremC}[Theorem \ref{thm:BesselRing}]
Suppose that one of the equivalent conditions in Theorem~A\,(1) is satisfied. Then the ring of differential operators on $\calO$ generated by the Bessel operators $_\phi\calB$, $\phi\in V^*$, is isomorphic to the ring of functions on $\calO$ which are restrictions of polynomials on $V$.
\end{theoremC}

Theorem C generalizes the results for $\gf\simeq\so(p,q)$ with $p+q$ even (see \cite[Chapter 2]{KM07b}), and follows from Theorem A for general $\gf$.\\

Our paper is organized as follows. In Section \ref{sec:BesselOperators} we briefly recall some Jordan theory necessary to define the Bessel operators and give a proof for the fact that they restrict to differential operators on the orbits of the structure group (Theorem~\ref{thm:BlambdaTangential}). Further, we show that they are symmetric operators with respect to the $L^2$ inner products corresponding to certain equivariant measures on the orbits.

Section \ref{ch:MinRep} is the heart of the paper. In Subsection \ref{sec:ConformalGroup} we relate the Jordan theoretic orbits with the minimal nilpotent orbits of the complexified groups. The main result here is Theorem \ref{thm:min-split-nilpotent-orbits} which determines the non-zero minimal nilpotent orbit. Subsection \ref{sec:ConstructionL2Model} contains the general construction of the Lie algebra representation on the Lagrangian submanifold $\calO$, and the proof of Theorems A and B. In Subsection \ref{sec:ExReps} we illustrate the construction by discussing the examples of the Segal--Shale--Weil representation and the minimal representation of $\upO(p,q)$.

In Section \ref{sec:Complements} we explain how our main results are related to previous work. Subsection \ref{sec:PrincipalSeries} is included to clarify the relation of our construction to the use of degenerate principal series representations. In Subsection \ref{sec:SpecialFcts} we find explicit $K$-finite $L^2$-functions for every $K$-type by means of the `special functions' we associated to certain order four differential operators in \cite{HKMM09a,HKMM09b,KM10}. Finally, in Subsection~\ref{sec: inversion operator} we prove Theorem C.\\

\emph{Acknowledgement:} It is a pleasure to thank G.~Mano, T.~Okuda and B.~\O{}rsted for helpful discussions on various aspects of this paper. We further thank the referee for careful reading.\\

Notation: $\NN=\{0,1,2,\ldots\}$, $\mathbb{R}_+=\{x\in\mathbb{R}:x>0\}$.

\section{Bessel operators}\label{sec:BesselOperators}

In this section we introduce the framework for the construction of minimal representations, namely the Hilbert spaces on which the minimal representations are realized and the Bessel operators which describe the crucial part of the Lie algebra action. For this we first introduce some basic structure theory for Jordan algebras needed in the construction. To each semisimple Jordan algebra one can associate its structure group which acts linearly on the Jordan algebra. Its minimal non-zero orbit provides the geometry of the representation space. We then introduce the Bessel operators and show that they are tangential to this orbit and symmetric with respect to a certain $L^2$-inner product.

The notation follows \cite{FK94} where most results of this chapter can be found, although only for the special case of euclidean Jordan algebras. A more detailed version of this material can be found in \cite[Chapter 1]{Moe10}. We thank G. Mano \cite{Man08} for sharing his ideas on Bessel operators with us.

\subsection{Jordan algebras and their structure constants}

The algebraic framework on our construction of $L^2$-models for minimal representations is the framework of Jordan algebras. We briefly recall the basic structure theory of real Jordan algebras to fix the notation.

\subsubsection{Jordan algebras}\label{sec:JordanAlgebras}

Let $V$ be a real or complex Jordan algebra with unit ${\bf e}\in V$. We denote by $L(x)\in\End(V)$\index{Lx@$L(x)$} the multiplication by $x\in V$. The operator
\begin{align*}
 P(x) &:= 2L(x)^2-L(x^2)\index{Px@$P(x)$}
\end{align*}
is called \textit{quadratic representation} and its polarized version is given by
\begin{align*}
 P(x,y) &= L(x)L(y)+L(y)L(x)-L(xy).\index{Pxy@$P(x,y)$}
\end{align*}
Further, the \textit{box operator} $u\Box v$\index{uBoxv@$u\Box v$} is defined by
\begin{align*}
 u\Box v &:= L(uv) + [L(u),L(v)].
\end{align*}
Denote by $n$\index{n@$n$} the dimension of $V$ and by
$r$\index{r@$r$} its \textit{rank}, i.e. the degree of a generic
minimal polynomial (see e.g. \cite[Section II.2]{FK94}). The
\textit{Jordan trace} $\tr(x)$\index{trx@$\tr(x)$} is a linear form on
$V$ and the \textit{Jordan determinant}
$\det(x)$\index{detx@$\det(x)$} is a homogeneous polynomial of degree $r$. To avoid confusion, we write $\Tr$\index{Tr@$\Tr$} and $\Det$\index{Det@$\Det$} for the usual trace and determinant of an endomorphism. Jordan trace and determinant can be written as the usual trace, respectively determinant, of certain operators on $V$:
\begin{align*}
 \tr(x) &= \frac{r}{n}\Tr\,L(x), & x&\in V,\\
 \det(x) &= (\Det\,P(x))^{\frac{r}{2n}}, & x&\in V.
\end{align*}

The symmetric bilinear form
\begin{align*}
 \tau(x,y) &:= \tr(xy), & x,y\in V,\index{tauxy@$\tau(x,y)$}
\end{align*}
is called the \textit{trace form} of $V$. It is associative, i.e. $\tau(xy,z)=\tau(x,yz)$ for all $x,y,z\in V$. If $V\neq0$ and $\tau$ is non-degenerate, we call $V$ \textit{semisimple}. Further, $V$ is called \textit{simple} if it is semisimple and has no non-trivial ideal.

For the remaining part of this subsection we assume that $V$ is real and simple. If $\tau$ is positive definite, we call $V$ \textit{euclidean}. To also obtain an inner product for general $V$ we choose a \textit{Cartan involution}\index{alpha@$\vartheta$} of $V$, i.e. an involutive automorphism $\vartheta$ of $V$ such that the symmetric bilinear form
\begin{equation}\label{eq:traceform}
 (x|y) := \tau(x,\vartheta(y))\index{1bracket@$(-"|-)$}
\end{equation}
is positive definite. Such a Cartan involution always exists and two Cartan involutions are conjugate by an automorphism of $V$ (see \cite[Satz 4.1, Satz 5.2]{Hel69}). We have the decomposition
\begin{align*}
 V=V^+\oplus V^-\index{V1@$V^+$}\index{V2@$V^-$}
\end{align*}
into $\pm1$ eigenspaces of $\vartheta$. The eigenspace $V^+$ is a euclidean Jordan subalgebra of $V$ with the same identity element ${\bf e}$. Note that if $V$ itself is euclidean, then the identity $\vartheta=\id_V$ is the only possible Cartan involution of $V$, so that $V^+=V$ and $V^-=0$. We denote by $n_0$\index{nnull@$n_0$} the dimension and by $r_0$\index{rnull@$r_0$} the rank of $V^+$ and call $r_0$ the \textit{split rank} of $V$. The constants $n_0$ and $r_0$ only depend on the isomophism class of the Jordan algebra $V$, not on the choice of $\vartheta$.

The following elementary examples will eventually lead to the metaplectic representation and the minimal representation of $\upO(p+1,q+1)$.

\begin{example}\label{ex:JordanAlgebras}
\begin{enumerate}
 \item[\textup{(1)}] Let $V=\Sym(k,\RR)$\index{SymkR@$\Sym(k,\RR)$} be the space of symmetric $k\times k$ matrices with real entries. Endowed with the multiplication
  \begin{equation*}
   x\cdot y := \textstyle\frac{1}{2}(xy+yx)
  \end{equation*}
  $V$ becomes a simple euclidean Jordan algebra of dimension $n=\frac{k(k-1)}{2}$ and rank $r=k$ whose unit element is the unit matrix $\1$. Trace and determinant are the usual ones for matrices:
  \begin{align*}
   \tr(x) &= \Tr(x), & \det(x) &= \Det(x).
  \end{align*}
  Hence, the trace form is given by $\tau(x,y)=\Tr(xy)$. The inverse $x^{-1}$ of $x\in V$ exists if and only if $\Det(x)\neq0$ and in this case $x^{-1}$ is the usual inverse of the matrix $x$.
 \item[\textup{(2)}] Let $V=\RR\times W$ where $W$ is a real vector space of dimension $n-1$ with a symmetric bilinear form $\beta:W\times W\rightarrow\RR$. Then $V$ turns into a Jordan algebra with multiplication given by
  \begin{equation*}
   (\lambda,u)\cdot(\mu,v) := (\lambda\mu+\beta(u,v),\lambda v+\mu u).
  \end{equation*}
  $V$ is of dimension $n$ and rank $2$ and its unit element is ${\bf e}=(1,0)$. Trace and determinant are given by
  \begin{align*}
   \tr(\lambda,u) &= 2\lambda, & \det(\lambda,u) &= \lambda^2-\beta(u,u),
  \end{align*}
  and an element $(\lambda,u)\in V$ is invertible if and only if $\det(\lambda,u)=\lambda^2-\beta(u,u)\neq0$. In this case the inverse is given by $(\lambda,u)^{-1}=\frac{1}{\det(\lambda,u)}(\lambda,-u)$. The trace form can be written as
  \begin{equation*}
   \tau((\lambda,u),(\mu,v)) = 2(\lambda\mu+\beta(u,v)).
  \end{equation*}
  Hence, $V$ is semisimple if and only if $\beta$ is non-degenerate and $V$ is euclidean if and only if $\beta$ is positive definite. For $W=\RR^{p+q-1}$ with bilinear form $\beta$ given by the matrix
  \begin{align*}
   \left(\begin{array}{cc}-\1_{p-1}&\\&\1_q\end{array}\right)
  \end{align*}
  we put $\RR^{p,q}:=\RR\times W$\index{Rpq@$\RR^{p,q}$}, $p\geq1$, $q\geq0$. Then
  \begin{align*}
   \tau(x,y) &= 2(x_1y_1-x_2y_2-\cdots-x_py_p+x_{p+1}y_{p+1}+\cdots+x_{p+q}y_{p+q}),\\
   \det(x) &= x_1^2+\cdots+x_p^2-x_{p+1}^2-\cdots-x_{p+q}^2.
  \end{align*}
  Thus, $\RR^{p,q}$ is euclidean if and only if $p=1$. In any case, a Cartan involution of $\RR^{p,q}$ is given by
  \begin{align}
   \vartheta &= \left(\begin{array}{ccc}1&&\\&-\1_{p-1}&\\&&\1_q\end{array}\right).\label{eq:VpqCartanInv}
  \end{align}
  With this choice the euclidean subalgebra $(\RR^{p,q})^+$ is
  \begin{align*}
   (\RR^{p,q})^+ &= \RR e_1\oplus\RR e_{p+1}\oplus\cdots\oplus\RR e_n \cong \RR^{1,q},
  \end{align*}
  where $(e_j)_{j=1,\ldots,n}$\index{ej@$e_j$} denotes the standard basis of $\RR^{p,q}=\RR^n$, $n=p+q$.
\end{enumerate}
\end{example}

\subsubsection{Peirce decomposition}\label{sec:PeirceDecomp}

The Peirce decomposition of $V$ is a Jordan analog of the Lie theoretic root decomposition. It describes the structure of a Jordan algebra in terms of its idempotents.

In this subsection $V$ always denotes a real simple Jordan algebra, $\vartheta$ a Cartan involution of $V$ and we further assume that $V^+$ is also simple.

An element $c\in V$ is called \textit{idempotent} if $c^2=c$. A non-zero idempotent is called \textit{primitive} if it cannot be written as the sum of two non-zero idempotents and two idempotents $c_1$ and $c_2$ are called \textit{orthogonal} if $c_1c_2=0$. A collection $c_1,\ldots,c_m$ of orthogonal primitive idempotents in $V^+$ with $c_1+\cdots+c_m={\bf e}$ is called a \textit{Jordan frame}. By \cite[Theorem III.1.2]{FK94} the number $m$ of idempotents in a Jordan frame is always equal to the rank $r_0$ of $V^+$. For every two Jordan frames $c_1,\ldots,c_{r_0}$ and $d_1,\ldots,d_{r_0}$ there exists an automorphism $g$ of $V$ such that $gc_i=d_i$, $1\leq i\leq r_0$ (see \cite[Satz 8.3]{Hel69}).

For a fixed Jordan frame $c_1,\ldots,c_{r_0}$\index{cj@$c_j$} in $V^+$ the operators $L(c_1),\ldots,L(c_{r_0})$ commute and hence are simultaneously diagonalizable. The spectrum of each $L(c_i)$ is contained in $\{0,\frac{1}{2},1\}$ and $\sum_{i=1}^{r_0}{L(c_i)}=L({\bf e})=\id_V$. This yields the \textit{Peirce decomposition}
\begin{align}
 V &= \bigoplus_{1\leq i\leq j\leq r_0}{V_{ij}},\label{eq:PeirceDecomp}
\end{align}
where
\begin{align}
 V_{ij} &= \textstyle\{x\in V:L(c_k)x=\frac{\delta_{ik}+\delta_{jk}}{2}x\ \forall\,1\leq k\leq r_0\} && \mbox{for }1\leq i,j\leq r_0.\notag\index{Vij@$V_{ij}$}
\end{align}
Since the endomorphisms $L(c_i)$, $1\leq i\leq r_0$, are all symmetric with respect to the inner product $(-|-)$, the direct sum in \eqref{eq:PeirceDecomp} is orthogonal. Further, the group of automorphisms contains all possible permutations of the idempotents $c_1,\ldots,c_{r_0}$, and hence the subalgebras $V_{ii}$ have a common dimension $e+1$\index{e@$e$} and the subspaces $V_{ij}$ ($i<j$) have a common dimension $d$\index{d@$d$}, so that
\begin{align}
 \Tr(L(c_i)) &= \frac{n}{r_0} = e+1+(r_0-1)\frac{d}{2} & \forall\,1\leq i\leq r_0.\label{eq:noverr0}
\end{align}
We call a Jordan algebra $V$ \textit{split} (or \textit{reduced}) if $V_{ii}=\RR c_i$ for every $i=1,\ldots,r_0$, or equivalently if $e=0$. From \cite[\S 8, Korollar 2]{Hel69} it follows that if $V$ is split, then $r=r_0$, and if $V$ is non-split, then $r=2r_0$. Euclidean Jordan algebras are always split and hence $V^+_{ii}:=V_{ii}\cap V^+=\RR c_i$\index{Vijplus@$V_{ij}^+$}. With $V_{ii}^-:=V_{ii}\cap V^-$ we then have $V_{ii}=V_{ii}^+\oplus V_{ii}^-$ and $e=\dim\,V_{ii}^-$. If we denote by $d_0$\index{dnull@$d_0$} the dimension of $V^+_{ij}:=V_{ij}\cap V^+$ ($i<j$), then equation \eqref{eq:noverr0} for the euclidean subalgebra $V^+$ reads
\begin{align*}
 \frac{n_0}{r_0} &= 1+(r_0-1)\frac{d_0}{2}.
\end{align*}
Table \ref{tb:Constants} lists all simple real Jordan algebras with simple $V^+$ and their corresponding structure constants. A closer look at the table allows the following observation: If $V$ is non-euclidean, then $d=2d_0$ except in the case where $V=\RR^{p,q}$ with $p\neq q$.

\begin{proposition}[{\cite[\S 6]{Hel69}}]\label{prop:ClassificationEuclSph}
Let $V$ be a simple real Jordan algebra, $\vartheta$ a Cartan involution and assume that $V^+$ is also simple. If the split rank $r_0>1$, then exactly one of the following three statements holds:
\begin{enumerate}
\item[\textup{(1)}] $V$ is euclidean and in particular $d=d_0$,
\item[\textup{(2)}] $V$ is non-euclidean of rank $r\geq3$ and $d=2d_0$,
\item[\textup{(3)}] $V\cong\RR^{p,q}$, $p,q\geq2$.
\end{enumerate}
For $r_0=1$ the only possible case is
\begin{enumerate}
\item[\textup{(4)}] $V\cong\RR^{k,0}$, $k\geq1$.
\end{enumerate}
\end{proposition}

\begin{example}
\begin{enumerate}
\item[\textup{(1)}] For $V=\Sym(k,\RR)$ the matrices $c_i:=E_{ii}$, $1\leq i\leq k$, form a Jordan frame. The corresponding Peirce spaces are
\begin{align*}
 V_{ii} &= \RR c_i && \mbox{for }1\leq i\leq k,\\
 V_{ij} &= \RR(E_{ij}+E_{ji}) && \mbox{for }1\leq i<j\leq k.
\end{align*}
Hence, $d=d_0=1$ and $e=0$.
\item[\textup{(2)}] For $V=\RR^{p,q}$, $p,q\geq1$, a Jordan frame is given by $c_1=\frac{1}{2}(e_1+e_n)$, $c_2=\frac{1}{2}(e_1-e_n)$, $n=\dim(V)=p+q$. The corresponding Peirce spaces are
\begin{align*}
 V_{11} &= \RR c_1, & V_{12} &= \RR e_2\oplus\cdots\oplus\RR e_{n-1}, & V_{22} &= \RR c_2.
\end{align*}
Therefore $V$ is split, i.e. $e=0$, and $d=p+q-2$, $d_0=q-1$.
\end{enumerate}
\end{example}

\subsubsection{The constant $\nu$}\label{sec:MuNu}

For every real simple Jordan algebra $V$ with simple $V^+$ we introduce another constant $\nu$ by
\begin{equation}\label{eqn: def nu}
 \nu = \nu(V) := \frac{d}{2}-\left|d_0-\frac{d}{2}\right|-e-1 = \min(d,2d_0)-d_0-e-1\in\ZZ.\index{nu@$\nu$}\index{nuV@$\nu(V)$}
\end{equation}
Using Proposition \ref{prop:ClassificationEuclSph} we can calculate $\nu$ explicitly:
\begin{equation}\label{eqn:1.4}
 \nu = \begin{cases}-1 & \mbox{if $V$ is euclidean,}\\\frac{d}{2}-e-1 & \mbox{if $V$ is non-euclidean of rank $r\geq3$,}\\\min(p,q)-2 & \mbox{if $V\cong\RR^{p,q}$, $p,q\geq2$,}\\-k & \mbox{if $V\cong\RR^{k,0}$, $k\geq1$.}\end{cases}
\end{equation}

The constant $\nu$ for every simple real Jordan algebra $V$ with simple $V^+$ can also be found in Table \ref{tb:Constants}. For $V$ non-euclidean of rank $r\geq3$ the definition in  \eqref{eqn: def nu} agrees with the definition in \cite{DS99}. (There $d_0=\frac{d}{2}$ is denoted by $d$.)

\subsubsection{Definition of the Bessel operators}\label{sec:DefBesselOp}

We denote by $\frac{\partial}{\partial x}:C^\infty(V)\longrightarrow C^\infty(V)\otimes V$\index{ddx@$\frac{\partial}{\partial x}$} the gradient with respect to the non-degenerate trace form $\tau$ on $V$. For any complex parameter $\lambda\in\CC$ we define a second order differential operator
$$\calB_\lambda:C^\infty(V)\longrightarrow C^\infty(V)\otimes V$$
called the \textit{Bessel operator}, mapping complex-valued functions to vector-valued functions, by
\begin{align}
 \calB_\lambda := P\left(\frac{\partial}{\partial x}\right)x+\lambda\frac{\partial}{\partial x}.\label{eq:BesselOp}\index{Blambda@$\calB_\lambda$}
\end{align}
This formal definition has the following meaning: Let $(e_\alpha)_\alpha$ be a basis of $V$ with dual basis $(\overline{e}_\alpha)_\alpha$ with respect to the trace form $\tau$. Further denote by $x_\alpha$ the coordinates of $x\in V$ with respect to the basis $(e_\alpha)_\alpha$. Then
\begin{align*}
 \calB_\lambda f(x) &= \sum_{\alpha,\beta}{\frac{\partial^2f}{\partial x_\alpha\partial x_\beta}P(\overline{e}_\alpha,\overline{e}_\beta)x}+\lambda\sum_\alpha{\frac{\partial f}{\partial x_\alpha}\overline{e}_\alpha}, & x\in V.
\end{align*}
These operators were introduced by H. Dib \cite{Dib90} (see also \cite[Section XV.2]{FK94} for a more systematic presentation) in the case of a euclidean Jordan algebra, and by G. Mano \cite{Man08} for $V\simeq\RR^{p,q}$. The above definition is a natural generalization to arbitrary Jordan algebras.\\

We collect two basic properties of the Bessel operators in the following proposition (see \cite[Lemma 1.7.1 and Proposition 2.1.2]{Moe10}):

\begin{proposition}\label{prop:BesselOpProperties}
The Bessel operators $\calB_\lambda$ have the following properties:
\begin{enumerate}
\item[\textup{(1)}] For fixed $\lambda\in\CC$ the family of operators $(v|\calB_\lambda)$, $v\in V$, commutes.
\item[\textup{(2)}] We have the following product rule:
\begin{align*}
 \calB_\lambda\left[f(x)g(x)\right] &= \calB_\lambda f(x)\cdot g(x) + 2P\left(\frac{\partial f}{\partial x}(x),\frac{\partial g}{\partial x}(x)\right)x+f(x)\cdot\calB_\lambda g(x).
\end{align*}
\end{enumerate}
\end{proposition}

\subsection{Orbits of the structure group and equivariant measures}

In this subsection we describe the Hilbert space on which we later realize the minimal representation. More precisely, we introduce the structure group of a Jordan algebra and find equivariant measures on its orbits. This gives natural Hilbert spaces $L^2(\calO,\td\mu)$. We further show that for certain pa\-ra\-me\-ters $\lambda$ the Bessel operators $\calB_\lambda$ are tangential to these orbits and define differential operators on the orbits which are symmetric with respect to the $L^2$-inner product.

\subsubsection{The structure group}\label{sec: structure group}

The \textit{structure group} $\Str(V)$\index{StrV@$\Str(V)$} of a real or complex semisimple Jordan algebra $V$ is the group of invertible linear transformations $g\in\GL(V)$ such that there exists a constant $\chi(g)\in\KK^\times$ with
\begin{align}
 \det(gx) &= \chi(g)\det(x) & \forall\,x\in V,\label{eq:DetEquiv}\index{chig@$\chi(g)$}
\end{align}
where $\KK=\RR$ or $\CC$, depending on whether $V$ is a real or a complex Jordan algebra. By \cite[Lemma VIII.2.3]{FK94} an invertible linear transformation $g\in\GL(V)$ is in $\Str(V)$ if and only if there exists an $h\in \GL(V)$ with $P(gx)=gP(x)h$ for all $x\in V$. The group $\Str(V)$ is linear reductive over $\KK$. The map $\chi:\Str(V)\rightarrow\KK^\times$ defines a character of $\Str(V)$ which on the identity component $L:=\Str(V)_0$\index{L@$L$} is given by
\begin{align}
 \chi(g) &= (\Det\,g)^{\frac{r}{n}} & \forall\, g\in L.\label{eq:ChiDet}
\end{align}
Denote by $\frakl=\str(V)$\index{l@$\frakl$}\index{strV@$\str(V)$} the Lie algebra of $\Str(V)$ and $L$.

Let $V$ be a complex simple Jordan algebra. For the moment we write $V^\RR$ for $V$, if it is considered as a \emph{real} Jordan algebra. Since $\GL(V)\subseteq\GL(V^\RR)$ the characterization above shows that $\Str(V)\subseteq\Str(V^\RR)$. Let  $J:V\to V$ be the  complex structure on $V$.
The complexification
$(V^\RR)_\CC$ is isomorphic to the direct sum $V\oplus V$ as a complex Jordan algebra via
\begin{align*}
&(V^\RR)_{\CC} := V^\RR \otimes_{\RR} \CC
\xrightarrow[\varphi_L\oplus\varphi_R]{\sim} V_L \oplus V_R,
\\
\intertext{where $\varphi_L:V\overset{\sim}{\to}V_L$ ($\CC$-linear) and
$\varphi_R: V\overset{\sim}{\to}V_R$ (antilinear) are given by}
&\varphi_L(x) = \frac{1}{2}(x-iJx),
 \
 \varphi_R(x) = \frac{1}{2}(x+iJx).
\end{align*}
This implies
$$
\str\big((V^\RR)_\CC\big)\cong \str(V)\oplus \str(V).
$$
Since $\big(\str(V^\RR)\big)_\CC=\str\big((V^\RR)_\CC\big)$, this proves
$$\dim_\RR\big(\str(V^\RR)\big)= 2\dim_\CC\big(\str(V)\big) =\dim_\RR \big(\str(V)\big).$$
Thus $\str(V)\subseteq\str(V^\RR)$ implies $\str(V)=\str(V^\RR)$, i.e., for a complex simple Jordan algebra viewed as a real simple Jordan algebra, the real and complex structure algebras are the same. Note that changing the viewpoint on such a Jordan algebra means changing the Jordan determinant and trace, i.e. the meaning of \eqref{eq:DetEquiv}.

Assume now that $V$ is real. We write $g^*$\index{gstar@$g^*$} for the adjoint of $g\in\Str(V)$ with respect to the inner product $(-|-)$. Then the map $\theta:\Str(V)\rightarrow\Str(V),\,g\mapsto(g^*)^{-1}=(g^{-1})^*$ defines a Cartan involution of $\Str(V)$ which restricts to a Cartan involution of $L$. Its fixed point group $K_L:=L^\theta$\index{KL@$K_L$} is a maximal compact subgroup of $L$. Note that $K_L$ is connected, since $L$ is. The Lie algebra of $K_L$ will be denoted by $\frakk_\frakl$\index{kl@$\frakk_\frakl$}.

\begin{example}\label{ex:StrGrp}
\begin{enumerate}
\item[\textup{(1)}] The identity component $L$ of the structure group of $V=\Sym(k,\RR)$ is isomorphic to $(\GL(k,\RR)/\{\pm\1\})_0$, the action being induced by
\begin{align*}
 g\cdot a &= ga\,{}^t\!g & \mbox{for }g\in\GL(k,\RR),a\in V.
\end{align*}
Therefore, its Lie algebra is $\frakl=\gl(k,\RR)=\sl(k,\RR)\oplus\RR$, acting by
\begin{align*}
 X\cdot a &= Xa+a\,{}^t\!X & \mbox{for }X\in\gl(k,\RR),a\in V.
\end{align*}
The maximal compact subgroup is given
 by $K_L=(\operatorname{O}(k)/\{\pm 1\})_0$ which acts by conjugation.
\item[\textup{(2)}] For $V=\RR^{p,q}$ we have $L=\RR_+\SO(p,q)_0$ with maximal compact subgroup $K_L=\SO(p)\times\SO(q)$.
\end{enumerate}
\end{example}

\subsubsection{Orbits of the structure group}\label{sec:OrbitsMeasures}

There are only finitely many orbits under the action of $L$ on $V$. An explicit description of these orbits can be found in Kaneyuki \cite{Kan98}. We are merely interested in the open orbit of $L$ containing the unit element $\bf e$ of $V$ and the orbits which are contained in its boundary.

Let $\Omega=L\cdot{\bf e}$\index{Omega@$\Omega$} be the open orbit of $L$ containing the identity element of the Jordan algebra. $\Omega$ is an open cone in $V$ and at the same time a reductive symmetric space. It has a polar decomposition in terms of the compact group $K_L$ and the Jordan frame:
\begin{align*}
 \Omega &= \left\{u\sum_{j=1}^{r_0}{t_jc_j}:u\in K_L,t_1\geq\ldots\geq t_{r_0}>0\right\}.
\end{align*}

The boundary $\partial\Omega$ is the union of orbits of lower rank. The closure $\overline{\Omega}$ of $\Omega$ admits the following stratification:
\begin{align*}
 \overline{\Omega} &= \calO_0\cup\ldots\cup\calO_{r_0},
\end{align*}
where $\calO_k=L\cdot s_k$\index{Ok@$\calO_k$} with
\begin{align*}
 s_k &:= c_1+\cdots+c_k, & 0\leq k\leq r_0.\index{sk@$s_k$}
\end{align*}
Every orbit is a homogeneous space, but in general these homogeneous spaces are not symmetric. A polar decomposition for the orbit $\calO_k$ is given by
\begin{align}
 \calO_k &= \left\{u\sum_{j=1}^k{t_jc_j}:u\in K_L,t_1\geq\ldots\geq t_k>0\right\}.\label{eq:PolarOk}
\end{align}

We will mostly be interested in the minimal non-zero orbit $\calO_1$. For later use we calculate its dimension:

\begin{lemma}\label{lem:DimO1}
$\dim\,\calO_1=e+1+(r_0-1)d$.
\end{lemma}

\begin{proof}
As a homogeneous space we have $\calO_1=L/S$, where $S=\Stab_L(c_1)$. Denote by $\fraks$ the Lie algebra of $S$. Using the results of \cite[Section 1.5.2]{Moe10}, we obtain that
\begin{align*}
 \frakl &= \fraks\oplus L(V_{11})\oplus\bigoplus_{j=2}^{r_0}{c_j\Box V_{1j}}.
\end{align*}
Hence, $\dim\,\calO_1=\dim\,\frakl-\dim\,\fraks=(e+1)+(r_0-1)d$.
\end{proof}

\begin{example}\label{ex:MinimalOrbit}
\begin{enumerate}
\item[\textup{(1)}] For $V=\Sym(k,\RR)$ the cone $\Omega$ is the convex cone of symmetric positive definite matrices. Its boundary contains the orbit $\calO_1$ of minimal rank which is given by
\begin{align*}
 \calO_1 &= \{x\,{}^t\!x:x\in\RR^k\setminus\{0\}\}.
\end{align*}
The map
\begin{align}
 \RR^k\setminus\{0\}\rightarrow\calO_1,\,x\mapsto x\,{}^t\!x,\label{eq:SymTwoFoldCovering}
\end{align}
is a surjective two-fold covering.
\item[\textup{(2)}] For $V=\RR^{p,q}$ let $n=\dim\,V=p+q$. We have to distinguish between two cases. If $p=1$, $q\geq2$, then $\Omega$ is the convex cone given by
\begin{align*}
 \Omega &= \{x\in\RR^{1,q}:x_1>0,x_1^2-x_2^2-\cdots-x_n^2>0\}.
\end{align*}
Its boundary is the union of the trivial orbit $\calO_0=\{0\}$ and the forward light cone
\begin{align*}
 \calO_1 &= \{x\in\RR^{1,q}:x_1>0,x_1^2-x_2^2-\cdots-x_n^2=0\}.
\end{align*}
For $p,q\geq2$ we have
\begin{align*}
 \Omega &= \{x\in\RR^{p,q}:x_1^2+\cdots+x_p^2-x_{p+1}^2-\cdots-x_n^2>0\},
\end{align*}
which is not convex. In this case the minimal non-trivial orbit is given by
\begin{align*}
 \calO_1 &= \{x\in\RR^{p,q}:x_1^2+\cdots+x_p^2-x_{p+1}^2-\cdots-x_n^2=0\}\setminus\{0\}.
\end{align*}
In both cases, $\calO_1$ can be parameterized by bipolar coordinates:
\begin{align}
 \RR_+\times S^{p-1}_0\times S^{q-1}\stackrel{\sim}{\rightarrow}\calO_1,\,(t,\omega,\eta)\mapsto(t\omega,t\eta),\label{eq:VpqPolarCoordinates}
\end{align}
where $S^{n-1}$\index{Sn@$S^{n-1}$} denotes the unit sphere in $\RR^n$. For $n=1$ the sphere is disconnected, and $S_0^{n-1}=\{1\}$. 
\end{enumerate}
\end{example}

\subsubsection{Equivariant measures}

We define a generalization of the \textit{Wallach set} (sometimes referred to as the \textit{Berezin--Wallach set}) by
\begin{align}
 \calW &:= \left\{0,\frac{r_0d}{2r},\ldots,(r_0-1)\frac{r_0d}{2r}\right\}\cup\left((r_0-1)\frac{r_0d}{2r},\infty\right).\index{W@$\calW$}\label{eq.BWset}
\end{align}
For $r_0=1$ this reduces to ${\cal W}= \left(0,\infty\right)$. 

For convenience we denote for $\lambda>(r_0-1)\frac{r_0d}{2r}$ the open orbit $\calO_{r_0}=\Omega$ by $\calO_\lambda$\index{Olambda@$\calO_\lambda$}. Similarly, for $\lambda=k\frac{r_0d}{2r}$, $k=0,\ldots,r_0-1$, we put $\calO_\lambda:=\calO_k$\index{Olambda@$\calO_\lambda$}. 
Note that if $r_0>1$ then $\calO_\lambda=\calO_1$ implies that $\lambda=\lambda_1:=\frac{r_0d}{2r}$ is the minimal non-zero discrete Wallach point. If $r_0=1$ then $\calO_\lambda=\calO_1$ is equivalent to $\lambda>0$.

The proof of the following result concerning equivariant measures on the orbits $\calO_\lambda $  for $\lambda$ in the generalized Wallach set is standard:

\begin{proposition}\label{thm:EquivMeasures}
Fix $\lambda\in\calW$ and let $k\in\{0,\ldots,r_0\}$ such that $\calO_\lambda=\calO_k$. For $k=0$ we have $\lambda=0$ and the Dirac measure $\td\mu_0:=\delta_0$ at $x=0$ defines an $L$-equivariant measure on $\calO_0=\{0\}$. For $k>0$ the formula
 \begin{align*}
  \int_{\calO_\lambda}{f(x)\td\mu_\lambda(x)}
  &:= \int_{K_L}{\int_{s_1>\ldots>s_k}{f\Big(u\sum_{j=1}^k{e^{s_j}c_j}\Big)J_\lambda({\bf s})\td{\bf s}}\td u},
 \end{align*}
 where
 \begin{align*}
  J_\lambda({\bf s}) &= e^{\frac{\lambda r}{k}\sum_{i=1}^k{s_i}}\prod_{1\leq i<j\leq k}{\sinh^{d_0}\left(\frac{s_i-s_j}{2}\right)\cosh^{d-d_0}\left(\frac{s_i-s_j}{2}\right)},
 \end{align*}
defines an $L$-equivariant measure $\td\mu_\lambda$ on $\calO_\lambda$. These measures transform according to
\begin{align}
 \td\mu_\lambda(gx) &= \chi(g)^\lambda\td\mu_\lambda(x) & \mbox{for }g\in L.\label{eq:dmulambdaEquivariance}
\end{align}
\begin{enumerate}
\item[\textup{(1)}] On $\calO_{r_0}=\Omega$ the $L$-equivariant measures which are locally finite near $0$ are (up to positive scalars) exactly the measures $\td\mu_\lambda$, $\lambda>(r_0-1)\frac{r_0d}{2r}$.
 Moreover, $\td\mu_\lambda$ is absolutely continuous with respect to the Lebesgue measure $\td x$\index{dx@$\td x$} on $\Omega$ and we have
 \begin{align*}
  \td\mu_\lambda(x) &= \const\cdot\det(x)^{\lambda-\frac{n}{r}}\td x & \mbox{for $\lambda>(r_0-1)\frac{r_0d}{2r}$.}
 \end{align*}
\item[\textup{(2)}] For $k=0,\ldots,r_0-1$, up to positive scalars, $\td\mu_k:=\td\mu_\lambda$\index{dmulambdax@$\td\mu_\lambda(x)$} is the unique $L$-equivariant measure on $\calO_k$.
\end{enumerate}
\end{proposition}

For the minimal non-trivial orbit $\calO_1$ the polar decomposition \eqref{eq:PolarOk} simplifies to $\calO_1=K_L\RR_+c_1$. Further, if $\calO_\lambda=\calO_1$, then the integral formula in Proposition \ref{thm:EquivMeasures} amounts to
\begin{align}
 \int_{\calO_1}{f(x)\td\mu_\lambda(x)} &= \int_{K_L}{\int_0^\infty{f(ktc_1)t^{\lambda r-1}\td t}\td k}.\label{eq:dmuIntFormula}
\end{align}

\begin{example}\label{ex:EquivMeasures}
\begin{enumerate}
\item[\textup{(1)}] For $V=\Sym(k,\RR)$ the two-fold covering \eqref{eq:SymTwoFoldCovering} induces a unitary (up to a scalar) isomorphism
\begin{align}
 \calU:L^2(\calO_1,\td\mu_1)\rightarrow L^2_{\textup{even}}(\RR^k),\ \calU\psi(x) := \psi(x^t\!x),\label{eq:DefCalU}\index{U@$\calU$}
\end{align}
where $L^2_{\textup{even}}(\RR^k)$\index{L2evenRk@$L^2_{\textup{even}}(\RR^k)$} denotes the space of even $L^2$-functions on $\RR^k$.
\item[\textup{(2)}] For $V=\RR^{p,q}$ the measure $\td\mu_1$ can be expressed in bipolar coordinates \eqref{eq:VpqPolarCoordinates}. Using \eqref{eq:dmuIntFormula} we obtain
\begin{align*}
 \td\mu_1 &= \const\cdot t^{p+q-3}\td t\td\omega\td\eta,
\end{align*}
where $\td\omega$ and $\td\eta$ denote the normalized euclidean measures on $S^{p-1}_0$ and $S^{q-1}$, respectively.
\end{enumerate}
\end{example}

\subsubsection{Tangential differential operators}

The Bessel operator $\calB_\lambda$ is defined on the ambient space $V$. We show that for $\lambda\in\calW$ it is tangential to the orbit $\calO_\lambda$ and induces a symmetric operator on $L^2(\calO_\lambda,\td\mu_\lambda)$. We have given a direct proof in \cite{KM07b} for this fact in the case $V=\RR^{p,q}$. In this subsection, we take another approach, namely, we introduce certain zeta functions and use the fact that the measures $\td\mu_k$, $0\leq k\leq r_0-1$, arise as their residues.

Denote by $\calS(V)$ the space of rapidly decreasing smooth functions on $V$ and by $\calS'(V)$ its dual, the space of tempered distributions on $V$. For $\lambda>(r_0-1)\frac{r_0d}{2r}$ we define the \textit{zeta function} $Z(-,\lambda)\in\calS'(V)$ by
\begin{align*}
 Z(f,\lambda) &:= \begin{cases}\displaystyle\int_\Omega{f(x)\det(x)^{\lambda-\frac{n}{r}}\td x} & \mbox{for $V$ euclidean or $V\cong\RR^{p,q}$, $p,q\geq2$,}\vspace{.3cm}\\\displaystyle\int_V{f(x)|\det(x)|^{\lambda-\frac{n}{r}}\td x} & \mbox{for $V$ non-euclidean and $V\ncong\RR^{p,q}$, $p,q\geq2$.}\end{cases}
\end{align*}
Then for every $f\in\calS(V)$ the function $\lambda\mapsto Z(f,\lambda)$ extends to a meromorphic function on the complex plane (see \cite[Chapter VII, Section 2]{FK94} for the euclidean case, \cite[Theorem~6.2\,(2)]{BSZ06} for the non-euclidean case $\ncong\RR^{p,q}$, and \cite[Chapter III.2]{GS64} for $V=\RR^{p,q}$).

\begin{proposition}\label{prop:ZetaFunctions}
\begin{enumerate}
\item[\textup{(1)}] Let $V\ncong\RR^{p,q}$, $p,q\geq2$. Then the measure $\td\mu_k$, $0\leq k\leq r_0-1$, is a constant multiple of the residue of the zeta function $Z(-,\lambda)$ at the value $\lambda=k\frac{r_0d}{2r}$.
\item[\textup{(2)}] Let $V=\RR^{p,q}$, $p,q\geq2$, then $r_0=2$. In this case the measure $\td\mu_0$ is just a scalar multiple of the Dirac delta distribution at $0$ and the measure $\td\mu_1$ is again a constant multiple of the residue of the zeta function $Z(-,\lambda)$ at the value $\lambda=\frac{r_0d}{2r}=\frac{p+q-2}{2}$.
\end{enumerate}
\end{proposition}

\begin{proof}
Part (1) is \cite[Proposition VII.2.3]{FK94} and \cite[Theorem 6.2]{BSZ06} and part (2) can be found in \cite[Section III.2.2]{GS64}.
\end{proof}

Similarly to the proof of \cite[Proposition XV.2.4]{FK94} one can now show the following symmetry property for the Bessel operators with respect to the zeta functions $Z(-,\lambda)$:

\begin{proposition}\label{prop:BnuSA}
For $f,g\in\calS(V)$ and $\lambda\in\CC$ we have
\begin{align*}
 Z((\calB_\lambda f)\cdot g,\lambda) &= Z(f\cdot(\calB_\lambda g),\lambda),
\end{align*}
as identity of meromorphic functions in $\lambda$.
\end{proposition}

Using the previous proposition we now prove the main result of this section. For this recall that a differential operator $D$ on $V$ is said to be tangential to a submanifold $M\subseteq V$ if for every $\varphi\in C^\infty(V)$ the property $\varphi|_M=0$ implies that $(D\varphi)|_M=0$. In this case it is easy to see that $D$ defines a differential operator acting on $C^\infty(M)$.

\begin{theorem}\label{thm:BlambdaTangential}
For every $\lambda\in\calW$ the differential operator $\calB_\lambda$ is tangential to the orbit $\calO_\lambda$ and defines a symmetric operator on $L^2(\calO_\lambda,\td\mu_\lambda)$.
\end{theorem}

\begin{proof}
If $\calO_\lambda=\Omega$ is the open orbit, then every differential operator is tangential. Symmetry follows immediately from Proposition \ref{prop:BnuSA}.\\
Now assume that $\calO_\lambda=\calO_k$, $0\leq k\leq r_0-1$. Let $\varphi\in C^\infty(V)$ such that $\varphi|_{\calO_\lambda}=0$. For any $\psi\in C_c^\infty(V)$ we obtain with Proposition \ref{prop:BnuSA}:
\begin{align*}
 \int_{\calO_\lambda}{\calB_\lambda\varphi\cdot\psi\td\mu_\lambda} &= \const\cdot\res_{\mu=\lambda}{Z\left(\calB_\mu\varphi\cdot\psi,\mu\right)}\\
 &= \const\cdot\res_{\mu=\lambda}{Z\left(\varphi\cdot\calB_\mu\psi,\mu\right)}\\
 &= \int_{\calO_\lambda}{\varphi\cdot\calB_\lambda\psi\td\mu_\lambda} = 0.
\end{align*}
Hence $(\calB_\lambda\varphi)|_{\calO_\lambda}=0$ in $L^2(\calO_\lambda,\td\mu_\lambda)$ which implies $\calB_\lambda\varphi(x)=0$ for every $x\in\calO_\lambda$ and therefore $\calB_\lambda$ is tangential to $\calO_\lambda$. Symmetry now follows again from Proposition \ref{prop:BnuSA}.
\end{proof}

\subsubsection{Action of the Bessel operator for the minimal orbit}\label{sec:BesselMinO}

We compute the action of $\calB_{\lambda}$ on radial functions on the minimal orbit $\calO:=\calO_1$, i.e. functions depending only on $\|x\|:=\sqrt{(x|x)}$. For convenience we use the following normalization:
\begin{align*}
 |x| := \sqrt{\frac{r}{r_0}}\|x\|=\left\{\begin{array}{cl}\|x\| & \mbox{if $V$ is split,}\\\sqrt{2}\|x\| & \mbox{if $V$ is non-split.}\end{array}\right.\index{1absvalb@{"|}$-${"|}}
\end{align*}
(Note that $\frac{r}{r_0}=({\bf e}|c_1)$.) Recall the Cartan involution $\vartheta$. If $\psi(x)=f(|x|)$, $x\in V$, is a radial function, then
\begin{align}
 \frac{\partial\psi}{\partial x}(x) &= \frac{r}{r_0}\frac{f'(|x|)}{|x|}\vartheta(x).\label{eq:ddxradial}
\end{align}
A simple calculation gives the following formula for the action of $\calB_\lambda$ on radial functions:

\begin{proposition}\label{prop:BnuRadial}
If $\psi(x)=f(|x|)$, $x\in\calO$, is a radial function on $\calO$, $f\in C^\infty(\RR_+)$, then for $x=ktc_1\in\calO$ ($k\in K_L$, $t>0$) we have
\begin{multline*}
 \calB_\lambda\psi(x) = \left(f''(|x|)+\left(\frac{r}{r_0}\lambda+\frac{d}{2}-d_0-e\right)\frac{1}{|x|}f'(|x|)\right)\vartheta(x)\\
 +\frac{r_0}{r}\left(d_0-\frac{d}{2}\right)f'(|x|)\vartheta(k{\bf e}).
\end{multline*}
\end{proposition}

The formula in Proposition \ref{prop:BnuRadial} can be simplified according to the four cases of Proposition \ref{prop:ClassificationEuclSph}. For this we introduce the ordinary differential operator $B_\alpha$ on $\RR_+$ which is defined by
\begin{align}
 B_\alpha &:= \frac{\td^2}{\td t^2}+(2\alpha+1)\frac{1}{t}\frac{\td}{\td t}-1 = \frac{1}{t^2}\left(\left(t\frac{\td}{\td t}\right)^2+2\alpha\left(t\frac{\td}{\td t}\right)-t^2\right).\label{eq:OrdinaryBesselOp}
\end{align}
The normalized $K$-Bessel function $\widetilde{K}_\alpha(z):=\left(\frac{z}{2}\right)^{-\alpha}K_\alpha(z)$ is an $L^2$-solution of the differential equation $B_\alpha u=0$.

The following corollary to Proposition \ref{prop:BnuRadial} is the key to prove $\frakk$-finiteness of the underlying $(\frakg,\frakk)$-module of the minimal representation in Theorems \ref{prop:Kfinite} and \ref{prop:Kfinite-split rank 1}. Recall the constant $\nu$ introduced in Subsection \ref{sec:MuNu} and set $\lambda_1:=\frac{r_0d}{2r}$. If $r_0>1$, this is the minimal non-zero discrete Wallach point. If $r_0=1$, then there are no non-zero discrete generalized Wallach points and the equivariant measures $\td\mu_\lambda$ on $\calO_1$ are parameterized by $\lambda>0$. For convenience we put $\sigma:=\frac{r}{r_0}\lambda$.

\begin{corollary}\label{cor:BnuRadial}
Let $\lambda\in\cal W$ be such that $\calO_\lambda=\calO_1$ and $\psi(x)=f(|x|)$ a radial function on $\calO_1$.
\begin{enumerate}
\item[\textup{(1)}] If $V$ is euclidean, then
 \begin{align*}
  (\calB_{\lambda_1}-\vartheta(x))\psi(x) &= B_{\frac{\nu}{2}}f(|x|)\vartheta(x)+\frac{d}{2}f'(|x|){\bf e}.
 \end{align*}
\item[\textup{(2)}] If $V$ is non-euclidean of rank $r\geq3$, then
 \begin{align*}
  (\calB_{\lambda_1}-\vartheta(x))\psi(x) &= B_{\frac{\nu}{2}}f(|x|)\vartheta(x).
 \end{align*}
\item[\textup{(3)}] If $V=\RR^{p,q}$, $p,q\geq2$, and $\vartheta$ is as in \eqref{eq:VpqCartanInv}, then with $x=(x',x'')\in\RR^p\times\RR^q$
 \begin{align*}
  (\calB_{\lambda_1}-\vartheta(x))\psi(x) &= B_{\frac{q-2}{2}}f(|x|)\vartheta(x',0)+B_{\frac{p-2}{2}}f(|x|)\vartheta(0,x'').
 \end{align*}
\item[\textup{(4)}] If $V=\RR^{k,0}$, $k\geq1$, then for $\sigma=\frac{r}{r_0}\lambda>0$
 \begin{align*}
  (\calB_\lambda-\vartheta(x))\psi(x) &= B_{\frac{\nu+\sigma}{2}}f(|x|)\vartheta(x).
 \end{align*}
\end{enumerate}
\end{corollary} 

\section{Construction of minimal representations}\label{ch:MinRep}

To every simple real Jordan algebra $V$ we associate its conformal group $G$ and conformal Lie algebra $\frakg$. For $V$ of split rank $r_0>1$ we construct a representation of $\frakg$ on $C^\infty(\calO)$, where $\calO=\calO_1$ is the minimal non-zero orbit of $L$. We further determine the cases in which this representation integrates to a unitary irreducible representation of a finite cover of $G$ on the Hilbert space $L^2(\calO,\td\mu)$, $\td\mu:=\td\mu_1$ being the unique $L$-equivariant measure on $\calO$. If $V$ is split, then the representation is minimal for which we give a conceptual proof. For the special cases $V=\Sym(k,\RR)$ and $V=\RR^{p,q}$ we identify this representation with the Segal--Shale--Weil representation and the minimal representation of $\upO(p+1,q+1)$, respectively. For $V$ of split rank $r_0=1$ the same method yields complementary series representations of $\SO(p+1,1)_0$ on the Hilbert spaces $L^2(\calO,\td\mu_\lambda)$, where $\lambda$ belongs to an open interval.

Throughout this section $V$ will always denote a simple real Jordan algebra, $\vartheta$ a Cartan involution on $V$ and we further assume that $V^+$ is simple.

\subsection{The conformal group}\label{sec:ConformalGroup}

For a real or complex simple Jordan algebra $V$ one has the \emph{conformal group} $\Co(V)$ which acts on $V$ by rational transformations. Its Lie algebra $\frakg=\co(V)$, also known as the \emph{Kantor--Koecher--Tits algebra}, is given by quadratic vector fields on $V$.
As for the structure group, for a complex simple Jordan algebra viewed as a real simple Jordan algebra, a priori, one has two different constructions. But since the ground field enters only via the structure group, the result is the same at least on the level of Lie algebras. Therefore we view complex simple Jordan algebras as real simple Jordan algebras, unless stated otherwise, and can speak about Cartan involutions and related real concepts also in this case.

We will describe $\frakg$ in some detail because its structure will play an important role in our construction of representations.
Further, for a maximal compact subalgebra $\frakk$ of $\frakg$ we recall the characterization of the highest weights of $\frakk_\frakl$-spherical $\frakk$-representations via the Cartan--Helgason theorem. These representations will appear as $\frakk$-types in the minimal representation.

\subsubsection{The Kantor--Koecher--Tits construction}\label{sec:KKT}

The conformal group of $V$ is built up from three different rational transformations.
\begin{enumerate}
\item[\textup{(1)}] First, $V$ acts on itself by translations
\begin{align*}
 n_a(x) &:= x+a & \forall\,x\in V\index{na@$n_a$}
\end{align*}
with $a\in V$. Denote by $N:=\{n_a:a\in V\}$\index{N@$N$} the abelian group of translations which is isomorphic to $V$.
\item[\textup{(2)}] The structure group $\Str(V)$ of $V$ acts on $V$ by linear transformations.
\item[\textup{(3)}] Finally, we define the \textit{conformal inversion element} $j$ by
\begin{align*}
 j(x) &= -x^{-1} & \forall\,x\in V^\times:=\{y\in V:y\mbox{ invertible}\}.\index{j@$j$}
\end{align*}
$j$ is a rational transformation of $V$.
\end{enumerate}
The \textit{conformal group} $\Co(V)$\index{CoV@$\Co(V)$} is defined as the subgroup of the group of rational transformations of $V$ which is generated by $N$, $\Str(V)$ and $j$:
\begin{align*}
 \Co(V) := \langle N,\Str(V),j\rangle_{\textup{grp}}.
\end{align*}
$\Co(V)$ is a simple Lie group with trivial center (see \cite[Chapter VIII, Section 6]{Jac68}, \cite[Theorem VIII.1.3]{Ber00}). The semidirect product $\Str(V)\ltimes N$ is a maximal parabolic subgroup of $\Co(V)$ (see e.g. \cite[Section X.6.3]{Ber00}).

We let $G:=\Co(V)_0$\index{G@$G$} be the identity component of the conformal group which is also simple with trivial center. (The proof of \cite[Theorem VIII.1.3]{Ber00} applies for the identity component as well.) The group $G$ is generated by $N$, $L=\Str(V)_0$ and $j$, but the intersection $L^{\max}:=G\cap\Str(V)$\index{Lmax@$L^{\textup{max}}$} is in general bigger than $L$. Therefore, the semidirect product
\begin{align}
 Q &:= L\ltimes N\index{Q@$Q$}\label{eq:DefQ}
\end{align}
is in general not maximal parabolic in $G$, but an open subgroup of the maximal parabolic subgroup $Q^{\textup{max}}:=L^{\max}\ltimes N$\index{Qmax@$Q^{\textup{max}}$}.\\

Now let us examine the structure of the Lie algebra $\frakg:=\co(V)$\index{g@$\frakg$}\index{coV@$\co(V)$} of $G$. An element $X\in\frakg$ corresponds to a quadratic vector field on $V$ of the form
\begin{align*}
 X(z) &= u + Tz - P(z)v, & z\in V
\end{align*}
with $u,v\in V$ and $T\in\frakl=\str(V)$. We use the notation $X=(u,T,v)$\index{X@$X$}\index{uTv@$(u,T,v)$} for short. In view of this, we have the decomposition
\begin{equation}
 \frakg = \nf + \frakl + \nfo,\label{eq:GelfandNaimark}
\end{equation}
where\index{l@$\frakl$}
\begin{alignat*}{3}
 \frakn &= \{(u,0,0):u\in V\} &&\cong V,\index{n@$\frakn$}\\
 \frakl &= \{(0,T,0):T\in\str(V)\} &&\cong \str(V),\\
 \overline{\frakn} &= \{(0,0,v):v\in V\} &&\cong V.\index{nbar@$\overline{\frakn}$}
\end{alignat*}
In this decomposition the Lie algebra $\frakq^{\textup{max}}$\index{qmax@$\frakq^{\textup{max}}$} of $Q^{\textup{max}}$ (and $Q$) is given by
\begin{align*}
 \frakq^{\textup{max}} &= \frakn + \frakl.
\end{align*}
If $X_j=(u_j,T_j,v_j)$, $j=1,2$, then the Lie bracket is given by
\begin{equation}
 [X_1,X_2] = (T_1u_2-T_2u_1,[T_1,T_2]+2(u_1\Box v_2)-2(u_2\Box v_1),-T_1^\# v_2+T_2^\# v_1),\label{eq:LieBracket}
\end{equation}
where $T^\#$ denotes the adjoint of $T$ with respect to the trace form $\tau$ and $u\Box v$ the box operator as introduced in Subsection \ref{sec:JordanAlgebras}. From this formula it is easy to see that the decomposition \eqref{eq:GelfandNaimark} actually defines a grading on $\frakg$:
\begin{align*}
 \frakg=\frakg_{-1}+\frakg_0+\frakg_1,
\end{align*}
where $\frakg_{-1}=\frakn$, $\frakg_0=\frakl$ and $\frakg_1=\overline{\frakn}$.\index{g-1@$\frakg_{-1}$}\index{g0@$\frakg_0$}\index{g1@$\frakg_1$}

\begin{example}\label{ex:ConfGrp}
Since $G$ has trivial center we can calculate it by factoring out the center from the universal covering: $G=\widetilde{G}/Z(\widetilde{G})$. Here the universal covering $\widetilde{G}$ of $G$ is uniquely determined by the Lie algebra $\frakg$.
\begin{enumerate}
\item[\textup{(1)}] Let $V=\Sym(k,\RR)$. Then $\frakg\cong\sp(k,\RR)$ via the isomorphism
\begin{align*}
 \frakg &\rightarrow \sp(k,\RR),\,(u,T+s\1,v)\mapsto\left(\begin{array}{cc}T+\frac{s}{2}&u\\v&-^tT-\frac{s}{2}\end{array}\right),
\end{align*}
where $u,v\in V$, $T\in\sl(k,\RR)$ and $s\in\RR$. Hence, $G\cong\Sp(k,\RR)/\{\pm\1\}$, where $\Sp(k,\RR)/\{\pm\1\}$ acts on $x\in V$ by fractional linear transformations:
\begin{align*}
 \left(\begin{array}{cc}A&B\\C&D\end{array}\right)\cdot x &= (Ax+B)(Cx+D)^{-1}.
\end{align*}
\item[\textup{(2)}] Let $V=\RR^{p,q}$. Then an explicit isomorphism $\frakg\stackrel{\sim}{\rightarrow}\so(p+1,q+1)$ is given by
\begin{align*}
 (u,0,0) &\mapsto \left(\begin{array}{c|cc|c}&-^t\!u'&^t\!u''&\\\hline u'&&&u'\\u''&&&u''\\\hline&^t\!u'&-^t\!u''&\end{array}\right), && u\in V,\\
 (0,s\1+T,0) &\mapsto \left(\begin{array}{c|c|c}&&-s\\\hline&T&\\\hline-s&&\end{array}\right), && T\in\so(p,q), s\in\RR,\\
 (0,0,\vartheta(v)) &\mapsto \left(\begin{array}{c|cc|c}&^t\!v'&^t\!v''&\\\hline-v'&&&v'\\v''&&&-v''\\\hline&^t\!v'&^t\!v''&\end{array}\right), && v\in V.
\end{align*}
Hence, $G\cong\SO(p+1,q+1)_0/Z(\SO(p+1,q+1)_0)$. The center $Z(\SO(p+1,q+1)_0)$ is equal to $\{\pm\1\}$ if $p$ and $q$ are both even, and it is trivial otherwise.
\end{enumerate}
\end{example}

The Cartan involution $\theta$ of $\Str(V)$ extends to a Cartan involution of $\Co(V)$ by
\begin{align*}
 \theta:\Co(V)\rightarrow\Co(V),\,g\mapsto\vartheta\circ j\circ g\circ j\circ\vartheta.\index{theta@$\theta$}
\end{align*}
It restricts to a Cartan involution of $G$. The corresponding involution $\theta$ of the Lie algebra $\frakg$ is given by (see \cite[Proposition 1.1]{Pev02})
\begin{align}
 \theta(u,T,v) &:= (-\vartheta(v),-T^*,-\vartheta(u)), & (u,T,v)\in\frakg.\label{eq:ThetaOng}\index{theta@$\theta$}
\end{align}
In the above notation $\nfo=\theta(\nf)$. We remark that the twisted Killing form $B(X_1,\theta X_2)$ restricted to $X_1,X_2\in\frakn$ is given by
\begin{align*}
 B(X_1,\theta X_2) &= -\frac{4n}{r}\tau(u_1,\vartheta u_2) = -\frac{4n}{r}(u_1|u_2), & X_i=(u_i,0,0),\,i=1,2,
\end{align*}
which is the trace form of $V$ twisted by the Cartan involution $\vartheta$ of $V$ (see \cite[Section 1.6.1]{Moe10}). Let $\frakg=\frakk+\frakp$\index{k@$\frakk$}\index{p@$\frakp$} be the corresponding Cartan decomposition of $\frakg$. Then
\begin{align}
 \frakk &= \{(u,T,-\vartheta(u)):u\in V,T\in\lf,T+T^*=0\}.\label{eq:frakk}
\end{align}
The fixed point group $K:=G^\theta$\index{K@$K$} of $\theta$ is a maximal compact subgroup of $G$ with Lie algebra $\frakk$. Then clearly $K_L=K\cap L$. The subgroup $K_L\subseteq K$ is symmetric, the corresponding involution being $g\mapsto(-\1)\circ g\circ(-\1)$.

Table \ref{tb:Groups} lists the conformal algebra $\frakg$, the structure algebra $\frakl$ and their maximal compact subalgebras $\frakk$ and $\frakk_\frakl$ for all simple real Jordan algebras.

The following observation on the center of $\frakk$ will be needed later.

\begin{lemma}\label{lem:CenterK}
Assume that $V$ and $V^+$ are simple. Then the center $Z(\frakk)$\index{Zk@$Z(\frakk)$} of $\frakk$ is non-trivial only if $V$ is euclidean. In this case it is given by $Z(\frakk)=\RR({\bf e},0,-{\bf e})$.
\end{lemma}

\subsubsection{Root space decomposition}\label{subsec:ConfGrpRoots}

We already mentioned in Subsection \ref{sec:PeirceDecomp} that the Peirce decomposition of the Jordan algebra is related to root decompositions of the corresponding Lie algebra.

We choose the maximal toral subalgebra
\begin{align*}
 \frakt &:= \left\{\left(\sum_{i=1}^{r_0}{t_ic_i},0,-\sum_{i=1}^{r_0}{t_ic_i}\right):t_i\in\RR\right\}\subseteq\frakk_\frakl^\perp\subseteq\frakk\index{t@$\frakt$}
\end{align*}
in the orthogonal complement of $\frakk_\frakl$ in $\frakk$. The corresponding root system of $(\frakg_\CC,\frakt_\CC)$ is of type $C_{r_0}$ and given by
\begin{align*}
 \Sigma(\frakg_\CC,\frakt_\CC) &= \left\{\frac{\pm\gamma_i\pm\gamma_j}{2}\right\},
\end{align*}
where
\begin{align*}
 \gamma_j\left(\sum_{k=1}^{r_0}{t_kc_k},0,-\sum_{k=1}^{r_0}{t_kc_k}\right) &:= 2\sqrt{-1}t_j.\index{gammai@$\gamma_i$}
\end{align*}
For the root spaces we find
\begin{align*}
 (\frakg_\CC)_{\pm\frac{\gamma_i+\gamma_j}{2}} &= \{(u,\mp2\sqrt{-1}L(u),u):u\in(V_{ij})_\CC\},\\
 (\frakg_\CC)_{\pm\frac{\gamma_i-\gamma_j}{2}} &= \{(u,\pm4\sqrt{-1}[L(c_i),L(u)],-u):u\in(V_{ij})_\CC\}.
\end{align*}
The constants $d$ and $e+1$ are exactly the multiplicities of the short and the long roots, respectively. Further, the root spaces of $\frakt_\CC$ in $\frakk_\CC$ are given by
\begin{align*}
 (\frakk_\CC)_{\pm\frac{\gamma_i+\gamma_j}{2}} &= \{(u,\mp2\sqrt{-1}L(u),u):u\in (V^-_{ij})_\CC\},\\
 (\frakk_\CC)_{\pm\frac{\gamma_i-\gamma_j}{2}} &= \{(u,\pm4\sqrt{-1}[L(c_i),L(u)],-u):u\in (V^+_{ij})_\CC\},
\end{align*}
where $V_{ij}^\pm=V_{ij}\cap V^\pm$. Thus, the multiplicities in $\frakk_\CC$ of the short roots $\pm\frac{\gamma_i+\gamma_j}{2}$ and $\pm\frac{\gamma_i-\gamma_j}{2}$, $i<j$, are given by $d-d_0$ and $d_0$, respectively. Since $V_{ij}^-=0$ for all $1\leq i\leq j\leq r_0$ if and only if $V$ is euclidean, and $V_{ii}^-=0$ if and only if $V$ is split, one immediately obtains that the root system $\Sigma(\frakk_\CC,\frakt_\CC)$ is of type
\begin{align*}
 \begin{cases}
  A_{r_0-1} & \mbox{ if $V$ is euclidean,}\\
  C_{r_0} & \mbox{ if $V$ is non-euclidean non-split (including complex non-split),}\\
  D_{r_0} & \mbox{ if $V$ is non-euclidean split.}
 \end{cases}
\end{align*}
We refer to these cases as case $A$, $C$ and $D$. Further, let $\Sigma^+(\frakk_\CC,\frakt_\CC)\subseteq\Sigma(\frakk_\CC,\frakt_\CC)$ be the positive system given by the ordering $\gamma_1>\ldots>\gamma_{r_0}>0$.

\begin{remark}
Note that the cases $A$, $C$ and $D$ do in general not give the type of the Lie algebra $\frakk$. The subalgebra $\frakt_\CC\subseteq\frakk_\CC$ is not necessarily a Cartan subalgebra.
\end{remark}

\subsubsection{Real minimal nilpotent orbits}

\begin{definition}\label{def:minnilp}
For a complex simple Lie algebra $\frakg_{\CC}$, there is a unique nilpotent coadjoint orbit of minimal (positive) dimension, which is called the \textit{minimal nilpotent orbit}. We denote it by $\calO_{\min}^{G_{\CC}}$. More generally, for a complex semisimple $\frakg_{\CC}$, we define the \textit{minimal nilpotent orbit} by
\[
\calO_{\min}^{G_{\CC}}
:= \calO_{\min}^{G_{1,\CC}} \times\dots\times \calO_{\min}^{G_{k,\CC}}
\]
according to the decomposition $\frakg_{\CC} = \frakg_{1,\CC} \oplus\dots\oplus \frakg_{k,\CC}$ into simple Lie algebras.
\end{definition}

In Table~\ref{table1} we list the dimensions of the minimal nilpotent orbits in the simple complex Lie algebras.
\begin{table}[H]
\[
\begin{array}{c|c|c|c|c|c|c|c|c}
\frakg_{\CC} &\sl(k,\CC) &\so(k,\CC) &\sp(k,\CC) &\frakg_2 &\mathfrak{f}_4 &\mathfrak{e}_6 &\mathfrak{e}_7 &\mathfrak{e}_8
\\
\hline
\frac{1}{2}\dim\calO_{\min}^{G_{\CC}} &k-1 &k-3&k &3 &8 &11 &17 &29
\end{array}
\]
\caption{Dimensions of minimal nilpotent orbits in $\frakg_{\CC}^*$\label{table1}}
\end{table}

Let $\frakg$ be a real simple Lie algebra, and $\frakg_{\CC}$ its complexification. We regard $\frakg^*$ as a real form of $\frakg_{\CC}^*$. For any complex nilpotent orbit $\calO^{G_{\CC}}$ in $\frakg_{\CC}^*$, the intersection $\calO^{G_{\CC}}\cap\frakg^*$ may be empty, and otherwise, it consists of a finite number of real nilpotent orbits, say, $\calO_1^G,\dots,\calO_k^G$ (in fact either $k=1$ or $k=2$), which are equi-dimensional:
\[
\dim\calO_1^G=\dots=\dim\calO_k^G=\dim_{\CC}\calO^{G_{\CC}}.
\]
In particular, if $\calO_{\min}^{G_{\CC}}\cap\frakg^*$ is non-empty, then its connected components are real nilpotent orbits of minimal (positive) dimension.

We note that for a real simple Lie algebra $\frakg$, nilpotent orbits of minimal (positive) dimension are not necessarily unique. However, they come from a unique complex nilpotent orbit. The following result due to T. Okuda (see \cite{Oku11a}) treats the case $\calO_{\min}^{G_{\CC}}\cap\frakg^*=\emptyset$ as well:

\begin{proposition}\label{prop:nilcpx}
Let $\frakg$ be a real simple Lie algebra.
\begin{enumerate}
\item[\textup{(1)}]
There exists a unique complex nilpotent orbit in $\frakg_\CC$,
to be denoted by $\calO_{\min,\frakg}^{G_{\CC}}$,
with the following property:
for any nilpotent orbit $\calO^G$ in $\frakg^*$ of minimal (positive) dimension,
$\calO^G$ is a connected component of $\calO_{\min,\frakg}^{G_{\CC}}\cap\frakg$.
\item[\textup{(2)}]
$\calO_{\min,\frakg}^{G_{\CC}}$ is the unique non-zero nilpotent orbit
 in $\frakg_\CC$ of minimal dimension
 with the following property:
\[
     \calO_{\min,\frakg}^{G_{\CC}}\cap \frakg^* \ne \emptyset.
\]
\end{enumerate}
\end{proposition}

As a consequence of Proposition~\ref{prop:nilcpx} we see that the minimal non-zero nilpotent coadjoint orbits in $\frakg^*$ are precisely the connected components of $\calO_{\min,\frakg}^{G_\CC}\cap \frakg^*$.

\begin{proposition}\label{prop:minimal split} Let $\frakg$ be a real simple Lie algebra. Then
the following three conditions are equivalent:
\begin{enumerate}
\item[\textup{(i)}]
$\calO_{\min,\frakg}^{G_{\CC}} \ne \calO_{\min}^{G_{\CC}}$.
\item[\textup{(ii)}]
$\calO_{\min}^{G_{\CC}}\cap\frakg^*$ is empty.
\item[\textup{(iii)}]
$\frakg$ is isomorphic to one of the following Lie algebras:
\[
\frakg = \su^*(2k), \so(k,1), \sp(p,q), \mathfrak{f}_{4(-20)}, \mathfrak{e}_{6(-26)}.
\]
\end{enumerate}
\end{proposition}

\begin{proof}
The equivalence (i) $\Leftrightarrow$ (ii) follows from Proposition \ref{prop:nilcpx}. The equivalence (ii) $\Leftrightarrow$ (iii) can probably be found in the literature, but it is also obtained easily from the criterion in \cite[Theorem 2.4]{Oku11}.
\end{proof}

\begin{example}
[complex simple Lie algebras]
\label{ex:cpxnilp}
Let $\frakg$ be a complex simple Lie algebra,
which we view as a real simple Lie algebra.
Then its complexification is given by
\begin{align}
&\gf_{\CC} := \gf \otimes_{\RR} \CC
\xrightarrow[\varphi_L\oplus\varphi_R]{\sim} \gf_L \oplus \gf_R,
\nonumber
\\
\intertext{where $\varphi_L:\gf\overset{\sim}{\to}\gf_L$ ($\CC$-linear) and
$\varphi_R:\gf\overset{\sim}{\to}\gf_R$ (antilinear) is given by}
&\varphi_L(X) = \frac{1}{2}(X-iJX),
 \
 \varphi_R(X) = \frac{1}{2}(X+iJX).
\label{eqn:gcisom}
\end{align}
It follows from Definition \ref{def:minnilp} that
\[
\calO_{\min}^{G_{\CC}}
\simeq \calO_{\min}^{G_L} \times \calO_{\min}^{G_R}.
\]
Then $\calO_{\min}^{G_{\CC}} \cap \frakg^*$ gives the minimal nilpotent orbit
 of a complex simple Lie algebra $\frakg$.
We note that
there are three smaller nilpotent orbits $\{0\}\times\{0\}$,
$\calO_{\min}^{G_L} \times \{0\}$,
and
$\{0\} \times \calO_{\min}^{G_R}$
of
$G_{\CC}\simeq G_L\times G_R$
on $\gf_{\CC}\simeq\gf_L\oplus\gf_R$.
\end{example}

The following result, which we shall use later, really is a part of the proof for Proposition \ref{prop:nilcpx}. For the convenience of the reader we give a sketch of its proof.

\begin{proposition}[\cite{Oku11a}]\label{prop:non-euclidean-non-split} For the three cases $\gf\simeq\sp(k,k)$,
$\su^*(2k)$, $\so(k,1)$, we have  $\calO_{\min}^{G_{\CC}} \cap \gf = \emptyset$ and
$\calO_{\min,\gf}^{G_{\CC}}$ is described by the weighted Dynkin diagram via the Dynkin--Kostant classification of nilpotent orbits as follows.\\
\noindent{$\gf=\su^*(2k) \quad (k\ge3)$:}
\begin{center}
\begin{longtable}{lll}
\toprule
	Partition & Dimension & Weighted Dynkin diagram  \\
\midrule
	$[2^2,1^{2k-4}]$ & $8k-8$ & \begin{xy}
	*++!D{0} *\cir<2pt>{}        ="A",
	(6,0) *++!D{1} *\cir<2pt>{} ="B",
	(12,0), *++!D{0} *\cir<2pt>{} ="C",
	(18,0), *++!D{0} *\cir<2pt>{} ="D",
	(24,0) = "D_1",
	(30,0) = "D_2",
	(36,0), *++!D{0} *\cir<2pt>{} ="E",
	(42,0), *++!D{0} *\cir<2pt>{} ="F",
	(48,0) *++!D{1} *\cir<2pt>{} ="G",
	(54,0) *++!D{0} *\cir<2pt>{} ="H",
	\ar@{-} "A";"B"
	\ar@{-} "B";"C"
	\ar@{-} "C";"D"
	\ar@{-} "D";"D_1"
	\ar@{.} "D_1";"D_2" ^*U{\cdots}
	\ar@{-} "D_2";"E"
	\ar@{-} "E";"F"
	\ar@{-} "F";"G"
	\ar@{-} "G";"H"
\end{xy} \\
\bottomrule
\end{longtable}
\end{center}
\noindent{$\gf = \so(2k-1,1)$:}
\begin{center}
\begin{longtable}{lll}
\toprule
	Partition & Dimension & Weighted Dynkin diagram  \\
\midrule
	$[3,1^{2k-3}]$ & $4k-4$ & \begin{xy}
	*++!D{2} *\cir<2pt>{}        ="A",
	(6,0) *++!D{0} *\cir<2pt>{} ="B",
	(12,0) *++!D{0}     *\cir<2pt>{} ="C",
	(18,0) = "C_1",
	(24,0) = "C_2",
	(30,0) *++!D{0} *\cir<2pt>{} ="D",
	(36,6) *++!D{0} *\cir<2pt>{} ="E",
	(36,-6) *++!D{0} *\cir<2pt>{} ="F",
	\ar@{-} "A";"B"
	\ar@{-} "B";"C"
	\ar@{-} "C";"C_1"
	\ar@{.} "C_1";"C_2" ^*U{\cdots}
	\ar@{-} "C_2" ; "D"
	\ar@{-} "D";"E"
	\ar@{-} "D";"F"
\end{xy} \\
\bottomrule
\end{longtable}
\end{center}
\noindent{$\gf = \so(2k,1)$:}
\begin{center}
\begin{longtable}{lll}
\toprule
	Partition & Dimension & Weighted Dynkin diagram  \\
\midrule
	$[3,1^{2k-2}]$ & $4k-4$ & \begin{xy}
	*++!D{2} *\cir<2pt>{}        ="A",
	(6,0) *++!D{0} *\cir<2pt>{} ="B",
	(12,0) *++!D{0} *\cir<2pt>{} ="C",
	(18,0) *++!D{0} *\cir<2pt>{} ="D",
	(24,0) = "D_1",
	(30,0) = "D_2",
	(36,0) *++!D{0} *\cir<2pt>{} ="E",
	(42,0) *++!D{0} *\cir<2pt>{} ="F",
	\ar@{-} "A";"B"
	\ar@{-} "B";"C"
	\ar@{-} "C";"D"
	\ar@{-} "D";"D_1"
	\ar@{.} "D_1";"D_2" ^*U{\cdots}
	\ar@{-} "D_2" ; "E"
	\ar@{=>} "E";"F"
\end{xy} \\
\bottomrule
\end{longtable}
\end{center}
\noindent{$\gf = \sp(k,k) \quad (k>1)$:}
\begin{center}
\begin{longtable}{lll} \toprule
	Partition & Dimension & Weighted Dynkin diagram  \\ \midrule
	$[2^2,1^{4k-4}]$ & $8k-2$ & \begin{xy}
	*++!D{0} *\cir<2pt>{}        ="A",
	(6,0) *++!D{1} *\cir<2pt>{} ="B",
	(12,0) *++!D{0} *\cir<2pt>{} ="C",
	(18,0) *++!D{0} *\cir<2pt>{} ="D",
	(24,0) = "D_1",
	(30,0) = "D_2",
	(36,0) *++!D{0} *\cir<2pt>{} ="E",
	(42,0) *++!D{0} *\cir<2pt>{} ="F",
	\ar@{-} "A";"B"
	\ar@{-} "B";"C"
	\ar@{-} "C";"D"
	\ar@{-} "D";"D_1"
	\ar@{.} "D_1";"D_2" ^*U{\cdots}
	\ar@{-} "D_2" ; "E"
	\ar@{<=} "E";"F"
\end{xy} \\
\bottomrule
\end{longtable}
\end{center}
\noindent{$\gf = \sp(1,1)$:}
\begin{center}
\begin{longtable}{lll} \toprule
	Partition & Dimension & Weighted Dynkin diagram  \\ \midrule
	$[2^2]$ & $6$ & \begin{xy}
	(36,0) *++!D{0} *\cir<2pt>{} ="E",
	(42,0) *++!D{2} *\cir<2pt>{} ="F",
	\ar@{<=} "E";"F"
\end{xy} \\
\bottomrule
\end{longtable}
\end{center}
\end{proposition}

\begin{proof} Using the characterization of the orbit closure relation in terms of domination of partitions \cite[Theorem~6.2.5]{CM93}, this amounts to the following: Calculate the weighted Dynkin diagrams for the partitions dominated by the ones given in the proposition and verify that the given partitions are the smallest ones satisfying the \textit{matching condition} given in \cite[Theorem~2.4]{Oku11}. For the calculation of the weighted Dynkin diagrams one can use \cite[Lemmas 3.6.4, 5.3.1, 5.3.4, and 5.3.3]{CM93} for $\su^*(2k)$, $\sp(k,k)$,  $\so(2k-1,1)$, and $\so(2k,1)$, respectively. Finally, the dimensions of the orbits can be determined from the weighted Dynkin diagram using \cite[Lemma~4.1.3]{CM93}.
\end{proof}

The Kantor--Koecher--Tits construction from Subsection~\ref{sec:KKT} shows that the conformal group  $\Co(V_\CC)$ of the complexified Jordan algebra $V_\CC$ contains the conformal group $\Co(V)$ as a subgroup. Moreover, the Lie algebra  $\co(V_\CC)$ is the complexification $\frakg_\CC$ of $\frakg=\co(V)$. Thus it makes sense to denote the identity component  $\Co(V_\CC)_0$ of $\Co(V_\CC)$ by $G_\CC$ and view $G$ as a subgroup of $G_\CC$.

Via the Killing form we identify $\frakg_\CC^*$ and $\frakg_\CC$ and view $\calO_{\min}^{G_\CC}$ also as an adjoint orbit in $\frakg_\CC$. We further identify $V$ with a subspace of $\frakg$ by the embedding $V\hookrightarrow\frakg, x\mapsto(x,0,0)$.

\begin{theorem}\label{lem:Lagrange} Let $V$ be a simple real Jordan algebra with simple $V^+$. Set $\calO_{\min}^{G}:=G\cdot(c_1,0,0)$. Then
\begin{enumerate}
\item[{\rm(1)}] $\calO_{\min}^{G}$ is a minimal nonzero nilpotent coadjoint orbit in $\frakg^*$.
\item[{\rm(2)}] $\calO_{\min,\frakg}^{G_\CC}= G_\CC\cdot(c_1,0,0)$.
\item[{\rm(3)}]  $\calO_{\min}^{G}$ is the connected component of $\calO_{\min,\frakg}^{G_\CC}\cap \frakg^*$ containing $(c_1,0,0)$.
\item[{\rm(4)}] The orbit $\calO$ is a Lagrangian submanifold of $\calO_{\min}^{G}$. In particular, we have $2\dim_\RR\calO = \dim_\RR\calO_{\min}^{G} =\dim_\CC \calO_{\min,\frakg}^{G_\CC}$.
\end{enumerate}
\end{theorem}

\begin{proof}
The orbit $\calO=L\cdot c_1\subseteq V$ of $L$ is obviously contained in the adjoint orbit $\calO_{\min}^{G}=G\cdot(c_1,0,0)\subseteq G_\CC\cdot(c_1,0,0)\cap \frakg$. Clearly $\calO_{\min}^{G}$ is a nonzero nilpotent adjoint orbit in $\frakg$, which, however, is not necessarily contained in $\calO_{\min}^{G_\CC}$. Thus, by Proposition~\ref{prop:minimal split}, we are in the situation of Proposition~\ref{prop:non-euclidean-non-split}, and (1) will follow by calculating the dimension of $\calO_{\min}^{G}$. The statements (2) and (3) are then clear from Proposition \ref{prop:nilcpx}.

Let $\frakg_\CC^{(c_1,0,0)}\subseteq\frakg_\CC$ be the centralizer subalgebra of the element $(c_1,0,0)$. We claim that
\begin{align}
 \frakg_\CC &= \frakg_\CC^{(c_1,0,0)} \oplus (0,\overline{\fraks}_\CC,0) \oplus \bigoplus_{j=1}^{r_0}{(0,0,(V_{1j})_\CC)},\label{eq:DimOminAlg}
\end{align}
where $\overline{\fraks}\subseteq\frakl$ is a complement of the centralizer $\fraks$ of $c_1$ in $\frakl$. In fact, we have
\begin{align*}
  [(c_1,0,0),(u,T,v)] = 0\
 \Leftrightarrow{}&\ (-Tc_1,2c_1\Box v,0) = 0\\
 \Leftrightarrow{}&\ Tc_1=0\mbox{ and }c_1\Box v=0\\
 \Leftrightarrow{}&\ T\in\fraks\mbox{ and }c_1\cdot v=0,
\end{align*}
which shows the claim. From \eqref{eq:DimOminAlg} we then obtain with Lemma \ref{lem:DimO1}:
\begin{align*}
 \dim_\CC(G_\CC\cdot(c_1,0,0)) &= \dim_\CC\frakg_\CC - \dim_\CC\frakg_\CC^{(c_1,0,0)}\\
 &= (\dim\,\frakl-\dim\,\fraks) + \sum_{j=1}^{r_0}{\dim\,V_{1j}}\\
 &= \dim\,\calO + (e+1) + (r_0-1)d\\
 &= 2(e+1+(r_0-1)d).
\end{align*}
Since $(c_1,0,0)\in\frakg$ we have $\dim_\RR((G_\CC\cdot(c_1,0,0))\cap\frakg)=\dim_\CC(G_\CC\cdot(c_1,0,0))$. Using Lemma \ref{lem:DimO1} again we find that $\dim\,\calO=\frac{1}{2}\dim_\RR((G_\CC\cdot(c_1,0,0))\cap\frakg)$. To show that $\calO$ is actually a Lagrangian submanifold it remains to show that the Kostant--Souriau symplectic form vanishes on $\calO$. But this is clear since $\calO\subseteq\frakn$ and $\frakn$ is an abelian subalgebra. This implies the first claim and the equality $2\dim_\RR\calO = \dim_\RR\calO_{\min}^{G}$. The last equality follows since
$(\frakg^{(c_1,0,0)})_\CC = (\frakg_\CC)^{(c_1,0,0)}$. Using these formulas one can now check the dimension of $\calO_{\min}^{G}$ and thus complete the proof.
\end{proof}

For the conformal group $G$ of a split Jordan algebra $V$, the minimal nilpotent orbit $\calO_{\min}^{G_\CC}$ has real points. More precisely, we have:

\begin{proposition}\label{thm:Lagrangian}
Assume that $V$ is split, i.e. $e=0$. Then
$$\calO_{\min}^{G_\CC}=\calO_{\min,\frakg}^{G_\CC}.$$
In particular, the orbit $\calO$ is a Lagrangian submanifold of the non-zero intersection $\calO_{\min}^{G_\CC}\cap\frakg$.
\end{proposition}

\begin{proof} This can also be derived from Okuda's results, but we give a proof which does not use the classification of nilpotent orbits.

We only need to show that the minimal adjoint orbit $\calO_{\min}^{G_\CC}$ contains the element $(c_1,0,0)$. By \cite[Theorem 4.3.3]{CM93} the adjoint orbit $\calO_{\min}^{G_\CC}$ contains every non-zero root vector for the highest root in any root system with respect to a Cartan subalgebra. To find such a root vector we complete the abelian subalgebra $\frakt_\CC$ to a Cartan subalgebra $\frakh_\CC\subseteq\frakg_\CC$. We choose an order on $\Sigma(\frakg_\CC,\frakh_\CC)$ such that the restriction to $\frakt_\CC$ preserves the order. Then the highest root projects onto $\gamma_1$ and a non-zero highest root vector is in the root space $(\frakg_\CC)_{\gamma_1}$ and hence of the form $(x,-2\sqrt{-1}L(x),x)$ for $x\in(V_{11})_\CC$. Since $V$ was assumed to be split we have $(V_{11})_\CC=\CC c_1$. Note that it suffices to prove that also $(c_1,0,0)$ is in the minimal non-zero nilpotent adjoint orbit. To prove this claim we first note that
\begin{align*}
 & \Ad(\exp(0,0,-\sqrt{-1}c_1))(x,-2\sqrt{-1}L(x),x)\\
 ={}& \exp(\ad(0,0,-\sqrt{-1}c_1))(x,-2\sqrt{-1}L(x),x)\\
 ={}& (x,-2\sqrt{-1}L(x),x) + (0,2\sqrt{-1}x\Box c_1,-2L(x)c_1) + \frac{1}{2}(0,0,2L(x)c_1)\\
 ={}& (x,-2\sqrt{-1}L(x),x) + (0,2\sqrt{-1}L(x),-2x) + (0,0,x)\\
 ={}& (x,0,0),
\end{align*}
and therefore $(x,0,0)\in\calO_{\min}^{G_\CC}$. Since the group $L_\CC\subseteq G_\CC$ contains all dilatations by elements in $\CC^\times=\CC\setminus\{0\}$, the claim follows.
\end{proof}

Suppose that $V$ is a complex simple Jordan algebra,
 viewed as a real Jordan algebra.
In this case $V$ is not split, and  
$\frak g$ is a complex simple Lie algebra viewed as a real simple Lie algebra (see Table~\ref{tb:Groups}).
With the notation as in \eqref{eqn:gcisom},
$(\varphi_L,\varphi_R)$ maps $\calO^G_{\min}= G\cdot (c_1,0,0)$ into the nilpotent $G_\CC$-orbit $\calO_{\min}^{G_{\CC}}\simeq\calO^{G_L}_{\min}\times \calO^{G_R}_{\min}$.
A dimension count now shows that
$\calO^G_{\min}$ is open in $\frakg\cap\big(\calO^{G_L}_{\min}\times \calO^{G_R}_{\min}\big)$ and hence
$\calO$ is a Lagrangian submanifold of $\frakg\cap\big(\calO^{G_L}_{\min}\times \calO^{G_R}_{\min}\big)$.

Combining the above considerations for complex simple Jordan algebras with Proposition~\ref{thm:Lagrangian} and Proposition~\ref{prop:non-euclidean-non-split} describing the three non-euclidean non-split cases we obtain:

\begin{theorem}\label{thm:min-split-nilpotent-orbits} Let $G$ be the conformal group of a simple real Jordan algebra $V$ with simple $V^+$. Then the nilpotent coadjoint orbit $\calO_{\min,\frakg}^{G_\CC}$ (see Proposition \ref{prop:nilcpx} for the definition) is given as follows:
\[
\calO_{\min,\frakg}^{G_\CC}
= \begin{cases}
     \calO_{\min}^{G_{\CC}}
     &\text{(split)}
  \\
     \calO_{\min}^{G_{\CC}}\simeq\calO_{\min}^{G_L} \times \calO_{\min}^{G_R}
     &\text{(complex, non-split)}
  \\
     \text{as in Proposition~\ref{prop:non-euclidean-non-split}}
     &\text{(non-euclidean, non-split)}
  \end{cases}
\]
\end{theorem}

\subsubsection{$\frakk$-representations with a $\frakk_\frakl$-fixed vector}\label{ssec: kl-fix vector}

As previously remarked, $(\frakk,\frakk_\frakl)$ is a symmetric pair. Using the Cartan--Helgason theorem we can describe the highest weights of all irreducible unitary $\frakk$-representations which have a non-zero $\frakk_\frakl$-fixed vector. For this we extend $\frakt$ to a maximal torus $\frakt^c$ of $\frakk$ with the property that $\frakt^c=\frakt\oplus(\frakt^c\cap\frakk_\frakl)$ and choose a positive system $\Delta^+(\frakk_\CC,\frakt^c_\CC)$ such that the restriction to $\frakt_\CC$ induces a surjection
\begin{align*}
 \Delta^+(\frakk_\CC,\frakt^c_\CC)\cup\{0\} &\rightarrow \Sigma^+(\frakk_\CC,\frakt_\CC)\cup\{0\}.
\end{align*}
Then the Cartan--Helgason theorem yields:

\begin{proposition}\label{prop:klSphericalsReps}
The highest weight $\alpha\in(\frakt^c_\CC)^*$ of an irreducible $\frakk$-representation with a non-zero $\frakk_\frakl$-fixed vector vanishes on $\frakt^c_\CC\cap(\frakk_\frakl)_\CC$. The possible highest weights which give unitary irreducible $\frakk_\frakl$-spherical representations are precisely given by
\begin{align*}
 \Lambda_{\frakk_\frakl}^+(\frakk) := \begin{cases}\displaystyle\left\{\sum_{i=1}^{r_0}{t_i\gamma_i}:t_i\in\RR,\,t_i-t_j\in\ZZ,\,t_1\geq\ldots\geq t_{r_0}\right\} & \mbox{in case $A$,}\\\displaystyle\left\{\sum_{i=1}^{r_0}{t_i\gamma_i}:t_i\in\ZZ,\,t_1\geq\ldots\geq t_{r_0}\geq0\right\} & \mbox{in case $C$,}\\\displaystyle\left\{\sum_{i=1}^{r_0}{t_i\gamma_i}:t_i\in\frac{1}{2}\ZZ,\,t_i-t_j\in\ZZ,\,t_1\geq\ldots\geq t_{r_0-1}\geq|t_{r_0}|\right\} & \mbox{in case $D$.}\end{cases}\index{Lambdakklplus@$\Lambda_{\frakk_\frakl}^+(\frakk)$}
\end{align*}
Further, in each irreducible $\frakk_\frakl$-spherical $\frakk$-representation the space of $\frakk_\frakl$-fixed vectors is one-dimensional.
\end{proposition}

For $\alpha\in\Lambda_{\frakk_\frakl}^+(\frakk)$ we denote by $E^\alpha$\index{Ealpha@$E^\alpha$} the irreducible $\frakk_\frakl$-spherical representation of $\frakk$ with highest weight $\alpha$.

\subsection{Construction of $L^2$-models}\label{sec:ConstructionL2Model}

In this section we construct $L^2$-models of representations $\pi$ of a finite cover $G^\vee$ of the conformal group $G$ with associated variety $\overline{\calO_{\min,\frakg}^{G_\CC}}$, the closure of $\calO^{G_\CC}_{\min,\frak g}$ (see Proposition \ref{prop:nilcpx} for its definition). This implies that the Gelfand--Kirillov dimension of $\pi$ is minimal among all irreducible infinite dimensional unitary representations of $G^\vee$.

We start by constructing a representation of the Lie algebra $\frakg$ on $C^\infty(\calO_\lambda)$ for every $\lambda\in\calW$. Then, for $\calO_\lambda=\calO$ the minimal non-zero orbit we prove that the associated variety of the representation on $C^\infty(\calO)$ is equal to $\overline{\calO_{\min,\frakg}^{G_\CC}}$. In particular, the representation is minimal if $V$ is split or complex, and $\frakg$ is not a type $A$ Lie algebra. We then define a subrepresentation $W$ of $C^\infty(\calO)$ which is generated by one single vector. This subrepresentation contains a non-zero $\frakk$-finite vector if and only if $V\ncong\RR^{p,q}$ with $p+q$ odd, $p,q\geq2$. It is contained in $L^2(\calO,\td\mu_\lambda)$ if for $V$ of split rank $r_0=1$ one assumes in addition that $\sigma:=r\lambda\in(0,-2\nu)$. Under the same conditions which guarantee square integrability we can finally integrate $W$ to a unitary irreducible representation of a finite cover of $G$ on the Hilbert space $L^2(\calO,\td\mu_\lambda)$ (see Theorem \ref{thm:IntgkModule}).

\subsubsection{Infinitesimal representations on $C^\infty(\calO_\lambda)$}\label{sec:InfinitesimalRep}

On each Hilbert space $L^2(\calO_\lambda,\td\mu_\lambda)$, $\lambda\in\calW$, there is a natural unitary representation $\rho_\lambda$\index{rholambda@$\rho_\lambda$} of the subgroup $Q$ (see \eqref{eq:DefQ}) given by
\begin{align}
 \rho_\lambda(n_a)\psi(x) &:= e^{\sqrt{-1}(x|a)}\psi(x) & n_a &\in N,\label{eq:L2Rep1}\\
 \rho_\lambda(g)\psi(x) &:= \chi(g^*)^{\frac{\lambda}{2}}\psi(g^*x) & g &\in L\label{eq:L2Rep2}
\end{align}
for $\psi\in L^2(\calO_\lambda,\td\mu_\lambda)$. The following proposition is a consequence of the Mackey theory:

\begin{proposition}\label{prop:MackeyRep}
For $\lambda\in\calW$ the representation $\rho_\lambda$ of $Q$ on $L^2(\calO_\lambda,\td\mu_\lambda)$ is unitary and irreducible.
\end{proposition}

We ask whether $\rho_\lambda$ extends to a unitary irreducible representation of $G$ (or some finite cover) on $L^2(\calO_\lambda,\td\mu_\lambda)$. For this we extend the differential representation $\td\rho_\lambda$ of $\frakq^{\textup{max}}$ to $\frakg$. Then for $\calO_\lambda=\calO$ the minimal non-zero orbit we determine the cases in which $\td\pi_\lambda$ integrates to a unitary representation of a Lie group having Lie algebra $\frakg$.

For each $\lambda\in\calW$ we define a Lie algebra representation $\td\pi_\lambda$\index{dpilambda@$\td\pi_\lambda$} of $\frakg$ on $C^\infty(\calO_\lambda)$ which extends the derived action of $\rho_\lambda$. On $\frakq^{\textup{max}}=\frakn+\frakl$ we let
\begin{align*}
 \td\pi_\lambda(X) &:= \left.\frac{\td}{\td t}\right|_{t=0}\rho_\lambda(e^{tX}) & \forall\,X\in\frakq^{\textup{max}}.
\end{align*}
For $\psi\in C^\infty(\calO_\lambda)$ we have
\begin{align}
 \td\pi_\lambda(X)\psi(x) &= \sqrt{-1}(x\psi(x)|u) & \mbox{for }X &= (u,0,0),\label{eq:L2DerRep1}\\
 \td\pi_\lambda(X)\psi(x) &= D_{T^*x}\psi(x)+\frac{r\lambda}{2n}\Tr(T^*)\psi(x) & \mbox{for }X &= (0,T,0),\label{eq:L2DerRep2}\\
\intertext{where we have used \eqref{eq:ChiDet} for the $\frakl$-action. Here $D_u\psi(x)=\left.\frac{\td}{\td t}\right|_{t=0}\psi(x+tu)$ is the derivative in the direction of $u$. In view of the Gelfand--Naimark decomposition \eqref{eq:GelfandNaimark} it remains to define $\td\pi_\lambda$ on $\overline{\frakn}$ in order to define a representation of the whole Lie algebra $\frakg$. For this we use the Bessel operator $\calB_\lambda$. By Theorem \ref{thm:BlambdaTangential} the operator $\calB_\lambda$ is tangential to $\calO_\lambda$ and hence, for $\psi\in C^\infty(\calO_\lambda)$ the formula}
 \td\pi_\lambda(X)\psi(x) &= \frac{1}{\sqrt{-1}}(\calB_\lambda\psi(x)|v) & \mbox{for }X &= (0,0,-v),\label{eq:L2DerRep3}
\end{align}
defines a function $\td\pi_\lambda(X)\psi\in C^\infty(\calO_\lambda)$.

\begin{proposition}\label{prop:LieAlgRep}
For $\lambda\in\calW$ the formulas \eqref{eq:L2DerRep1}, \eqref{eq:L2DerRep2} and \eqref{eq:L2DerRep3} define a representation $\td\pi_\lambda$ of $\frakg$ on $C^\infty(\calO_\lambda)$. This representation is compatible with $\rho_\lambda$, i.e. for $g\in Q$ and $X\in\frakg$ we have
\begin{align}
 \rho_\lambda(g)\td\pi_\lambda(X) &= \td\pi_\lambda(\Ad(g)X)\rho_\lambda(g).\label{eq:CompatibilityRhoPi}
\end{align}
\end{proposition}

The proof is a lengthy but elementary calculation which can be found in \cite[Proposition 2.1.2]{Moe10}.

\begin{remark}
In Proposition~\ref{prop:IntertwinerPrincipalSeries} we will show that $\td\pi_\lambda$ coincides with the Fourier transform of the differential action on a degenerate principal series representation in the non-compact picture. The definition \eqref{eq:L2DerRep3} of the $\overline{\frakn}$-action is motivated by these considerations. This also gives an alternative proof that $\td\pi_\lambda$ is indeed a Lie algebra representation.
\end{remark}

\subsubsection{Associated varieties and the Joseph ideal}\label{sec:AssocVar}

Recall that for a finitely generated representation $\tau$ of $\frakg$ with annihilator $\calJ:=\Ann(\tau)\subseteq U(\frakg)$ the associated variety $\mathcal V(\calJ)\subseteq\frakg_\CC^*$ is the variety corresponding to the graded ideal
\begin{align*}
 J &:= \gr(\calJ)\subseteq\gr(U(\frakg)) \cong S(\frakg_\CC) \cong \CC[\frakg_\CC^*].
\end{align*}

For a simple Lie algebra
$\frakg$ not of type $A_n$, Joseph \cite{Jo76} introduced a unique completely prime ideal $\calJ\subseteq U(\frakg)$ with the property that  $\mathcal V(\calJ)$ is equal to the closure $\overline{\calO_{\min}^{G_\CC}}\subseteq\frakg_\CC^*$ (see also \cite[Theorem 3.1]{GS04}). This ideal is primitive, and  is called the \textit{Joseph ideal}.

\begin{definition}\label{def:minrep}
Let $M$ be a simple $\frakg$-module. We say $M$ is \textit{minimal} if its annihilator is the Joseph ideal. For an irreducible unitary representation $\pi$ of a real simple Lie group $G$, we say $\pi$ is a \textit{minimal representation} if the annihilator of the differential representation $d\pi$ is the Joseph ideal.
\end{definition}

We note that if $G$ is a complex simple Lie group, we have $\frakg_{\CC} \simeq \frakg\oplus\overline{\frakg} \equiv \frakg_L\oplus\frakg_R$, and the Joseph ideal is given by $\mathcal{I}_L \otimes U(\frakg_R)+U(\frakg_L) \otimes \mathcal{I}_R$ (see Example \ref{ex:cpxnilp}).

For any admissible irreducible representation $\pi$ of a real reductive group $G$, the associated variety ${\cal V}(\ker d\pi)$ has real points. In particular, there is no minimal representation (in the sense of Definition \ref{def:minrep}) of a simple Lie group $G$ if $\calO_{\min}^{G_{\CC}}$ does not have real points. In view of Proposition \ref{prop:nilcpx}, the closure of $\calO_{\min,\frakg}^{G_\CC}$ is the smallest possible associated variety of such a representation in any case. We shall see in Theorem \ref{thm:Minimality} that our unitary representation $\pi_\lambda$ for $\lambda\in\calW$ such that $\calO_\lambda=\calO$ actually attain the associated varieties of the annihilator ideals $\calV(\ker\td\pi_\lambda)$.

{}From now on we restrict ourselves to the representations $\td\pi_{\lambda}$, where $\lambda\in\calW$ is such that $\calO_\lambda=\calO$. Then $\lambda=\lambda_1=\frac{r_0d}{2r}$ for $r_0>1$ and $\lambda>0$ for $r_0=1$.

Let $\overline{\XX}:=\overline{\calO}\cup(-\overline{\calO})=\calO\cup(-\calO)\cup\{0\}$.
By \cite[Theorem 2.9]{GK98} we have
\[
     \overline{\XX}=\{x\in V:\rk(P(x))\leq\rk(P(c_1))=e+1\}
\]
and hence, $\overline{\XX}$ is a real affine subvariety of $V$.  Note that $K_L$, being a connected real algebraic group, is irreducible, whence also $K_L\times\RR c_1$ is irreducible. But $\overline{\XX}$ is the image of $K_L\times\RR c_1$ under the map $(k,tc_1)\mapsto ktc_1$, so it is irreducible as well. The origin is the only singular point of $\overline{\XX}$. By \cite[Theorem~2.4.10 and Proposition~1.7.3]{Pal81} $\XX:=\overline{\XX}\setminus\{0\}$ is an open, dense, smooth irreducible affine algebraic subspace of $\overline{\XX}$. Fix a basis for $V$ and denote the coordinates of a point in $V$ by  $(x_1,\ldots,x_n)$. Then, adding a coordinate $t$ and the equation $t(x_1^2+\cdots+x_n^2)=1$ to the description of $\overline{\XX}$ as a real affine subvariety of $V$, we see that also $\XX$ is a real affine algebraic variety.

Denote by $\DD(\XX)$ the algebra of regular differential operators on $\XX$. Then Proposition~\ref{prop:LieAlgRep} implies:

\begin{proposition}\label{prop:D(X)-factorization}
For $\lambda\in\calW$ with $\calO_\lambda=\calO$ the representation $\td\pi_\lambda$ factors through the algebra homomorphism $U(\frakg)\to\DD(\XX)$.
\end{proposition}

Below we will prove (see Corollary~\ref{cor:annihilator of dpi}) that the annihilator of $\td\pi_\lambda$ coincides with the annihilator of a finitely generated $(\frakg,\frakk)$-module if we assume that $\frakg\ncong\so(p,q)$ with $p+q$ odd, $p,q\ge 3$, and that if $\frakg\cong\so(n,1)$, then $\sigma=r\lambda\in(0,2(n-1))$. This allows us to prove the following theorem.

\begin{theorem}\label{thm:Minimality}
Let $\lambda\in\calW$ such that $\calO_\lambda=\calO$. Further suppose that $\mathfrak g\not \cong \so(p,q)$ with $p+q$ odd, $p,q\ge 3$, and that if $\frakg\cong\so(n,1)$, then $\sigma=r\lambda\in(0,2(n-1))$. Then the annihilator $\calJ$ of $\td\pi_\lambda$ is completely prime and its associated variety $\mathcal V(\calJ)$ is the closure of $\calO_{\min,\frakg}^{G_\CC}$.
\end{theorem}

\begin{proof}
By \cite{Smi84} the Gelfand--Kirillov dimension of $\DD(\XX)$ is given by $2\dim\,\XX$.
The dimension of $\calO$, and hence of $\XX$, by Theorem~\ref{lem:Lagrange}\,(4) is equal to $\dim_{\CC} \calO_{\min,\frakg}^{G_\CC}$.
Therefore, the Gelfand--Kirillov dimension of $U(\frakg)/\calJ$ does not exceed it.
It is equal to the Krull dimension of $S(\frakg_\CC)/\gr(\calJ)$, which, on the other hand, equals the dimension of the associated variety $\mathcal V(\calJ)$.
Therefore, the associated variety has dimension less or equal to $\dim_\CC \calO_{\min,\frakg}^{G_\CC}$.
By Proposition~\ref{prop:nilcpx},
$\mathcal V(\calJ)$ has minimal dimension and is equal to $\overline{\calO_{\min,\frakg}^{G_\CC}}$.

It remains to check that $\calJ$ is completely prime.  Since $\XX$ is irreducible, the ring $\DD(\XX)$ does not contain zero-divisors (see e.g. \cite[Proposition 2.4 and the remark thereafter]{BW04}). Therefore, the annihilator $\calJ$ has to be completely prime.
\end{proof}

\subsubsection{Construction of the $(\frakg,\frakk)$-module}\label{sec:ConstructiongkModule}

Again we fix $\lambda\in\calW$ such that $\calO_\lambda=\calO$. For $r_0>1$ we have $\lambda=\lambda_1$, but for $r_0=1$ arbitrary parameters $\lambda>0$ can occur. In this case we again put $\sigma:=\frac{r}{r_0}\lambda=r\lambda$. The representation $\td\pi_\lambda$ extends to a representation of the universal enveloping algebra $U(\frakg)$\index{Ug@$U(\frakg)$} on $C^\infty(\calO_\lambda)$.

Following \cite{KM07b}, we renormalize the $K$-Bessel function as
$$\widetilde{K}_\alpha(z)=\left(\frac{z}{2}\right)^{-\alpha}K_\alpha(z).$$
Let $\nu\equiv\nu(V)$ be the integer given in \eqref{eqn: def nu}.
We then introduce a radial function $\psi_0$ on $\calO$ as follows:
\begin{enumerate}
\item[\textup{(1)}] If $r_0>1$, we put
\begin{align}
 \psi_0(x) &:= \widetilde{K}_{\frac{\nu}{2}}(|x|),& x\in\calO.\label{def:Psi0}
\end{align}
\item[\textup{(2)}] If $r_0=1$, we put
\begin{align}
 \psi_0(x) &:= \widetilde{K}_{\frac{\nu+\sigma}{2}}(|x|), & x\in\calO.\label{def:Psi0-r_0=1}
\end{align}
\end{enumerate}
In both cases we further let
\begin{align*}
 W_0 &:= \td\pi_\lambda(U(\frakk))\psi_0 & \mbox{and} && W &:= \td\pi_\lambda(U(\frakg))\psi_0.\index{W0@$W_0$}\index{W@$W$}
\end{align*}

For $\frakg=\sp(k,\RR)$ this construction only leads to the even part of the Weil representation, but it is also possible to construct the odd part in the same spirit. For $V=\Sym(k,\RR)$, $k\geq1$, denote by $H:=\Stab_{\GL(k,\RR)}(c_1)\subset\GL(k,\RR)$ the stabilizer of $c_1=E_{11}\in\calO$. It is explicitly given by $H=(\{\pm1\}\times\GL(k-1,\RR))\ltimes\RR^{k-1}$. Let $\calL$ be the $\GL(k,\RR)$-equivariant line bundle associated to the character of $H$ given by
\begin{align*}
 (\pm1,g,n)\mapsto\pm1, && g\in\GL(k-1,\RR),n\in\RR^{k-1}.
\end{align*}
Since the line bundle $\calL\to\calO$ is flat, the Lie algebra action $\td\pi_\lambda$ ($\lambda=\lambda_1$) of $\frakg=\sp(k,\RR)$ on $C^\infty(\calO)$ induces an action $\td\pi_\lambda^-$ of $\frakg$ on smooth sections of the bundle $\calL\to\calO$. Further observe that the $\GL(k,\RR)$-equivariant measure $\td\mu$ on $\calO$ also defines an $L^2$-space of sections of the line bundle $\calL\to\calO$ which we denote by $L^2(\calO,\calL)$. Note that the folding map $\RR^k\setminus\{0\}\to\calO,\,x\mapsto x\,{}^t\!x$ induces a unitary isomorphism (up to scalar multiples)
\begin{align}
 \calU^-:L^2(\calO,\calL)\to L^2_{\textup{odd}}(\RR^k),\,\calU^-\psi(x):=\psi(x\,{}^t\!x).\label{eq:DefCalU-}
\end{align}
We put
\begin{align}
 \psi_0^-(x) &:= (x|c_1)^{\frac{1}{2}}\widetilde{K}_{\frac{\nu}{2}}(|x|) = \frac{\sqrt{\pi}}{2}\sqrt{x_{11}}e^{-|x|}, & x\in\calO.\label{eq:DefPsi-}
\end{align}
Then $\psi_0^-$ gives an $L^2$-section of the line bundle $\calL\to\calO$.
Define
\begin{align*}
 W_0^- &:= \td\pi_\lambda^-(U(\frakk))\psi_0^- & \mbox{and} && W^- &:= \td\pi_\lambda^-(U(\frakg))\psi_0^-.\index{W@$W$}\index{W0@$W_0$}
\end{align*}
The space $W$ (resp. $W^-$) is clearly a $\frakg$-subrepresentation of $C^\infty(\calO)$ (resp. $C^\infty(\calO,\calL)$) and $W_0$ (resp. $W_0^-$) is a $\frakk$-subrepresentation of $W$ (resp. $W^-$). In order to show that $W$ (resp. $W^-$) is actually a $(\frakg,\frakk)$-module, we shall prove that $W_0$ (resp. $W_0^-$) is finite-dimensional. This can be done by direct computation as follows.

We start with the case $r_0>1$. In this case we will need the following notation to give a precise statement: Denote by
\begin{align*}
 \calH^k(\RR^n) &:= \{p\in\CC[x_1,\ldots,x_n]:p\mbox{ is homogeneous of degree $k$ and harmonic}\}\index{HkRn@$\calH^k(\RR^n)$}
\end{align*}
the space of \textit{spherical harmonics} on $\RR^n$ of degree $k$.

In the case $V=\RR^{p,q}$ we view polynomials in $\calH^k(\RR^p)$ and $\calH^k(\RR^q)$ as polynomials on $V$ via the projections $\RR^p\times\RR^q\ni(x',x'')\mapsto x'\in\RR^p$ and $\RR^p\times\RR^q\ni(x',x'')\mapsto x''\in\RR^q$. For $\calP$ either $\calH^k(\RR^p)$ or $\calH^k(\RR^q)$ we denote by $\widetilde{K}_\alpha\otimes\calP$ the space of functions
\begin{align*}
 \widetilde{K}_\alpha\otimes\varphi:\calO\rightarrow\CC,\,x\mapsto\widetilde{K}_\alpha(|x|)\varphi(x)
\end{align*}
with $\varphi\in\calP$.

In the case $V=\Sym(k,\RR)$ we set for $u\in\CC^k$
\begin{align}
 \varphi_u(x) &:= (\,{}^t\!u xu)^{\frac{1}{2}}e^{-|x|}, & x\in\calO.\label{eq:DefPhiu}
\end{align}
Then $\varphi_u$ is not well-defined as a function on $\calO$, but gives a section of the line bundle $\calL\to\calO$.

\begin{theorem}\label{prop:Kfinite}
Let $V$ be a simple Jordan algebra with simple $V^+$. Assume that $r_0>1$ so that $\lambda=\lambda_1$. Then the $\frakk$-module $W_0$ is finite-dimensional if and only if $V\ncong\RR^{p,q}$ with $p+q$ odd, $p,q\geq2$. If this is the case, $W_0$ is irreducible with highest weight
\begin{align}
 \alpha_0 &:= \begin{cases}\frac{d}{4}\sum_{i=1}^{r_0}{\gamma_i} & \mbox{if $V$ is euclidean,}\\0 & \mbox{if $V$ is non-euclidean of rank $\geq3$,}\\\frac{1}{2}\left|d_0-\frac{d}{2}\right|\gamma_1+\frac{1}{2}\left(d_0-\frac{d}{2}\right)\gamma_2 & \mbox{if $V\cong\RR^{p,q}$, $p,q\geq2$.}\end{cases}\label{eq:DefGamma0}\index{alpha0@$\alpha_0$}
\end{align}
In the case $V=\Sym(k,\RR)$ the $\frakk$-module $W_0^-$ is also irreducible with highest weight
\begin{align}
 \alpha_0^- &:= \left(\frac{d}{4}\sum_{i=1}^{r_0}{\gamma_i}\right)+\frac{\gamma_1}{2}.\label{eq:DefGamma0-}
\end{align}
More precisely:
\begin{enumerate}
\item[\textup{(a)}] If $V$ is euclidean, then
 \begin{align*}
  W_0 &= \CC\psi_0
 \end{align*}
 and $\frakk$ acts by
 \begin{align*}
  \td\pi_{\lambda_1}(u,D,-u)\psi_0 &= \frac{d}{2}\sqrt{-1}\tr(u)\psi_0.
 \end{align*}
\item[\textup{(b)}] If $V$ is non-euclidean of rank $r\geq3$, then
 \begin{align*}
  W_0 &= \CC\psi_0
 \end{align*}
 and $\psi_0$ is a $\frakk$-fixed vector.
\item[\textup{(c)}] If $V=\RR^{p,q}$ with $p+q$ even, $p,q\geq2$, then
 \begin{align}
  W_0 &= \bigoplus_{k=0}^{\left|\frac{p-q}{2}\right|}{\widetilde{K}_{\frac{\nu}{2}+k}\otimes\calH^k(\RR^{\min(p,q)})}\cong\calH^{\left|\frac{p-q}{2}\right|}(\RR^{\min(p,q)+1}).\label{eq:MinKtypeRk2finite}
 \end{align}
\item[\textup{(d)}] If $V=\RR^{p,q}$ with $p+q$ odd, $p,q\geq2$, then
 \begin{align}
  W_0 &= \bigoplus_{k=0}^\infty{\widetilde{K}_{\frac{\nu}{2}+k}\otimes\calH^k(\RR^{\min(p,q)})}.\label{eq:MinKtypeRk2infinite}
 \end{align}
\item[\textup{(e)}] If $V=\Sym(k,\RR)$ then
\begin{align*}
 W_0^- &= \{\varphi_u:u\in\CC^k\}.
\end{align*}
\end{enumerate}
\end{theorem}

\begin{proof}
Since $\psi_0$ is $K_L$-invariant, clearly $\td\pi(\frakk_\frakl)\psi_0=0$. To obtain the whole $\frakk$-action on $\psi_0$ we have to apply elements of the form $(u,0,-\vartheta(u))\in\frakk$, $u\in V$, to $\psi_0$. By \eqref{eq:L2DerRep1} and \eqref{eq:L2DerRep3} we have
\begin{align}
 \td\pi_\lambda(u,0,-\vartheta(u))\psi(x) &= \frac{1}{\sqrt{-1}}\tau((\calB_\lambda-\vartheta(x))\psi(x),u) & \forall\,\psi\in C^\infty(\calO).\label{eq:KActionInTermsOfBesselOp}
\end{align}
Now we have to treat four cases separately. For simplicity we write $\td\pi$ for $\td\pi_{\lambda_1}$. Recall the operator $B_\alpha$ from \eqref{eq:OrdinaryBesselOp}.
\begin{enumerate}
 \item[\textup{(1)}] If $V$ is euclidean, then by Corollary~\ref{cor:BnuRadial}\,(1)
  \begin{align*}
   \td\pi(u,0,-\vartheta(u))\psi_0(x) &= \frac{1}{\sqrt{-1}}B_{\frac{\nu}{2}}\widetilde{K}_{\frac{\nu}{2}}(|x|)(x|u)+\frac{1}{\sqrt{-1}}\frac{d}{2}\widetilde{K}_{\frac{\nu}{2}}'(|x|)({\bf e}|u)
  \end{align*}
  Now, $B_{\frac{\nu}{2}}\widetilde{K}_{\frac{\nu}{2}}=0$. Further, since $\widetilde{K}_{-\frac{1}{2}}(|x|)=\frac{\sqrt{\pi}}{2}e^{-|x|}$ we have $\widetilde{K}_{\frac{\nu}{2}}'(|x|)=-\widetilde{K}_{\frac{\nu}{2}}(|x|)$ for $\nu=-1$. Together this gives
  \begin{align}
   \td\pi(u,0,-\vartheta(u))\psi_0(x) = \sqrt{-1}\frac{d}{2}({\bf e}|u)\psi_0(x).\label{eq:KactionEucl}
  \end{align}
  Hence, $W_0=\CC\psi_0$. Since $({\bf e}|u)=\tr(u)$ this gives the action of $\frakk$. Further, for $u=c_i$, $1\leq i\leq r_0$, we find that $W_0$ is of highest weight $\frac{d}{4}\sum_{i=1}^{r_0}{\gamma_i}$.
 \item[\textup{(2)}] If $V$ is non-euclidean of rank $r\geq3$, then $d=2d_0$ (see Proposition \ref{prop:ClassificationEuclSph}) and with Corollary~\ref{cor:BnuRadial}\,(2) we obtain
  \begin{align}
   \td\pi(u,0,-\vartheta(u))\psi_0(x) &= \frac{1}{\sqrt{-1}}B_{\frac{\nu}{2}}\widetilde{K}_{\frac{\nu}{2}}(|x|)(x|u).\label{eq:KactionNonEuclSph}
  \end{align}
  Again, $B_{\frac{\nu}{2}}\widetilde{K}_{\frac{\nu}{2}}=0$, which implies that $W_0=\CC\psi_0$ is the trivial representation.
 \item[\textup{(3)}] For $V=\RR^{p,q}$, $p,q\geq2$, we assume without loss of generality that $p\leq q$, the case $p\geq q$ is treated similarly. Denote by $(e_j)_{j=1,\ldots,n}$\index{ej@$e_j$} the standard basis of $V=\RR^n$. For $j=1,\ldots,p$ we define operators $(-)^\pm_j$ on $\calH^k(\RR^p)$ by
\begin{align*}
 (-)^+_j:\calH^k(\RR^p) &\rightarrow \calH^{k+1}(\RR^p), & \varphi^+_j(x) &:= x_j\varphi(x)-\frac{x_1^2+\cdots+x_p^2}{p+2k-2}\frac{\partial\varphi}{\partial x_j}(x),\\
 (-)^-_j:\calH^k(\RR^p) &\rightarrow \calH^{k-1}(\RR^p), & \varphi^-_j(x) &:= \frac{1}{p+2k-2}\frac{\partial\varphi}{\partial x_j}(x).
\end{align*}
For convenience we also put $(-)^+_j:=(-)^-_j:=0$ for $j=p+1,\ldots,n$. Using the operators $(-)^+_j$ and $(-)^-_j$ as well as Corollary~\ref{cor:BnuRadial}\,(3) we find that for $j=1,\ldots,n$ the action of $(e_j,0,-\vartheta(e_j))\in\frakk$ on $\widetilde{K}_{\frac{\nu}{2}+k}\otimes\varphi\in\widetilde{K}_{\frac{\nu}{2}+k}\otimes\calH^k(\RR^p)$ is given by
\begin{multline*}
 \td\pi(e_j,0,-\vartheta(e_j))(\widetilde{K}_{\frac{\nu}{2}+k}\otimes\varphi)\\
 = \frac{1}{\sqrt{-1}}\left[(2k+p-q)\widetilde{K}_{\frac{\nu}{2}+k+1}\otimes\varphi_j^+-(2k+p+q-4)\widetilde{K}_{\frac{\nu}{2}+k-1}\otimes\varphi_j^-\right].
\end{multline*}
Note that the coefficient $(2k+p-q)$ only vanishes for $k=\frac{q-p}{2}$ which is an integer if and only if $p+q$ is even. Now
\begin{align*}
 \frakk = \frakk_\frakl \oplus \{(u,0,-\vartheta(u)):u\in V\}
\end{align*}
and $\frakk_\frakl=\so(p)\oplus\so(q)$ acts irreducibly on $\calH^k(\RR^p)$ for every $k\geq0$. Therefore, \eqref{eq:MinKtypeRk2finite} and \eqref{eq:MinKtypeRk2infinite} follow.
\item[\textup{(4)}] For $V=\Sym(k,\RR)$ we note that $\psi_0^-=\frac{\sqrt{\pi}}{2}\varphi_{e_1}$ and hence it suffices to show that $\{\varphi_u:u\in\CC^k\}$ is an irreducible $\frakk$-module. Fix $u\in\CC^k$ and put $\widetilde{u}:=u\,{}^t\!u\in\Sym(k,\CC)$. Extending the trace form $(-|-)$ $\CC$-bilinearly to $\Sym(k,\CC)$ we can write $\varphi_u(x)=(x|\widetilde{u})^{\frac{1}{2}}\psi_0(x)$. We calculate the action of $\frakk$ on $\varphi_u$. First, using the product rule for the Bessel operator (see Proposition \ref{prop:BesselOpProperties} (2)) we obtain
\begin{align*}
 & (\calB_\lambda-x)\varphi_u(x)\\
 ={}& (x|\widetilde{u})^{\frac{1}{2}}\cdot(\calB_\lambda-x)\psi_0(x)-(x|\widetilde{u})^{-\frac{1}{2}}\psi_0(x)P(\overline{u},{\bf e})x+\calB_\lambda(x|\widetilde{u})^{\frac{1}{2}}\cdot\psi_0(x)\\
 ={}& -\frac{d}{2}\varphi_u(x){\bf e}-(x|\widetilde{u})^{-\frac{1}{2}}\psi_0(x)L(\widetilde{u})x+\calB_\lambda(x|\widetilde{u})^{\frac{1}{2}}\cdot\psi_0(x)
\end{align*}
by part (1). It is easy to see that for $x\in\calO$, $u\in\CC^k$ and $v\in V$ we have
\begin{align*}
 \calB_\lambda(x|\widetilde{u})^{\frac{1}{2}} &= 0 && \mbox{and} & (x|L(\widetilde{u})v) &= (x|\widetilde{u})^{\frac{1}{2}}(x|\widetilde{vu})^{\frac{1}{2}}
\end{align*}
and hence we find for $v\in V=\Sym(k,\RR)$
\begin{align*}
 \td\pi_\lambda^-(v,0,-\vartheta v)\varphi_u(x) &= \sqrt{-1}\frac{d}{2}\tr(v)\varphi_u(x)+\sqrt{-1}\varphi_{vu}.
\intertext{Similarly one shows that for $T\in\frakk_\frakl\cong\so(k)$ we have}
 \td\pi_\lambda^-(0,T,0)\varphi_u(x) &= \varphi_{Tu}(x).
\end{align*}
Together this shows that $\{\varphi_u:u\in\CC^k\}$ is an irreducible $\frakk$-module.\qedhere
\end{enumerate}
\end{proof}

\begin{remark}
The observation that $\psi_0$ is not $\frakk$-finite if $V=\RR^{p,q}$ with $p+q$ odd, $p,q\geq2$, reflects the fact that no covering group of $\SO(p+1,q+1)_0$ has a minimal representation if $p+q$ is odd and $p,q\geq3$ (see \cite[Theorem 2.13]{Vog81}). Nevertheless, for $\SO(p+1,3)_0$ there exists a minimal representation also if $p$ is odd, see Sabourin \cite{Sab96} for the $p=3$ case.  For these minimal representations, however, no $L^2$-model with explicit Lie algebra action is known.
\end{remark}

\begin{remark}
In the case $V=\Sym(k,\RR)$ the pullback of the section $\varphi_u\in L^2(\calO,\calL)$ under the folding map $\RR^k\setminus\{0\},\,x\mapsto x\,{}^t\!x$ is given by
\begin{align*}
 \calU^-\varphi_u(x) &= \varphi_u(x\,{}^t\!x) = (u_1x_1+\cdots+u_kx_k)e^{-|x|^2}, & x\in\RR^k,u\in\CC^k.
\end{align*}
In Subsection \ref{sec:ExMetRep} we shall see that the isomorphism $\calU^-:L^2(\calO,\calL)\to L^2_{\textup{odd}}(\RR^k)$ intertwines the $\frakg$-action $\td\pi^-$ on $C^\infty(\calO,\calL)$ and the differential action of the Weil representation on the (classical) Schr\"odinger model. Since the functions $\calU^-\varphi_u$, $u\in\CC^k$, form the minimal $\frakk$-type of the odd part of the Weil representation it is then clear that Propositions \ref{prop:L2ness}, \ref{prop:L2Unitary} and \ref{prop:Admissible} also hold for $W^-$ with the obvious formulation. However, they can be proved in the same fashion as for $W$.
\end{remark}

Next we turn to the split rank $1$ case. Recall that if $r_0=1$, then $V\cong\RR^{k,0}$ for some $k\geq1$.

\begin{theorem}\label{prop:Kfinite-split rank 1}
Let $V=\RR^{k,0}$, $k\geq1$, and $\lambda=\frac{r_0}{r}\sigma$ with $\sigma>0$. Then
\begin{align*}
 W_0 &= \CC\psi_0
\end{align*}
and $\psi_0$ is a $\frakk$-fixed vector. In particular, $W_0$ is an irreducible $\frakk$-module with highest weight $\alpha_0=0$.
\end{theorem}

\begin{proof}
Again we use \eqref{eq:KActionInTermsOfBesselOp}. With Corollary~\ref{cor:BnuRadial}\,(4) we obtain
\begin{align}
 \td\pi_\lambda(u,0,-\vartheta(u))\psi_0(x) &= \frac{1}{\sqrt{-1}}B_{\frac{\nu+\sigma}{2}}\widetilde{K}_{\frac{\nu+\sigma}{2}}(|x|)(x|u).\label{eq:KactionNonEuclSph-split rank 1}
\end{align}
Now, $B_{\frac{\nu+\sigma}{2}}\widetilde{K}_{\frac{\nu+\sigma}{2}}=0$ and hence
\begin{align*}
 \td\pi_\lambda(u,0,-\vartheta(u))\psi_0(x) &= 0.
\end{align*}
This implies the claim.
\end{proof}

In order to prove that $W$ is a $(\gf,\kf)$-module, it is sufficient
to show that the generator $\psi_0$ is a $\kf$-finite vector. For the
sake of completeness, we pin down this fact (cf.\ \cite{kinvent98}) as follows:

\begin{lemma}
Let $W$ be a $\gf$-module generated by a $\frakk$-finite vector $\psi_0$. Then $W=U(\frakg)\psi_0$ is a $(\frakg,\frakk)$-module.
\end{lemma}

\begin{proof}
Let $\frakg_1:=\frakg_\CC\oplus\CC\subseteq U(\frakg)$ and define $W_{n+1}:=\frakg_1W_n$ for $n\geq0$. We claim that
\begin{enumerate}
 \item[\textup{(1)}] $W_n$ is finite-dimensional for every $n$,
 \item[\textup{(2)}] $W_n$ is $\frakk$-invariant for every $n$,
 \item[\textup{(3)}] $W=\bigcup_n{W_n}$.
\end{enumerate}
The first statement follows easily by induction on $n$, since $W_0$ and $\frakg_1$ are finite-dimensional. The third statement is also clear by the definition of $U(\frakg)$. For the second statement we give a proof by induction on $n$:\\
For $n=0$ the statement is clear by the definition of $W_0$. For the induction step let $w\in W_{n+1}$ and $X\in\frakk$. Then $w=\sum_j{Y_jv_j}$ with $Y_j\in\frakg_1$ and $v_j\in W_n$. We have
\begin{align*}
 Xw &= \sum_j{X(Y_jv_j)} = \sum_j{\left([X,Y_j]v_j + Y_j(Xv_j)\right)}.
\end{align*}
Here $[X,Y_j]\in\frakg_1$ and hence $[X,Y_j]v_j\in W_{n+1}$ for each $j$. Furthermore $Xv_j\in W_n$ by the induction assumption and hence $Y_j(Xv_j)\in W_{n+1}$ for every $j$. Together this gives $Xw\in W_{n+1}$ which shows that $W_{n+1}$ is $\frakk$-invariant.\\
Now the $\frakk$-finiteness of every vector $w\in W$ follows.
\end{proof}

Now let us return to the general scalar case. To integrate $W$ to a unitary group representation on $L^2(\calO,\td\mu_\lambda)$ we need further properties. First, we analyze the functions in $W$ in more detail. For this we introduce some more notation. Denote by $\CC[\calO]$ the space of restrictions of polynomials on $V$ to $\calO$. Further, $\CC[\calO]_{\geq k}$ is defined as the space of those polynomials in $\CC[\calO]$ which are sums of homogeneous polynomials of degree $\geq k$. Finally, $\widetilde{K}_\alpha\otimes\CC[\calO]_{\geq k}$ is the space of functions
\begin{align*}
 \widetilde{K}_\alpha\otimes\varphi:\calO\rightarrow\CC,\,x\mapsto\widetilde{K}_\alpha(|x|)\varphi(x)
\end{align*}
with $\varphi\in\CC[\calO]_{\geq k}$.

\begin{proposition}\label{prop:L2ness}
Let $V$ be a simple Jordan algebra with simple $V^+$ and $W$ be one of the $(\frakg,\frakk)$-modules given in Theorem~\ref{prop:Kfinite} and Theorem~\ref{prop:Kfinite-split rank 1}.
\begin{enumerate}
\item[\textup{(a)}] If $V$ is of rank $r\geq3$, then
 \begin{align}
  W &\subseteq \bigoplus_{\ell=0}^\infty{\widetilde{K}_{\frac{\nu}{2}+\ell}\otimes\CC[\calO]_{\geq2\ell}} \subseteq L^2(\calO,\td\mu).
 \end{align}
\item[\textup{(b)}] If $V=\RR^{p,q}$ with $p+q$ even, $p,q\geq2$, then
 \begin{align}
  W &\subseteq \bigoplus_{\ell=0}^\infty{\bigoplus_{k=0}^{|\frac{p-q}{2}|}{\widetilde{K}_{\frac{\nu}{2}+k+\ell}\otimes\CC[\calO]_{\geq k+2\ell}}} \subseteq L^2(\calO,\td\mu).
 \end{align}
\item[\textup{(c)}] If $V=\RR^{k,0}$, $k\geq1$, and $\lambda=\frac{r_0}{r}\sigma>0$, then
 \begin{align}
  W &\subseteq \bigoplus_{\ell=0}^\infty{\widetilde{K}_{\frac{\nu+\sigma}{2}+\ell}\otimes\CC[\calO]_{\geq2\ell}}.
 \end{align}
Therefore, $W \subseteq L^2(\calO,\td\mu_\lambda)$ if and only if $\sigma\in(0,-2\nu)=(0,2k)$.
\end{enumerate}
\end{proposition}

\begin{proof}
Since $\frakg=\frakk+\frakq^{\max}$ we have $W=U(\frakq^{\max})W_0$ by the Poincar\'e--Birkhoff--Witt Theorem. Since in each case, $W_0$ is already contained in the direct sum above, it remains to show that these direct sums are stable under the action of $\frakq^{\max}=\frakl+\frakn$ to obtain the first inclusions. Clearly they are stable under the $\frakn$-action which is given by multiplication with polynomials. For the $\frakl$-action the formula $\frac{\td}{\td t}\widetilde{K}_\alpha(t)=-\frac{t}{2}\widetilde{K}_{\alpha+1}(t)$ gives the claim.\\
To show the second inclusions, we use the integral formula \eqref{eq:dmuIntFormula}. A function $\widetilde{K}_\alpha(|x|)\phi(x)$ with $\phi$ homogeneous of degree $\beta$ is contained in $L^2(\calO,\td\mu_\lambda)$ if and only if the function $\widetilde{K}_\alpha(t)t^\beta$ is contained in $L^2(\RR_+,t^{\lambda r-1}\td t)$. Together with the asymptotic behavior of the $K$-Bessel function at $t=0$ and $t=\infty$ this gives the claim. For the convenience of the reader we do the calculation for the split rank one case in Lemma~\ref{lem:Bessel} below.
\end{proof}

\begin{lemma}\label{lem:Bessel}
Suppose that $r_0=1$. Then $W\subseteq L^2(\calO,\td\mu_\lambda)$ if and only if $\sigma=r\lambda\in(0,-2\nu)$.
\end{lemma}

\begin{proof}
For $r_0=1$ we have $\nu=-k$ and $\sigma=r\lambda$. Since the $K$-Bessel functions rapidly decrease as $t\to\infty$, only the asymptotic behavior of $\widetilde{K}_\alpha(t)$ at $t=0$ is relevant. It is given by
$$\widetilde{K}_\alpha(t)=\begin{cases}
\frac{\Gamma(\alpha)}{2}\left(\frac{t}{2}\right)^{-2\alpha}+ o(t^{-2\alpha})&\text{ if } \alpha>0,\\
-\log\left(\frac{t}{2}\right)+o\left(\log\left(\frac{t}{2}\right)\right)&\text{ if } \alpha=0,\\
\frac{\Gamma(-\alpha)}{2}+ o(1)&\text{ if } \alpha<0.
\end{cases}
$$
Here $o(-)$ denotes the Bachmann--Landau symbol. We first show that the condition $\sigma\in(0,-2\nu)$ is sufficient for $W$ to be contained in $L^2(\calO,\td\mu_\lambda)$. In view of the inclusion in Proposition~\ref{prop:L2ness}\,(c) it suffices to show that $\widetilde{K}_{\frac{\nu+\sigma}{2}+\ell}(t)t^{2\ell+m}\in L^2(\RR_+,t^{\sigma-1}\td t)$ for all $\ell,m\in\NN$. We distinguish three cases.
\begin{enumerate}
\item[\textup{(a)}] $\frac{\nu+\sigma}{2}+\ell>0$. In this case we have as $t\to0$:
\begin{align*}
 |\widetilde{K}_{\frac{\nu+\sigma}{2}+\ell}(t)t^{2\ell+m}|^2t^{\sigma-1}\sim t^{-2\nu-\sigma+2m-1}
\end{align*}
which is integrable near $t=0$ since $\sigma<-2\nu$.
\item[\textup{(b)}] $\frac{\nu+\sigma}{2}+\ell=0$. The asymptotic behavior as $t\to0$ is given by:
\begin{align*}
 |\widetilde{K}_{\frac{\nu+\sigma}{2}+\ell}(t)t^{2\ell+m}|^2t^{\sigma-1}\sim \log(t)^2t^{\sigma+4\ell+2m-1}
\end{align*}
which is integrable near $t=0$ since $\sigma>0$.
\item[\textup{(c)}] $\frac{\nu+\sigma}{2}+\ell<0$. As $t\to0$ we have:
\begin{align*}
 |\widetilde{K}_{\frac{\nu+\sigma}{2}+\ell}(t)t^{2\ell+m}|^2t^{\sigma-1}\sim t^{\sigma+4\ell+2m-1}
\end{align*}
which is integrable near $t=0$ since $\sigma>0$.
\end{enumerate}
Now we show that the condition $\sigma\in(0,-2\nu)$ is also necessary for $W$ to be contained in $L^2(\calO,\td\mu_\lambda)$. For this it suffices to assume that the function $\psi_0$ is contained in $L^2(\calO,\td\mu_\lambda)$. This implies that $\widetilde{K}_{\frac{\nu+\sigma}{2}}\in L^2(\RR_+,t^{\sigma-1}\td t)$. Again we distinguish between three cases.
\begin{enumerate}
\item[\textup{(a)}] $\frac{\nu+\sigma}{2}>0$. Then automatically $\sigma>-\nu>0$. Further, as $t\to0$ we have
\begin{align*}
 |\widetilde{K}_{\frac{\nu+\sigma}{2}}(t)|^2t^{\sigma-1}\sim t^{-2\nu-\sigma-1}
\end{align*}
which is integrable near $t=0$ if and only if $\sigma<-2\nu$. Therefore $\sigma\in(0,-2\nu)$.
\item[\textup{(b)}] $\frac{\nu+\sigma}{2}=0$. In this case $\sigma=-\nu\in(0,-2\nu)$.
\item[\textup{(c)}] $\frac{\nu+\sigma}{2}<0$. This implies that $\sigma<-\nu<-2\nu$. Further, as $t\to0$ we have
\begin{align*}
 |\widetilde{K}_{\frac{\nu+\sigma}{2}}(t)|^2t^{\sigma-1}\sim t^{\sigma-1}
\end{align*}
which is integrable near $t=0$ if and only if $\sigma>0$. Hence, also in this case $\sigma\in(0,-2\nu)$ and the proof is complete.\qedhere
\end{enumerate}
\end{proof}

Now we can prove the necessary properties to integrate $W$ to a group representation. First, we show that the $(\frakg,\frakk)$-module $W$ is infinitesimally unitary.

\begin{proposition}\label{prop:L2Unitary}
Let $V$ be a simple Jordan algebra with simple $V^+$ and assume that $V\ncong\RR^{p,q}$, $p,q\geq2$, with $p+q$ odd. If $r_0=1$, further assume that $\sigma=r\lambda\in(0,-2\nu)$. Then the following properties hold:
\begin{enumerate}
\item[\textup{(1)}] $W$ is contained in $L^2(\calO,\td\mu_\lambda)\cap C^\infty(\calO)$.
\item[\textup{(2)}] The action $\td\pi_\lambda$ of $\frakg$ on $W$ is infinitesimally unitary with respect to the $L^2$-inner product.
\end{enumerate}
\end{proposition}

\begin{proof}
\begin{enumerate}
\item[\textup{(1)}] This is Proposition \ref{prop:L2ness}.
\item[\textup{(2)}] For the action of $\frakn$ and $\frakl$ this is clear as the action of $\frakq^{\max}=\frakl+\frakn$ is the derived action of the unitary representation $\rho_\lambda$. The action of $\overline{\frakn}$ is infinitesimally unitary since it is given as the multiple of the Bessel operator $\calB_\lambda$ by $\sqrt{-1}$ (see \eqref{eq:L2DerRep3}) which is symmetric by Theorem \ref{thm:BlambdaTangential}.\qedhere
\end{enumerate}
\end{proof}

Using the previous proposition we can now prove that $W$ is admissible.

\begin{proposition}\label{prop:Admissible}
Let $V$ be a simple Jordan algebra with simple $V^+$ and assume that $V\ncong\RR^{p,q}$, $p,q\geq2$, with $p+q$ odd. If $r_0=1$, further assume that $\sigma=r\lambda\in(0,-2\nu)$. Then
\begin{enumerate}
\item[\textup{(1)}] The $\frakg$-module $W$ is $Z(\frakg)$-finite.
\item[\textup{(2)}] $W$ is an admissible $(\frakg,\frakk)$-module.
\end{enumerate}
\end{proposition}

\begin{proof}
Any finitely generated $(\frakg,\frakk)$-module is admissible if it is $Z(\frakg)$-finite (see \cite[Corollary 3.4.7]{Wal88}). Therefore, part (2) follows from part (1).\\
To show (1) note that the representation $\rho_\lambda$ of $Q=L\ltimes N$ on $\calH=L^2(\calO,\td\mu_\lambda)$ is unitary and irreducible by Proposition \ref{prop:MackeyRep}. By Proposition~\ref{prop:L2Unitary}\,(1) the space $D:=\rho_\lambda(Q)W$ is contained in $\calH$ and since $\rho_\lambda$ is irreducible it is also dense in $\calH$. Further note that since $W\subseteq C^\infty(\calO)$ by Proposition~\ref{prop:L2Unitary}\,(1), we also have $D\subseteq C^\infty(\calO)$ since the action of $Q$ leaves $C^\infty(\calO)$ invariant (see \eqref{eq:L2Rep1} and \eqref{eq:L2Rep2}). Now let $X\in Z(\frakg)$ be any central element and put $T:=\td\pi_\lambda(X)$. Then $TW\subseteq W$ because $W$ is a $\frakg$-module. Further, $T$ extends to $C^\infty(\calO)$ as it acts as a differential operator. Then the compatibility property \eqref{eq:CompatibilityRhoPi} implies
\begin{align*}
 T(D) &= T(\rho_\lambda(Q)W) = \rho_\lambda(Q)TW \subseteq \rho_\lambda(Q)W=D
\end{align*}
The same argument applies for the formal adjoint $S=T^*$ of $T$ which is also given by the Lie algebra action since the Lie algebra representation is infinitesimally unitary by Proposition~\ref{prop:L2Unitary}\,(2). On $D$ the compatibility property \eqref{eq:CompatibilityRhoPi} assures that $T$ commutes with every $\rho_\lambda(g)$, $g\in Q$. Now, finally, by a variant of Schur's Lemma (see \cite[Proposition 1.2.2]{Wal88}), $T$ acts on $D$ as a scalar multiple of the identity. Therefore, $W$ is $Z(\frakg)$-finite and the proof is complete.
\end{proof}

\subsubsection{Integration of the $(\frakg,\frakk)$-module}\label{sec:IntgkModule}

Now we can finally integrate the $(\frakg,\frakk)$-module $W$ to a unitary representation of a finite cover of $G$. We first find the minimal cover of $G$ to which the $(\frakg,\frakk)$-module integrates. We use the following classical fact:

\begin{lemma}\label{lem:LiftLemma}
Let ${G}^\vee$ be any connected reductive Lie group with Lie algebra $\frakg$ and denote by ${K}^\vee$ the maximal compact subgroup (modulo the center of $G^\vee$) having the Lie algebra $\frakk$. Then an admissible $(\frakg,\frakk)$-module which is generated by a single $\frakk$-type lifts to a representation of ${G}^\vee$ if and only if the generating $\frakk$-type lifts to ${K}^\vee$.
\end{lemma}

Thus, we only have to deal with the $\frakk$-type $W_0$. In view of Theorems \ref{prop:Kfinite} and \ref{prop:Kfinite-split rank 1} we define a finite covering ${G}^\vee\rightarrow G$ as follows:

\begin{definition}[Minimal covering group $G^\vee$]\label{def:minG}
\begin{enumerate}
\item[\textup{(1)}] For euclidean $V$ we treat the five cases separately:
\begin{enumerate}
\item[\textup{(a)}] $\frakg=\sp(k,\RR)$. The metaplectic group $\Mp(k,\RR)$ is a $4$-fold cover of $G=\Sp(k,\RR)/\{\pm\1\}$. On the level of $K=U(k)/\{\pm\1\}$ this $4$-fold cover is given by $U(k)^{(2)}=\{(g,z)\in U(k)\times\CC^\times:z^2=\det(g)\}$ and hence the fiber over the identity is
\begin{align*}
 \{(\1,1),(\1,-1),(-\1,e^{\sqrt{-1}\pi\frac{k}{2}}),(-\1,-e^{\sqrt{-1}\pi\frac{k}{2}})\}
\end{align*}
which is $\ZZ_4$ for $k$ odd and $\ZZ_2\times\ZZ_2$ for $k$ even. In particular, $(-\1,1)$ is in the fiber if and only if $k$ is even. Define
\begin{align*}
 G^\vee := G^\vee_+ &:= \begin{cases}\Mp(k,\RR)/\{(\1,1),(-\1,1)\} & \mbox{for $k$ even,}\\\Mp(k,\RR) & \mbox{for $k$ odd,}\end{cases}\\
 G^\vee_- &:= \begin{cases}\Mp(k,\RR)/\{(\1,1),(-\1,-1)\} & \mbox{for $k$ even,}\\\Mp(k,\RR) & \mbox{for $k$ odd.}\end{cases}
\end{align*}
Note that for both even and odd $k$ the groups $G^\vee_+$ and $G^\vee_-$ are not linear.
\item[\textup{(b)}] $\frakg=\su(k,k)$. We realize $\SU(k,k)$ as
\begin{align*}
 \SU(k,k) &= \left\{g\in\SL(2k,\CC):g^{-1}=\left(\begin{array}{cc}0&\1\\\1&0\end{array}\right)g^*\left(\begin{array}{cc}0&\1\\\1&0\end{array}\right)\right\},
\end{align*}
where $g^*$ denotes the conjugate transpose matrix. Then the center of $\SU(k,k)$ is given by $\{e^{\sqrt{-1}\pi\frac{j}{k}}\1_{2k}:j=0,\ldots,2k-1\}$. Define
\begin{align*}
 G^\vee &:= \SU(k,k)/\{e^{\sqrt{-1}\pi\frac{2j}{k}}\1_{2k}:j=0,\ldots,k-1\}.
\end{align*}
\item[\textup{(c)}] $\frakg=\so^*(4k)$. We realize $\SO^*(4k)$ as
\begin{align*}
 \SO^*(4k) &= \left\{g\in\SL(4k,\CC):g^{-1}=J_{2k}g^*J_{2k}=\diag(J_k,J_k)\overline{g}\diag(J_k,J_k)\right\},
\end{align*}
where $\overline{g}$ denotes the conjugate matrix and $J_m\in M(2m\times2m,\RR)$ is given by
\begin{align*}
 J_m &:= \left(\begin{array}{cc}0&\1_m\\-\1_m&0\end{array}\right).
\end{align*}
Then the center of $\SO^*(4k)$ is given by
\begin{align*}
 Z(\SO^*(4k)) &= \{e^{\sqrt{-1}\pi\frac{j}{2k}}\1_{4k}:j=0,\ldots,4k-1\}.
\end{align*}
Define
\begin{align*}
 G^\vee &:= G = \SO^*(4k)/Z(\SO^*(4k)).
\end{align*}
\item[\textup{(d)}] $\frakg=\so(2,k)$. Let $\SO(2)^{(2)}$ be the double covering group of $\SO(2)$ and denote by $\eta\in\SO(2)^{(2)}$ the unique element of order $2$. Then there is a unique double cover $\SO(2,k)_0^{(2)}$ of $\SO(2,k)_0$ with maximal compact subgroup $\SO(2)^{(2)}\times\SO(k)$ such that the kernel of the covering map $\SO(2,k)_0^{(2)}\to\SO(2,k)_0$ is given by $\{(\1,\1),(\eta,\1)\}$. Define
\begin{align*}
 G^\vee &:= \begin{cases}\SO(2,k)_0/\{\pm\1\} & \mbox{for $k\in4\ZZ+2$,}\\\SO(2,k)_0 & \mbox{for $k\in4\ZZ$,}\\\SO(2,k)_0^{(2)} & \mbox{for $k$ odd.}\end{cases}
\end{align*}
\item[\textup{(5)}] $\frakg=\mathfrak{e}_{7(-25)}$. In this case we put
\begin{align*}
 G^\vee &:= G = E_{7(-25)}/Z(E_{7(-25)}).
\end{align*}
\end{enumerate}
\item[\textup{(2)}] For $V$ non-euclidean of rank $\geq3$ we let ${G}^\vee:=G$.
\item[\textup{(3)}] Now let $V=\RR^{p,q}$ with $p+q$ even, $p,q\geq2$.
\begin{enumerate}
\item[\textup{(a)}] If $p$ and $q$ are both even, then define
\begin{align*}
 {G}^\vee &:= G = \SO(p+1,q+1)_0.
\end{align*}
\item[\textup{(b)}] If $p$ and $q$ are both odd, we have $G=\SO(p+1,q+1)_0/\{\pm\1\}$. In this case we put
\begin{align*}
 {G}^\vee := \begin{cases}\SO(p+1,q+1)_0/\{\pm\1\} & \mbox{if $p-q\equiv0\ (\mod\ 4)$,}\\\SO(p+1,q+1)_0 & \mbox{if $p-q\equiv2\ (\mod\ 4)$.}\end{cases}
\end{align*}
\end{enumerate}
\item[\textup{(4)}] Finally, for $V=\RR^{k,0}$, $k\geq1$, we also put ${G}^\vee:=G= \SO(k+1,1)_0$.
\end{enumerate}
\end{definition}

\begin{theorem}\label{thm:IntgkModule}
Let $V$ be a simple Jordan algebra with simple euclidean subalgebra $V^+$, and 
assume that $V\ncong\RR^{p,q}$ with $p+q$ odd, $p,q\geq2$. If $r_0=1$, further assume that $\sigma=r\lambda\in(0,-2\nu)$. Then the $(\frakg,\frakk)$-module $W$ lifts to an irreducible unitary representation $\pi$\index{pi@$\pi$} of ${G}^\vee$ on $L^2(\calO,\td\mu_\lambda)$. Moreover, ${G}^\vee$ is the minimal covering of $G$ to which $W$ lifts.\\
For $V=\Sym(k,\RR)$ the $(\frakg,\frakk)$-module $W^-$ lifts to an irreducible unitary representation $\pi^-$ of ${G}^\vee_-$ on $L^2(\calO,\calL)$ and ${G}^\vee_-$ is the minimal covering of $G$ to which $W^-$ lifts.
\end{theorem}

We remark that $r_0=1$ is equivalent to $\frakg\cong\so(k+1,1)$, $k\geq1$. In this case $\nu=-k$ and $\sigma\in(0,2k)$ parameterizes the spherical complementary series representations of $\SO(k+1,1)$.

\begin{proof}[Proof of Theorem \ref{thm:IntgkModule}]
Denote by ${K}^\vee\subseteq{G}^\vee$ the maximal compact subgroup with Lie algebra $\frakk$. We only have to show that the $\frakk$-module $W_0$ lifts to a ${K}^\vee$-module and that the covering $G^\vee$ is minimal with this property. Then the $(\frakg,\frakk)$-module $W$ lifts to a $(\frakg,{K}^\vee)$-module. By Propositions \ref{prop:L2Unitary} and \ref{prop:Admissible} this $(\frakg,{K}^\vee)$-module is admissible, contained in $L^2(\calO,\td\mu_\lambda)$ and infinitesimally unitary with respect to the $L^2$ inner product. Hence, it integrates to a unitary representation $\pi$ of ${G}^\vee$ on a Hilbert space $\calH\subseteq L^2(\calO,\td\mu_\lambda)$. Since the Lie algebra actions of $\pi$ and $\rho_\lambda$ agree on the maximal parabolic subalgebra $\frakq^{\max}$, the representation $\pi$ descends to the group $Q$ on which it agrees with $\rho_\lambda$. But $\rho_\lambda$ is irreducible on $L^2(\calO,\td\mu_\lambda)$ and therefore, $\pi$ has to be irreducible and $\calH=L^2(\calO,\td\mu_\lambda)$.\\
It remains to show that ${G}^\vee$ is the minimal covering of $G$ to which $W$ integrates. By Lemma \ref{lem:LiftLemma} we only have to check that ${G}^\vee$ is minimal among the coverings of $G$ with the property that $W_0$ integrates to ${K}^\vee$.
\begin{enumerate}
\item[\textup{(1)}] For euclidean $V$ we note by Theorem~\ref{prop:Kfinite}\,(a) that $\frakk$ acts on $\psi_0$ by the character
\begin{align*}
 \td\xi:\frakk\to\CC,\,(u,D,-u)\mapsto\frac{d}{2}\sqrt{-1}\tr(u).
\end{align*}
We check the five cases separately:
\begin{enumerate}
\item[\textup{(a)}] $\frakg=\sp(k,\RR)$. The map $\frakk\to\fraku(k),\,(u,D,-u)\mapsto D+\sqrt{-1}u$ is an isomorphism. Under this isomorphism the character $\td\xi$ is given by $\fraku(k)\to\CC,\,X\mapsto\frac{1}{2}\Tr(X)$ (we have $d=1$). Therefore it integrates to the character $U(k)^{(2)}\to\CC^\times,\,(g,z)\mapsto z$. This character is only trivial for the elements $(g,1)$ and the claim follows by Definition~\ref{def:minG}\,(1)\,(a).
\item[\textup{(b)}] $\frakg=\su(k,k)$. The map $\frakk\to\fraks(\fraku(k)\oplus\fraku(k)),\,(u,D,-u)\mapsto(D+\sqrt{-1}u,D-\sqrt{-1}u)$ is an isomorphism. Under this isomorphism the character $\td\xi$ is given by $\fraks(\fraku(k)\oplus\fraku(k))\to\CC,\,(X,Y)\mapsto\frac{1}{2}(\Tr(X)-\Tr(Y))=\Tr(X)$ (we have $d=2$) and hence integrates to the character $S(U(k)\times U(k))\to\CC^\times,\,(g,h)\mapsto\Det(g)$. The central element $\diag(e^{\sqrt{-1}\pi\frac{j}{k}})\in\SU(k,k)$, $j=0,\ldots,2k-1$, corresponds to the element $\diag(e^{\sqrt{-1}\pi\frac{j}{k}})\in S(U(k)\times U(k))$ and the claim follows with Definition~\ref{def:minG}\,(1)\,(b).
\item[\textup{(c)}] $\frakg=\so^*(4k)$. The map $\frakk\to\fraku(2k),\,(u,D,-u)\mapsto D+\sqrt{-1}u$ is an isomorphism. Under this isomorphism the character $\td\xi$ corresponds to the character $\fraku(2k)\to\CC,\,X\mapsto2\Tr(X)$ which integrates to the character $U(2k)\to\CC^\times,\,g\mapsto\Det(g)^2$. The central element $\diag(e^{\sqrt{-1}\pi\frac{j}{2k}})\in\SO^*(4k)$, $j=0,\ldots,4k-1$, corresponds to the element $\diag(e^{\sqrt{-1}\pi\frac{j}{2k}})\in U(2k)$. Since $\Det(\diag(e^{\sqrt{-1}\pi\frac{j}{2k}}))^2=1$ all central elements act trivially and the claim follows with Definition~\ref{def:minG}\,(1)\,(c).
\item[\textup{(d)}] $\frakg=\so(2,k)$. Under the isomorphism of Example~\eqref{ex:ConfGrp}\,(2) the character $\td\xi$ corresponds to the character
\begin{align*}
 \so(2)\oplus\so(k)\to\CC,\,\left(\left(\begin{array}{cc}0&t\\-t&0\end{array}\right),X\right)\mapsto-\tfrac{k-2}{2}\sqrt{-1}t.
\end{align*}
This character integrates to $\SO(2)^{(2)}\times\SO(k)$ and factors to $\SO(2)\times\SO(k)$ if and only if $k$ is even and further to $(\SO(2)\times\SO(k))/\{\pm1\}$ if and only if $k\in4\ZZ+2$. This gives the claim by Definition~\ref{def:minG}\,(1)\,(d).
\item[\textup{(e)}] $\frakg=\mathfrak{e}_{7(-25)}$. The maximal compact subgroup $\widetilde{K}\subseteq\widetilde{G}$ is isomorphic to $\widetilde{E_6}\times\RR_+$. The center of $\widetilde{G}$ is isomorphic to $\ZZ$ (see \cite[page 48]{Tit67}) and under the isomorphism $\widetilde{K}\cong\widetilde{E_6}\times\RR_+$ a generator is given by $(z_1,z_2)$ with $z_1\in Z(\widetilde{E_6})$ non-trivial and $z_2\in\RR_+$ (see \cite[pages 46 \& 48]{Tit67}). Since $Z(\widetilde{E_6})=\ZZ_3$ (see \cite[page 46]{Tit67}) we also have $(\1,z_2^3)=(z_1,z_2)^3\in Z(\widetilde{G})$. The element $z_2$ can be written as $z_2=\exp(t({\bf e},0,-{\bf e}))$ since $Z(\frakk)=\RR({\bf e},0,-{\bf e})$ by Lemma \ref{lem:CenterK}. Since $z_2^3\in Z(\widetilde{G})$ we must have $\spec(\ad(3t({\bf e},0,-{\bf e})))\subseteq2\pi\sqrt{-1}\ZZ$ which yields $3t\in\pi\ZZ$. Now the character $\xi$ integrates to $\widetilde{K}$ and is on $(z_1,z_2)$ given by
\begin{align*}
 \xi(z_1,z_2) &= e^{\td\xi(t({\bf e},0,-{\bf e}))} = e^{\frac{d}{2}\sqrt{-1}\tr({\bf e})t} = 1
\end{align*}
since $d=8$, $\tr({\bf e})=r=3$ and $t\in\pi\ZZ$. Hence the character factors to $\widetilde{K}/Z(\widetilde{G})$ which obviously gives the minimal covering $G^\vee=G$ of $G$.
\end{enumerate}
\item[\textup{(2)}] In the case where $V$ is non-euclidean of rank $\geq3$, the vector $\psi_0$ is $\frakk$-spherical and hence $W_0$ integrates to $K^\vee=K$. We further have ${G}^\vee=G$ and hence ${G}^\vee$ is automatically the minimal covering of $G$.
\item[\textup{(3)}] Let $V=\RR^{p,q}$ with $p+q$ even, $p,q\geq2$. Then ${G}^\vee=G$ in all cases except when $p$ and $q$ are both odd and $p-q\equiv2\ (\mod\ 4)$. In this case $Z({G}^\vee)=\{\pm\1\}$. By Theorem \ref{prop:Kfinite} the minimal $\frakk$-type $W_0$ is isomorphic to $\calH^{|\frac{p-q}{2}|}(\RR^{\min(p,q)+1})$ and hence integrates to $\SO(p+1)\times\SO(q+1)$. Further the element $-\1$ acts on $W_0$ by $(-1)^{|\frac{p-q}{2}|}$ and hence we can factor out $-\1$ if and only if $p-q\equiv2\ (\mod\ 4)$.
\item[\textup{(4)}] The case of $V=\RR^{k,0}$, $k\geq1$, is similar to case (2).
\item[\textup{(5)}] For $V=\Sym(k,\RR)$ we consider the $\frakk$-module $W_0^-$. As in (1)~(a) we see, using the proof of Theorem \ref{prop:Kfinite} (e), that the action of $\frakk\cong\fraku(k)$ on $W_0^-\cong\CC^k$ integrates to the representation $U(k)^{(2)}\to\GL(k,\CC),\,(g,z)\mapsto zg$ and the claim follows.
\end{enumerate}
Therefore the proof is complete.\qedhere
\end{proof}

\begin{corollary}\label{cor:annihilator of dpi}
Assume that $V\ncong\RR^{p,q}$ with $p+q$ odd, $p,q\geq2$. If $r_0=1$, further assume that $\sigma=r\lambda\in(0,-2\nu)$. Then the kernels of $\td\pi$ (resp. $\td\pi^-$) and its restriction to $\frakk$-finite vectors agree.
\end{corollary}

\begin{proof}
The space of $\frakk$-finite vectors is dense in $L^2(\calO,\td\mu_\lambda)$ by Theorem~\ref{thm:IntgkModule}, and consequently also in the space of smooth vectors of $\pi$. If $\td\pi(X)$, $X\in U(\frakg)$, annihilates all $\frakk$-finite vectors it also annihilates all smooth vectors, which implies the claim. The same argument goes through for $\td\pi^-$.
\end{proof}

\begin{corollary}\label{cor:UniqueMinRep}
Let $\pi$ (resp. $\pi^-$) be the irreducible unitary representation of $G^\vee$ (resp. $G^\vee_-$) on $L^2(\calO,\td\mu_\lambda)$ (resp. $L^2(\calO,\calL)$) constructed from a simple Jordan algebra in Theorem \ref{thm:IntgkModule}.
Assume that $V\ncong\RR^{p,q}$ with $p+q$ odd, $p,q\geq2$, and that $V$ is split or complex. Assume further that $\frakg_\CC$ is not of type $A_n$. Then the representations $\pi$ and $\pi^-$ are minimal in the sense of Definition \ref{def:minrep}. Conversely, all minimal representations of any covering group of $G$ are equivalent to $\pi$ or its dual or additionally $\pi^-$ or its dual for $V\cong\Sym(k,\RR)$.
\end{corollary}

\begin{proof}
Combining Theorem~\ref{thm:min-split-nilpotent-orbits} and Theorem \ref{thm:Minimality} we see that the underlying $(\frakg,\frakk)$-module $W$ of $\pi$ is a minimal representation, since our hypotheses guarantee that the annihilator ideal is the Joseph ideal. Now, the group representation $\pi$ is minimal by definition. That in fact all minimal representations are obtained in this way follows by comparing the tables in \cite{Tor97} with Table~\ref{tb:Groups}.
\end{proof}

\subsection{Two prominent examples}\label{sec:ExReps}

We show that for $V=\Sym(k,\RR)$ the representations $\pi^\pm$ of ${G}^\vee$ are isomorphic to the even and odd part of the Segal--Shale--Weil representation (see \cite[Chapter 4]{Fol89}) and for $V=\RR^{p,q}$ the representation $\pi$ is isomorphic to the minimal representation of $\upO(p+1,q+1)$ as studied by T. Kobayashi, B. {\O}rsted and G. Mano in \cite{KM07b,KM07a,KO03a,KO03b,KO03c}.

\subsubsection{The Segal--Shale--Weil representation}\label{sec:ExMetRep}

The \textit{Segal--Shale--Weil representation} is a unitary representation of the metaplectic group $\Mp(k,\RR)$, the double cover of the symplectic group $\Sp(k,\RR)$. We compare the (classical) Schr\"odinger model of $\mu$\index{mu@$\mu$} realized on $L^2(\RR^k)$ with our construction of the minimal representation associated to the Jordan algebra $V=\Sym(n,\RR)$ via the folding map \eqref{eq:SymTwoFoldCovering}. For this purpose, it is enough to work with the action $\td\mu$\index{dmu@$\td\mu$} of the Lie algebra $\sp(n,\RR)$, and we shall use the same normalization as in \cite[Chapter 4]{Fol89}:
\begin{align*}
 \td\mu\left(\begin{array}{cc}0&0\\C&0\end{array}\right) &= -\pi\sqrt{-1}\sum_{i,j=1}^k{C_{ij}y_iy_j} && \mbox{for $C\in\Sym(k,\RR)$,}\\
 \td\mu\left(\begin{array}{cc}A&0\\0&-^t\!A\end{array}\right) &= -\sum_{i,j=1}^k{A_{ij}y_j\frac{\partial}{\partial y_i}}-\frac{1}{2}\Tr(A) && \mbox{for $A\in M(k,\RR)$,}\\
 \td\mu\left(\begin{array}{cc}0&B\\0&0\end{array}\right) &= \frac{1}{4\pi\sqrt{-1}}\sum_{i,j=1}^k{B_{ij}\frac{\partial^2}{\partial y_i\partial y_j}} && \mbox{for $B\in\Sym(k,\RR)$.}
\end{align*}

The Weil representation splits into two irreducible components (see \cite[Theorem 4.56]{Fol89}):
\begin{align*}
 L^2(\RR^k) &= L^2_{\textup{even}}(\RR^k) \oplus L^2_{\textup{odd}}(\RR^k),
\end{align*}
where $L^2_{\textup{even}}(\RR^k)$ and $L^2_{\textup{odd}}(\RR^k)$ denote the spaces of even and odd $L^2$-functions, respectively. Let $\mu=\mu^+\oplus\mu^-$ be the corresponding decomposition of the representation $\mu$.

Next we recall from \eqref{eq:DefCalU} and \eqref{eq:DefCalU-} that the folding map $\RR^k\setminus\{0\},\,x\mapsto x\,{}^t\!x$ induces unitary isomorphisms
\begin{align*}
 &\calU^+:L^2(\calO)\to L^2_{\textup{even}}(\RR^k),\\
 &\calU^-:L^2(\calO,\calL)\to L^2_{\textup{odd}}(\RR^k).
\end{align*}
Third, for the Jordan algebra $V=\Sym(k,\RR)$, the conformal Lie algebra $\frakg\cong\sp(k,\RR)$ acts via $\td\pi^+$ resp. $\td\pi^-$ on the space of smooth vectors for $L^2(\calO)$ resp. $L^2(\calO,\calL)$ by skew-adjoint operators. Now we realize $\frakg\cong\sp(k,\RR)$ in the matrix form as in Example~\ref{ex:ConfGrp}\,(1) and define an automorphism of $\sp(k,\RR)$ by
\begin{align*}
 \left(\begin{array}{cc}A&C\\B&-^t\!A\end{array}\right)\mapsto&\  k_0\left(\begin{array}{cc}A&C\\B&-^t\!A\end{array}\right) k_0^{-1}%
 =\left(\begin{array}{cc}-^t\!A&-\pi B\\-\frac{1}{\pi}C&A\end{array}\right),
\end{align*}
where $k_0=\begin{pmatrix}0&\sqrt{\pi}{\mathbf 1}\\-\frac{1}{\sqrt{\pi}}\mathbf 1&0\end{pmatrix}$. We show that under these identifications the representation $\td\pi^+$ resp. $\td\pi^-$ agrees with $\td\mu^+$ resp. $\td\mu^-$. More precisely, we have the following identity of skew-adjoint operators on $L^2(\calO)$ resp. $L^2(\calO,\calL)$:

\begin{proposition}
For $A\in M(k,\RR)$ and $B,C\in\Sym(k,\RR)$ we have
\begin{align}
 \td\pi^\pm(C,A,B) &= (\calU^\pm)^{-1}\circ\td\mu^\pm\left(\begin{array}{cc}-^t\!A&-\pi B\\-\frac{1}{\pi}C&A\end{array}\right)\circ\calU^\pm.\label{eq:MetaIntertwiningFormulaDer}
\end{align}
\end{proposition}

\begin{proof}
It suffices to prove the intertwining formula for $\td\pi=\td\pi^+$, $\td\mu=\td\mu^+$ and $\calU=\calU^+$. Choose an orthonormal basis $(e_\alpha)_\alpha$ of $V=\Sym(k,\RR)$ with respect to the inner product $(x|y)=\Tr(xy)$. Then for $1\leq i\leq k$:
\begin{align*}
 \frac{\partial\calU\psi}{\partial y_i}(y) &= \frac{\partial}{\partial y_i}\psi(yy^t) = \sum_\alpha{\frac{\partial\psi}{\partial x_\alpha}(yy^t)\frac{\partial(yy^t)_\alpha}{\partial y_i}} = \sum_\alpha{\frac{\partial\psi}{\partial x_\alpha}(yy^t)\frac{\partial}{\partial y_i}\Tr(yy^te_\alpha)}\\
 &= 2\sum_\alpha{\frac{\partial\psi}{\partial x_\alpha}(yy^t)(e_\alpha y)_i} = 2\left(\frac{\partial\psi}{\partial x}(yy^t)y\right)_i.
\end{align*}
\begin{enumerate}
\item[\textup{(a)}] Let $(C,0,0)\in\frakg$, $C\in\Sym(k,\RR)$. Then
 \begin{align*}
  & \left(\td\mu\left(\begin{array}{cc}0&0\\-\frac{1}{\pi}C&0\end{array}\right)\circ\calU\right)\psi(y)\\
  ={}& \sqrt{-1}\sum_{i,j=1}^k{C_{ij}y_iy_j}\,\calU\psi(y) = \sqrt{-1}\Tr(yy^tC)\calU\psi(y)\\  
  ={}& \sqrt{-1}(yy^t\psi(yy^t)|C) = \left(\calU\circ\td\pi(C,0,0)\right)\psi(y).
 \end{align*}
\item[\textup{(b)}] Let $(0,A,0)\in\frakg$, $A\in\gl(k,\RR)$. $A$ acts on $V$ by $A\cdot x=Ax+xA^t$ (see Example \ref{ex:StrGrp} (1)). Then
 \begin{align*}
  & \left(\td\mu\left(\begin{array}{cc}-A^t&0\\0&A\end{array}\right)\circ\calU\right)\psi(y)\\
  ={}& \sum_{i,j=1}^k{A_{ji}y_j\frac{\partial}{\partial y_i}\psi(yy^t)}+\frac{1}{2}\Tr(A)\psi(yy^t)\\
  ={}& 2\sum_{i=1}^k{(A^ty)_i\left(\frac{\partial\psi}{\partial x}(yy^t)y\right)_i}+\frac{1}{2}\Tr(A)\psi(yy^t)\\
  ={}& \left(A^t(yy^t)+(yy^t)A\left|\frac{\partial\psi}{\partial x}(yy^t)\right.\right)+\frac{1}{2}\Tr(A)\psi(yy^t)\\
  ={}& \left(\calU\circ\td\pi(0,A,0)\right)\psi(y),
 \end{align*}
 since
 \begin{align*}
  \Tr(V\rightarrow V,x\mapsto A\cdot x) &= (k+1)\Tr(A).
 \end{align*}
\item[\textup{(c)}] Let $(0,0,B)\in\frakg$, $B\in\Sym(k,\RR)$. Then
 \begin{align*}
  & \left(\td\mu\left(\begin{array}{cc}0&\pi B\\0&0\end{array}\right)\circ\calU\right)\psi(y)\\
  ={}& \frac{1}{4\sqrt{-1}}\sum_{i,j=1}^k{B_{ij}\frac{\partial^2}{\partial y_i\partial y_j}\psi(yy^t)} = \frac{1}{2\sqrt{-1}}\sum_{i,j=1}^k{B_{ij}\frac{\partial}{\partial y_i}\left[\sum_\alpha{\frac{\partial\psi}{\partial x_\alpha}(yy^t)(e_\alpha y)_i}\right]}\\
  ={}& \frac{1}{\sqrt{-1}}\sum_{i,j=1}^k{B_{ij}\sum_{\alpha,\beta}{\frac{\partial^2\psi}{\partial x_\alpha\partial x_\beta}(yy^t)(e_\alpha y)_i(e_\beta y)_j}}+\frac{1}{2\sqrt{-1}}\sum_{i,j=1}^k{B_{ij}\sum_\alpha{\frac{\partial\psi}{\partial x_\alpha}(yy^t)(e_\alpha)_{ij}}}\\
  ={}& \frac{1}{\sqrt{-1}}\sum_{\alpha,\beta}{\frac{\partial^2\psi}{\partial x_\alpha\partial x_\beta}(yy^t)\sum_{i,j=1}^k{B_{ij}(e_\alpha y)_i(e_\beta y)_j}}+\frac{1}{2\sqrt{-1}}\sum_{i,j=1}^k{B_{ij}\left(\frac{\partial\psi}{\partial x}(yy^t)\right)_{ij}}\\
  ={}& \frac{1}{\sqrt{-1}}\sum_{\alpha,\beta}{\frac{\partial^2\psi}{\partial x_\alpha\partial x_\beta}(yy^t)\left(\left.P(e_\alpha,e_\beta)(yy^t)\right|B\right)}+\frac{1}{2\sqrt{-1}}\left(\left.\frac{\partial\psi}{\partial x}(yy^t)\right|B\right)\\
  ={}& \frac{1}{\sqrt{-1}}\left(\left.\calB_{\frac{1}{2}}\psi(yy^t)\right|B\right) = \left(\calU\circ\td\pi(0,0,-B)\right)\psi(y).\qedhere
 \end{align*}
\end{enumerate}
\end{proof}

Note that the groups ${G}_\pm^\vee$ are by Definition~\ref{def:minG}\,(1)\,(a) always quotients of the metaplectic group $\Mp(k,\RR)$. Therefore, in order to obtain an intertwining operator between the group representations $\pi^\pm$ and $\mu^\pm$, we may and do lift $\pi^\pm$ to representations of $\Mp(k,\RR)$ which we also denote by $\pi^\pm$. Then we have the following intertwining formula:

\begin{corollary}\label{cor:IntertwinerMetaGr}
For $g\in\Mp(k,\RR)$ we have
\begin{align*}
 \calU^\pm\circ\pi^\pm(g) &= \mu^\pm\left(k_0gk_0^{-1}\right)\circ\calU^\pm.
\end{align*}
Hence $\mu^\pm(k_0^{-1})\circ\calU^\pm$ are intertwining operators between $\pi^\pm$ and $\mu^\pm$.
\end{corollary}

\begin{proof}
This now follows immediately from \eqref{eq:MetaIntertwiningFormulaDer}.
\end{proof}

\begin{remark}
Together with Definition~\ref{def:minG}\,(1)\,(a) the previous proposition shows that for even $k$ the two components of the Weil representation of $\Mp(k,\RR)$ descend to representations of a quotient group of index $2$ which is not a linear group. This quotient group though is different for the two components. This fact can also be seen from the explicit calculation of the cocycle of the Weil representation in \cite[Section 1.6]{LV80}.
\end{remark}

\subsubsection{The minimal representation of $\upO(p+1,q+1)$}

Let $V=\RR^{p,q}$, $p,q\geq2$. Then by Example~\ref{ex:MinimalOrbit}\,(2) the minimal orbit $\calO$ is the isotropic cone
\begin{align*}
 \calO &= \{x\in\RR^{p+q}:x_1^2+\cdots+x_p^2-x_{p+1}^2-\cdots-x_{p+q}^2=0\}\setminus\{0\},
\end{align*}
and the group ${G}^\vee$ is a quotient of $\SO(p+1,q+1)_0$ by Definition~\ref{def:minG}\,(3). In \cite{KO03c} the authors construct a realization of the minimal representation of $\upO(p+1,q+1)$ on $L^2(\calO)$. We use the notation of \cite{KO03c} and denote by $\varpi$\index{varpi@$\varpi$} the minimal representation of $\upO(p+1,q+1)$ on $L^2(\calO)$. The action $\varpi$ of the identity component $\SO(p+1,q+1)_0$ is uniquely determined by the corresponding Lie algebra action $\td\varpi$\index{dvarpi@$\td\varpi$}. Let $f:\frakg\rightarrow\so(p+1,q+1)$\index{f@$f$} be the isomorphism of Example~\ref{ex:ConfGrp}\,(2). Then by \cite[Equation (3.2.1a) and Lemma 3.2]{KO03c} we have
\begin{align*}
 \td\varpi(f(u,0,0)) &= 2\sqrt{-1}\sum_{j=1}^n{u_jx_j} = \sqrt{-1}(x|u) && \mbox{for }u\in V,\\
 \td\varpi(f(0,T,0)) &= D_{T^*x} && \mbox{for }T\in\so(p,q),\\
 \td\varpi(f(0,\1,0)) &= E+\frac{p+q-2}{2} && \mbox{for }s\in\RR,\\
 \td\varpi(f(0,0,-\vartheta(v))) &= \frac{1}{2}\sqrt{-1}\sum_{j=1}^n{v_jP_j}, && \mbox{for }v\in V,
\end{align*}
where $E$ is the \textit{Euler operator}
\begin{align*}
 E &= \sum_{j=1}^n{x_j\frac{\partial}{\partial x_j}}\index{E@$E$}
\end{align*}
and the \textit{fundamental differential operators} $P_j$ on the isotropic cone $\calO$ are the second order differential operators defined by
\begin{align*}
 P_j &= \varepsilon_jx_j\Box_\varepsilon-(2E+n-2)\frac{\partial}{\partial x_j}\index{Pj@$P_j$}
\end{align*}
with
\begin{align*}
 \Box_\varepsilon &= \sum_{j=1}^n{\varepsilon_j\frac{\partial^2}{\partial x_j^2}},\index{Box@$\Box$}\\
 \varepsilon_j &= \begin{cases}+1 & \mbox{for }1\leq j\leq p,\\-1 & \mbox{for }p+1\leq j\leq n.\end{cases}\index{epsj@$\varepsilon_j$}
\end{align*}

\begin{proposition}
For $X\in\frakg$ we have
\begin{align}
 \td\varpi(f(X)) &= \td\pi(X).\label{eq:Intertwiningopq}
\end{align}
\end{proposition}

The previous proposition now implies the following result for the group representations:

\begin{corollary}
The representation $\varpi$ of $\SO(p+1,q+1)_0$ descends to the group $G^\vee$ on which it agrees with $\pi$ if $p,q>1$.
\end{corollary}

\begin{remark}
Second order differential operators similar to the fundamental differential operators $P_j$ also appear in different contexts. See the operators $D_u$ in \cite[Section 3]{BT77} and the operators $\Phi_j$ and $\Theta_j$ in \cite[Section 2]{LS99}, for example.
\end{remark} 

\section{Relations with previous results}\label{sec:Complements}

In this section we relate our results to previous work on the subject. In particular, we interpret the representation $\pi$ as a subrepresentation of a certain (degenerate) principal series representation. This allows us to state a unified $\frakk$-type formula of our minimal representations. We further find explicit $\frakk$-finite vectors in every $\frakk$-type which establishes a connection to special functions satisfying certain differential equations of fourth order. Finally, we prove Theorem C from the introduction using the unitary inversion operator $\calF_\calO$.

\subsection{Degenerate principal series}\label{sec:PrincipalSeries}

In this subsection we show
that for the minimal non-zero orbit $\calO_\lambda=\calO$  the representation $\pi=\pi_\lambda$ is a subrepresentation of a degenerate principal series representation $\omega_s$.


We start with some general observations.

For a $\frakk$-module $V$ we set
$$V_\frakk:=\{ v\in V : \dim U(\frakk)v<\infty\}.$$
Let $K$ be a connected compact Lie group with Lie algebra $\frakk$. Although we have not assumed that there is an action of $K$ on $V$, we can define the space $V_K$ of vectors $v\in V_\frakk$ for which the $\frakk$-action on $U(\frakk)v$ lifts to $K$. Then
\begin{equation}\label{eqdef:WK}
V_K=\sum_{(\tau,V_\tau)\in\hat K}\sum_{\psi\in \Hom_\frakk(V_\tau,V)}\psi(V_\tau),
\end{equation}
where $\hat K$ is the unitary dual of $K$.

\begin{lemma}\label{lem:K-fin} Let $K$ be a connected compact Lie group, $M\subseteq K$ a closed subgroup and $U$ a connected open neighborhood of the canonical base point in $K/M$. Denote the sheaves of real analytic functions and hyperfunctions by $\cal A$ and $\cal B$ respectively. Then the Lie algebra $\frakk$ acts on both, $\mathcal B(U)$ and $\mathcal A(U)$.
 \begin{enumerate}
 \item[{\rm (1)}] $\mathcal B(U)_K =  \mathcal A(U)_K$.
 \item[{\rm (2)}] There is a natural inclusion $\mathcal A(U)_K\hookrightarrow \mathcal A(K/M_0)_K$, where $M_0$ is the identity component of $M$.
 \end{enumerate}
\end{lemma}

\begin{proof} 
Part (1) follows from elliptic regularity. To prove (2), in view of
\eqref{eqdef:WK}, it suffices to show that for each
$(\tau,V_\tau)\in\widehat{K}$ and
$\psi\in\Hom_\frakk\big(V_\tau,\mathcal A(U)\big)$ the space
$\psi(V_\tau)$ can be naturally viewed as a subspace of $\mathcal
A(K/M_0)_K$. To do that, compose $\psi$ with the evaluation map
$\mathrm{ev}_o\colon \mathcal A(U)_K\to\CC$ at the base point $o=eM\in
K/M$ to obtain an element $v=v_\psi$ in the dual space $V_\tau^\vee$.
Then
\begin{equation*}
\psi(u)(0) = \langle u,v \rangle
\quad\text{for $u\in V_\tau$}.
\end{equation*}
Let $\tau^\vee$ be the 
contragredient representation on $V_{\tau^\vee}\simeq V_\tau^\vee$,
and $d\tau^\vee$ the differential representation of $\tau^\vee$.
For $Y\in \frakk$ we can calculate
\[
-\langle u,\mathrm d\tau^\vee(Y)v\rangle
=\langle \mathrm d \tau(Y)u,v\rangle
=\psi\big(\mathrm d\tau(Y)u\big)(o)
=\widetilde Y_o\psi(u).
\]
Here $\widetilde Y$ denotes the vector field on $K/M$ associated with $Y\in \frakk$ via 
$$
\widetilde Y_x:= \frac{\mathrm d}{\mathrm dt}\vert_{t=0} \exp(-tY)\cdot x\in T_x(K/M).
$$ 
We extend this Lie algebra representation to a representation of $U(\frakk)$ in the differential operators on $K/M$.

If $Y\in\frakm$, then $\widetilde Y_o=0$ and therefore $\mathrm d\tau^\vee(Y)v=0$. This means that $v\in (V_\tau^\vee)^\frakm=(V_\tau^\vee)^{M_0}$. Hence the matrix coefficient
$$f_u(k):=\langle \tau(k^{-1})u,v\rangle=\langle u,\tau^\vee(k)v\rangle\in \mathcal A(K)$$
may be regarded as a real analytic function on $K/M_0$, resulting in a $\frakk$-homomorphism
$V_\tau\to \mathcal{A}(K/M)$, $u\mapsto f_u$. It remains to be shown that $\psi(u)(k\cdot o)=f_u(k)$ for $k\in K$ close enough to the identity. Since both functions are real analytic, this follows if all derivatives agree in $o$, i.e.  from the calculation
\[
(\widetilde Y\psi(u))(o)
=\psi(\mathrm d\tau(Y)(u))(o)
=\langle \mathrm d \tau(Y)u,v\rangle
=(\widetilde Y f_u)(e),
\]
for any $Y\in U(\kf)$.
\end{proof}

Next,
let $G$ be any connected real reductive group with Iwasawa decomposition $G=KAN$. Suppose $P\subseteq G$ is a parabolic subgroup with Levi decomposition $P=LN$. Note that $P$ may have more than one connected component. Further, let $\chi:P\to\CC^\times$ be a character of $P$. We consider the induced representation $\Ind_P^G(\chi)$ of $G$ (normalized parabolic induction). It is given by the left regular action of $G$ on
\begin{align*}
 \Ind_P^G(\chi) &= \{f\in C^\infty(G): f(gp)=(\delta\chi)(p)^{-1}f(g)\ \forall\, g\in G,p\in P\},
\end{align*}
where
\begin{align}\label{eqn:deltaP}
 \delta(g) &:= |\Det(\Ad(g)|_N)|^{\frac{1}{2}} 
\quad\text{for $g\in P$}.
\end{align}
Let $\overline{N}$ be the unipotent radical of the opposite parabolic
subgroup.
Since $\overline{N}P$ is open and dense in $G$ by the Gelfand--Naimark decomposition,
a function $f\in\Ind_P^G(\chi)$ is already determined by its
restriction to $\overline{N}$. Identifying $\overline{N}$ 
  with its Lie algebra $\overline{\frakn}$ via the
exponential function, we can view  the representation $\Ind_P^G(\chi)$ 
as a representation on a subspace of
$C^\infty(\overline{\frakn})$. The differential representation of
the Lie
algebra $\frakg$ of $G$
 can be extended to a representation  on the space
$\calB(\nfo)$ of hyperfunctions,
 which we denote  by $\Ind_{\mathfrak{p}}^{\gf}(d\chi)$, 
as this differential representation depends only on $\td\chi$. In particular it does not depend on either coverings of $G$ or the values of $\chi$ on connected components of $P$ other than the identity component.

Decompose the induced representation $\Ind_{P_0}^P(\chi|_{P_0})$ into
irreducible $P$-representations. We denote by $T_{\td\chi}$ the set of
irreducible $P$-representations occurring in this decomposition. 
Using the induction by stages,
we get an isomorphism of $G$-modules:
\begin{equation}\label{eqn:indPPG}
\Ind_{P_0}^G (\chi|_{P_0}) \simeq \bigoplus_{\tau\in T_{d\chi}}
\dim\tau \, \Ind_P^G(\tau).
\end{equation}
Since $P$ centralizes the split center of $L$, we have
\begin{align}
 \td\tau &= \td\chi\cdot\id\label{eq:dtau=dchiid},
\end{align}
for any $\tau\in T_{d\chi}$.
If $P$ is connected, then $T_{d\chi}$ consists of a single element
$\chi$ which is one-dimensional.
For general $P$,
in light that $K\cap P$ meets every connected component of $P$,
we see that any $\tau\in T_{d\chi}$ is determined by the restriction 
$\tau|_{K\cap P}$,
which is still irreducible.

The following lemma will be instrumental in proving that the representation $\pi$ is a subrepresentation of a degenerate principal series representation.

\begin{lemma}\label{lem: scalar K type}
Let $\gf$ act on $\calB(\overline{\nf})$ by $\Ind_{\mathfrak{p}}^{\gf}(d\chi)$.
Suppose that $W$ is a $(\frakg,K)$-module such that $\Hom_\frakg(W,\calB(\nfo))\neq0$.
\begin{itemize}
\item[1)]
There exists at least one $\tau\in T_{d\chi}$ such that
\begin{equation}\label{eqn:WInd}
\Hom_{(\gf,K)}(W,\Ind_P^G(\tau))\ne0.
\end{equation}
\item[2)]
If $W$ admits a scalar $K$-type $\mu$, then such $\tau\in T_{d\chi}$
satisfying \eqref{eqn:WInd} exists
uniquely.
Further, this $\tau$ is one-dimensional.
It is characterized by the following three conditions:
$\tau\in T_{d\chi}$, $\tau$ is one-dimensional, and
 $\tau\vert_{K\cap P}=\mu\vert_{K\cap P}$.
\end{itemize}
\end{lemma}

\begin{proof} 
1)\enspace
We apply Lemma~\ref{lem:K-fin} for $M:=K\cap P$ and
$U := \overline{N}P/P \simeq \overline{\nf}$ regarded as an open dense
set in $K/K \cap P \simeq G/P$.
In view of \eqref{eqn:indPPG}, we get:
\begin{equation}\label{lem:GeneralPrincipalSeries2a}
 \calB(\overline{\frakn})_K \cong \Big(\Ind_{P_0}^G(\chi|_{P_0})\Big)_K \cong \bigoplus_{\tau\in T_{\td\chi}}{\dim\tau(\Ind_P^G(\tau))_K}.
\end{equation}

Let $W$ be a $(\frakg,K)$-module. Then \eqref{lem:GeneralPrincipalSeries2a} implies
\begin{equation}\label{lem:GeneralPrincipalSeries2}
 \Hom_\frakg(W,\calB(\overline{\frakn})) \cong \bigoplus_{\tau\in T_{\td\chi}}{\dim\tau\,\Hom_{(\frakg,K)}(W,\Ind_P^G(\tau))}.
\end{equation}
Hence the first statement follows.

2)\enspace
Since
$G=KP$ we have an isomorphism:
$\Ind_{P}^{G}(\tau) \simeq \Ind_{K\cap P}^{K}(\tau\vert_{K\cap P})$ 
as $K$-modules.
Hence, we get
\begin{equation}\label{eq:TauCondition2a}
\Hom_{K\cap P}(\mu|_{K\cap P},\tau|_{K\cap P})\cong \Hom_{K}\Big(\mu,\Ind_{K\cap P}^{K}(\tau\vert_{K\cap P})\Big)\not=0,
\end{equation}
by the Frobenius reciprocity,
if $\mu$
 occurs as a $K$-type in $\Ind_P^G(\tau)$.
Since $\tau\vert_{K\cap P}$ is irreducible,
$\tau$ must be one-dimensional by \eqref{eq:TauCondition2a} if $\mu$
is one-dimensional.
The last statement is clear because $K\cap P$ meets every connected
component of $P$.
\end{proof}


We return to the situation, where $G=\Co(V)_0$ is the identity
component of the conformal group of a simple real Jordan algebra $V$
with simple $V^+$, excluding the case $V\cong\RR^{p,q}$, $p+q$ odd,
$p,q\geq2$. Recall that $Q^{\textup{max}}$ denotes the maximal
parabolic subgroup of $G$ corresponding to the maximal parabolic
subalgebra $\frakq^{\textup{max}}$ (see Subsection
\ref{sec:KKT}). $Q^{\textup{max}}$ has a Langlands decomposition
$Q^{\textup{max}}=L^{\textup{max}}\ltimes N$ with
$L^{\textup{max}}\subseteq\Str(V)$. Recall the character $\chi$ of
$\Str(V)$ defined by \eqref{eq:DetEquiv}. For $s\in\CC$ we introduce
the character of $L^{\max}$ by
\begin{align*}
 \chi_s(g) &:= |\chi(g)|^s 
  \quad\text{for $g\in L^{\textup{max}}$}.\index{chisg@$\chi_s(g)$}
\end{align*}
We extend $\chi_s$ to the opposite parabolic $\overline{Q^{\textup{max}}}:=L^{\textup{max}}\ltimes\overline{N}$ by letting $\overline N$ act trivially. Then the character $\delta$ (see \eqref{eqn:deltaP}) amounts to
\begin{align*}
 \delta(g) &= |\Det(\Ad(g)|_{\overline{N}})|^{\frac{1}{2}} = |\chi(g)|^{-\frac{n}{2r}} = \chi_{-\frac{n}{2r}}(g) 
\quad\text{for $g\in L^{\max}$}.
\end{align*}

The degenerate principal series representation $\omega_s$ of $G$ on
$\Ind \frac{G}{Q^{\max}}(\chi_s)$ takes the form
\begin{equation}\label{eq:omega s}
 \omega_s(g)\eta(x) = (\delta\tau)(Dg^{-1}(x))^{-1}\eta(g^{-1}\cdot x)
\quad\text{for $x\in V\simeq N$, $g\in G$}.
\end{equation}
Here, $Dg^{-1}(x)$ denotes the differential of the conformal transformation 
$x\mapsto g^{-1}\cdot x$ whenever it is defined. 
Then the differential representation 
$\td\omega_s = \Ind\frac{\frakg}{\frakq^{\max}}(d\chi_s)$
is given in terms of the Jordan algebra
  as follows  (cf.\ Pevzner \cite[Lemma 2.6]{Pev02}):
\begin{align*}
 \td\omega_s(X)\eta(x) &= -D_u\eta(x) & \mbox{for }X &= (u,0,0),\\
 \td\omega_s(X)\eta(x) &= \left(\frac{rs}{n}-\frac{1}{2}\right)\Tr(T)\eta(x)-D_{Tx}\eta(x) & \mbox{for }X &= (0,T,0),\\
 \td\omega_s(X)\eta(x) &= \left(2s-\frac{n}{r}\right)\tau(x,v)\eta(x)-D_{P(x)v}\eta(x) & \mbox{for }X &= (0,0,-v).
\end{align*}

Note that $L^2(\calO_\lambda,\td\mu_\lambda)$ is contained in the
space $\calS'(V)$ of tempered distributions. For $\lambda\in\calW$ consider the Fourier transform $\calF_\lambda:L^2(\calO_\lambda,\td\mu_\lambda)\rightarrow\calS'(V)$ given by
\begin{align}\label{eq: FT}
 \calF_\lambda\psi(x) &= \int_{\calO_\lambda}{e^{-\sqrt{-1}(x|y)}\psi(y)\td\mu_\lambda(y)}, & x\in V,
\end{align}
where $(\ \mid\ )$ is the trace form from \eqref{eq:traceform}. Combining Proposition~\ref{prop:LieAlgRep} and Theorem~\ref{thm:BlambdaTangential},
we can verify that $\calF_\lambda$ intertwines the actions of $\td\pi_\lambda$ and $\td\omega_s$ for $s=\frac{1}{2}\left(\frac{n}{r}-\lambda\right)$ (see \cite[Proposition~2.2.1]{Moe10}, for the detailed calculation):

\begin{proposition}\label{prop:IntertwinerPrincipalSeries}
Let $\lambda\in\calW$ and $s=\frac{1}{2}\left(\frac{n}{r}-\lambda\right)$. Then for $X\in\frakg$ we have
\begin{align*}
 \calF_\lambda\circ\td\pi_\lambda(X) &= \td\omega_s(X)\circ\calF_\lambda
 \quad\text{on $C_c^\infty(\calO_\lambda)$}.
\end{align*}
\end{proposition}

We now restrict to the case where $\lambda\in\calW$ such that $\calO_\lambda=\calO$ and we put again $s=\frac{1}{2}\left(\frac{n}{r}-\lambda\right)$. In the case of split rank $r_0=1$ as before we assume that $\sigma=r\lambda\in(0,-2\nu)$. Under these assumptions we constructed the representation $\pi$ of $G^\vee$ on $L^2(\calO,\td\mu_\lambda)$ in Subsection \ref{sec:ConstructionL2Model}. 

Denote by $\calH\subseteq\calS'(V)$ the image of $L^2(\calO,\td\mu_\lambda)$ under the Fourier transform $\calF_\lambda$ and endow it with the Hilbert space structure turning $\calF_\lambda$ into a unitary isomorphism, see \eqref{eq: FT}.

 Denote by $\overline{Q^{\textup{max}}}^\vee$ the maximal parabolic
 subgroup of $G^\vee$ which projects onto
 $\overline{Q^{\textup{max}}}$ under the covering map $G^\vee\to G$. 
The characters $\chi_s$ and $\delta$ naturally lift to $\overline{Q^{\textup{max}}}^\vee$, and we denote these lifts by the same letters.

Below, we will show that there is a unique character $\tau_s$ for
which the maximal globalization $(I_s^{\max},\omega_s)$ of the
degenerate principal series representation
$\Ind\genfrac{}{}{0pt}{}{G^\vee}{\overline{Q^{\max}}^\vee}(\tau_s)$ 
contains the unitary representation $\pi$ as a subrepresentation.

To be more precise,
we define a one-dimensional representation $\tau_s$ of
$\overline{Q^{\max}}^\vee$ as follows:
On the connected component of $\overline{Q^{\max}}^\vee$ containing
the identity,
$\tau_s$ is subject to
\begin{equation*}
d\tau_s = d\chi_s
\quad\text{with $s = \frac12\left(\frac{n}{r}-\lambda\right)$}.
\end{equation*}
For possible disconnected components of $\overline{Q^{\max}}^\vee$,
we divide the cases according to the dimension of the minimal $K$-type
$\mu$ of $\pi$ on $L^2(\calO,d\mu_\lambda)$ which we found
explicitly in Theorems \ref{prop:Kfinite} and \ref{prop:Kfinite-split rank 1}.

\medskip
\noindent
Case 1. $V\not\simeq\RR^{p,q}$

In this case, $\mu$ is one-dimensional.
Then our $\tau_s$ is characterized by 
\begin{equation}\label{eqn:tau1}
\tau_s|_{M_L} = \mu|_{M_L},
\end{equation}
where $M_L:=\overline{Q^{\max}}^\vee \cap K^\vee$.

\medskip
\noindent
Case 2. $V\simeq\RR^{p,q}$.

We note we have excluded the case $p+q$ is odd and $p,q\ge2$.
If both $p$ and $q$ are odd and $p-q\equiv0\bmod4$,
then $G^\vee=G$ and $\overline{Q^{\max}}$ is connected.
Otherwise,
$G^\vee\simeq\SO(p+1,q+1)_0$ is a double covering group of $G$
(see Definition \ref{def:minG} (3) and (4)),
and the parabolic subgroup $\overline{Q^{\textup{max}}}^\vee$ has two components $(\overline{Q^{\textup{max}}}^\vee)_0$ and $m_0(\overline{Q^{\textup{max}}}^\vee)_0$, $m_0$ being the non-trivial element of $G^\vee$ projecting onto the identity element (see \cite{KO03c}).  
Then $T_{d\chi}$ consists of two characters of
$\overline{Q^{\max}}$ according as the evaluation at $m_0$ is $1$
or $-1$.
We define $\tau_s\in T_{d\chi}$ characterized by
\begin{equation}\label{eqn:tau2}
\tau_s(m_0) = \begin{cases}
                 -1 &\text{if $p,q$ both odd, $p-q\equiv2\bmod4$,} \\
                 +1 &\text{otherwise.}
              \end{cases}
\end{equation}

\begin{theorem}\label{thm:principal series} 
Let $\lambda\in\calW$ such that $\calO_\lambda=\calO$ and set $s=\frac{1}{2}\left(\frac{n}{r}-\lambda\right)$. In the case of split rank $r_0=1$ assume that $\sigma=r\lambda\in(0,-2\nu)$. 
\begin{enumerate}
\item[{\rm(1)}]
There exists a unique character $\tau_s$ of $\overline{Q^{\textup{max}}}$ such that $\td\tau_s=\td\chi_s$
and that the degenerate principal series representation
$\Ind\genfrac{}{}{0pt}{}{G^\vee}{\overline{Q^{\max}}^\vee}(\tau_s)$ contains a
$(\gf,K^\vee)$-module which is isomorphic to the underlying
$(\gf,K^\vee)$-module of the unitary representation $\pi$ on 
$L^2(\calO,d\mu_\lambda)$.
Such a character $\tau_s\in T_{d\chi_s}$ is characterized by
\eqref{eqn:tau1} and \eqref{eqn:tau2}.
\item[{\rm(2)}]
The Fourier transform $\calF_\lambda$ is an intertwining operator from
$L^2(\calO,\td\mu_\lambda)$ into the maximal globalization of the
degenerate principal series representation 
$\Ind\genfrac{}{}{0pt}{}{G^\vee}{\overline{Q^{\max}}^\vee}(\tau_s)$.
\end{enumerate}
\end{theorem}

\begin{proof}
We apply Lemma~\ref{lem: scalar K type} for $P=\overline{Q^{\textup{max}}}^\vee$, $\chi=\chi_s$, and the $(\frakg,K^\vee)$-module $W=\td\pi_\lambda(U(\frakg))\psi_0$ from Section~\ref{sec:ConstructiongkModule}.
The formula in 
 Proposition \ref{prop:IntertwinerPrincipalSeries} on
$C_c^\infty(\calO)$ still holds for
$K$-finite vectors of $L^2(\calO,d\mu_\lambda)$ in light of the
asymptotic behaviours of $K$-finite vectors (see
 Proposition
\ref{prop:L2ness}), and therefore we have $\Hom_\frakg(W,\calS'(V))\neq0$.
Therefore, Theorem \ref{thm:principal series} follows from Lemma
\ref{lem: scalar K type} if $W$ admits a scalar $K^\vee$-type $\mu$.

In the remaining case where
 $W$ does not necessarily admit a scalar $K^\vee$-type, i.e., 
for $V=\RR^{p,q}$,
Theorem \ref{thm:principal series} was proved
 in
 \cite[Theorem 4.9]{KO03c}.
For the sake of completeness,
we give a proof 
along the same line of argument here.
We already know that $\overline{Q^{\max}}^\vee$ has at most two
 connected components and therefore
 that $T_{\td\chi}$ consists of characters. Thus 
it is sufficient to determine a one-dimensional representation
$\tau\in T_{d\chi}$ such that $\tau(m_0)=\mu(m_0)$ when
$G^\vee=\SO(p+1,q+1)_0$.
 If  $p\geq q$, we have $\mu=\1\boxtimes\calH^{\frac{p-q}{2}}(\RR^{q+1})$ on which $m_0$ acts as the scalar $(-1)^{\frac{p-q}{2}}$. Therefore, the $\tau $ we are looking for is characterized by
\begin{align*}
 \tau_s(g) &= \chi_s(g), & g\in(\overline{Q^{\textup{max}}}^\vee)_0.\\
 \tau_s(m_0) &= (-1)^{\frac{p-q}{2}}.
\end{align*}
The case $q\geq p$ is treated similarly.
\end{proof}

\begin{remark}
In the split rank $1$ case, i.e. $\frakg\cong\so(k+1,1)$, $k\geq1$, we obtain the entire complementary series for $\SO(k+1,1)_0$. For all parameters $\lambda$ for which $W$ is contained in $L^2(\calO,d\mu_\lambda)$ the Fourier transform ${\mathcal F}_\lambda$ intertwines the principal series realizations with the $L^2$-models. Moreover, these $L^2$-models coincide with the ``commutative models'' studied by Vershik--Graev \cite{VG06}.
\end{remark}

\begin{remark}
In the 1990s a number of papers appeared dealing with the structure of
degenerate principal series representations. Among others
\cite{Sah93}, \cite{Sah95}, and \cite{Zha95} determine the irreducible
and unitarizable constituents of the degenerate principal series
representations associated to conformal groups of euclidean and
non-euclidean Jordan algebras. The proofs are of an algebraic
nature. Using these results, A. Dvorsky and S. Sahi as well as
L. Barchini, M. Sepanski and R. Zierau considered unitary
representations of the corresponding groups on $L^2$-spaces of orbits
of the structure group. In \cite{Sah92} the case of a euclidean Jordan
algebra is treated and the non-euclidean case is studied in
\cite{DS99}, \cite{DS03} and \cite[Section 8]{BSZ06}. However, they
all exclude the case $V=\RR^{p,q}$ for general $p$ and $q$ such that
$\pi$ does not contain a scalar $K$-type (i.e.\ $p+q$: even,
$p,q\ge2$, and $p\neq q$). 
In fact, contrary to what was claimed in \cite[p. 206]{DS99},
it is possible to extend the Mackey representation of $\overline{Q^{\textup{max}}}^\vee$ to the whole group $G^\vee = SO(p+1,q+1)_0$.
In this case  the $L^2$-model of the minimal representation was established by T. Kobayashi and B. {\O}rsted in \cite{KO03c}.
\end{remark}

The interpretation of the minimal representation  $\td\pi$ as a subrepresentation of a degenerate principal series representation allows us to compute its infinitesimal character. We parameterize the infinitesimal character by the Harish-Chandra isomorphism
 $$\mathrm{Hom}_{\text{$\CC$-algebra}}(Z(\mathfrak g),\CC)\cong \mathfrak h^*/W,$$
where $Z(\mathfrak g)$ is the center of $U(\mathfrak g)$, $\mathfrak h$ is a Cartan subalgebra containing $(0,\id,0)\in\frakl_\CC\subseteq\frakg_\CC$, and $W$ is the corresponding Weyl group. It is possible to choose $\frakh\subseteq\frakl_\CC$. We denote by $\rho_{\frakl_\CC}$ the half-sum of all positive roots of $\frakl_\CC$ with respect to $\frakh$.

Since the infinitesimal character is preserved by the normalized parabolic induction, we get

\begin{theorem}\label{thm: inf char}
The representation $\td\pi$ of $\frakg$ has the infinitesimal character $\td\chi_s+\rho_{\frakl_\CC}$, where $s=\frac{1}{2}\left(\frac{n}{r}-\lambda\right)$.
\end{theorem}

\subsection{Special functions in the $L^2$-models}\label{sec:SpecialFcts}

In this section we find explicit $\frakk$-finite vectors in each $\frakk$-type of the $L^2$-models for the representations from Theorem~\ref{thm:IntgkModule}. These $\frakk$-finite vectors are essentially one-variable functions solving a fourth-order differential equation. The differential equation as well as its corresponding solutions are studied in detail in \cite{HKMM09a,HKMM09b} building on the minimal representation of $O(p,q)$. It is noteworthy that this same set of special functions appears in a uniform fashion for the $L^2$-models for minimal representations of all other groups that were constructed in the previous section.\\

In view of Proposition \ref{prop:IntertwinerPrincipalSeries} and Theorem \ref{thm:principal series}, the $(\frakg,\frakk)$-module $W$ is realized as a subrepresentation of the (degenerate) principal series representation $I_s$. Using the structural results for the composition series of $I_s$, \cite[Equation (7)]{Sah93} for the euclidean case, \cite[\S 0 and Theorems 4.A and 4.B]{Sah95} for the case of a non-euclidean Jordan algebra $\ncong\RR^{p,q}$, $p,q\geq2$, and \cite[Lemma 2.6~(2)]{KO03c} for the case of $V=\RR^{p,q}$, we find the $\frakk$-type decomposition of the $(\frakg,\frakk)$-module $W$. To this end we put
\begin{align}
 W^j &:= E^{\alpha_0+j\gamma_1}.\label{eq:DefWj}
\end{align}
Here $\alpha_0$ denotes the highest weight of the minimal $\frakk$-type $W_0\cong W^0$ (see Theorems \ref{prop:Kfinite} and \ref{prop:Kfinite-split rank 1}), $\gamma_1$ was defined in Subsection \ref{subsec:ConfGrpRoots}, and $E^\alpha$ was introduced in Subsection~\ref{ssec: kl-fix vector}.

\begin{theorem}\label{thm: K-type formula} The $\frakk$-type decomposition of the $(\frakg,\frakk)$-module $W$
is given by
\begin{align*}
 W &\cong \bigoplus_{j=0}^\infty{W^j}.
\end{align*}
In each $\frakk$-type $W^j$ the space of $\frakk_\frakl$-fixed vectors is one-dimensional.
\end{theorem}

The $\frakk_\frakl$-fixed vectors are exactly the radial functions $\psi(x)=f(|x|)$, $x\in\calO$. Denote by $L^2(\calO,\td\mu_\lambda)_{\rad}\subseteq L^2(\calO,\td\mu_\lambda)$ the subspace of radial functions. By \eqref{eq:dmuIntFormula}, the map $\calO\rightarrow\RR_+,\,x\mapsto|x|,$ induces an isomorphism $L^2(\calO,\td\mu_\lambda)_{\rad}\cong L^2(\RR_+,t^{r\lambda-1}\td t)$. Let $W^j_{\rad}:=W^j\cap L^2(\calO,\td\mu_\lambda)_{\rad}$. We then obtain that the algebraic direct sum
\begin{align*}
 \bigoplus_{j=0}^\infty{W^j_{\rad}}\subseteq L^2(\calO,\td\mu_\lambda)_{\rad}
\end{align*}
is dense. To determine a generator for the one-dimensional subspaces $W^j_{\rad}$ we compute the action of the $\frakk$-Casimir on radial functions.\\

On $\frakk$ we define an $\ad$-invariant inner product by
\begin{align*}
 \langle X_1,X_2\rangle &:= B_\frakl(T_1,T_2^*)+2\Tr(T_1T_2^*)+\frac{8n}{r}(u_1|u_2),
\end{align*}
where $B_\frakl$ denotes the Killing form of $\frakl$ (cf. \cite[proof of Proposition 1.1]{Pev02}). Choose an orthonormal basis $(X_\alpha)_\alpha$ with respect to this inner product and define the Casimir element $C_\frakk$ of $\frakk$ by
\begin{align*}
 C_\frakk &:= \sum_\alpha{X_\alpha^2}.
\end{align*}
$C_\frakk$ is independent of the choice of the basis $(X_\alpha)_\alpha$ and defines a central element of degree $2$ in $U(\frakk)$. Therefore, thanks to Schur's Lemma, $\td\pi(C_\frakk)$ acts on each $\frakk$-type $W^j$ by a scalar. This scalar can be determined using root data. For this we define a constant $\mu=\mu(V)$ by
\begin{equation}\label{eq:mu}
 \mu = \mu(V) := \frac{n}{r_0}+\left|d_0-\frac{d}{2}\right|-2.
\end{equation}
Table \ref{tb:MuNu} lists the possible values of $\mu$ and $\nu$ (see \eqref{eqn: def nu} for the definition of $\nu$) for all simple real Jordan algebras.

\begin{proposition}\label{prop:CasimirScalar}
The Casimir operator $\td\pi(C_\frakk)$ acts on every $\frakk$-type $W^j$ of $W$ by the scalar
\begin{align*}
 -\frac{r_0}{8n}\left(4j(j+\mu+1)+\frac{r_0d}{2}\left|d_0-\frac{d}{2}\right|\right).
\end{align*}
\end{proposition}

Recall from \cite{HKMM09a} the  fourth order differential operator  $\widetilde{\calD}_{\alpha,\beta}$ in one variable given by
\begin{align*}
 \widetilde{\calD}_{\alpha,\beta} &= \frac{1}{t^2}\left((\theta+\alpha+\beta)(\theta+\alpha)-t^2\right)\left(\theta(\theta+\beta)-t^2\right),
\end{align*}
for $\alpha,\beta\in\CC$, where $\theta=t\frac{\td}{\td t}$ denotes the one-dimensional Euler operator. 
\begin{remark}\label{rem:Dab}
To be precise,
we have introduced in \cite[(1.11)]{HKMM09a} the following differential
operator:
\[
\calD_{\alpha,\beta}
:= \widetilde{\calD}_{\alpha,\beta}
   -\frac{1}{2}(\alpha-\beta)(\alpha+\beta+2),
\]
so that the symmetry
$\calD_{\alpha,\beta}=\calD_{\beta,\alpha}$ holds.
\end{remark}
The action of $\td\pi(C_\frakk)$ on radial functions can be expressed in terms of $\widetilde{\calD}_{\alpha,\beta}$:

\begin{theorem}\label{thm:CasimirAction}
Let $\psi(x)=f(|x|)$ ($x\in\calO$) be a radial function for some $f\in C^\infty(\RR_+)$.
\begin{enumerate}
\item[\textup{(1)}] For $V$ of split rank $r_0>1$ we have:
\begin{align*}
 \td\pi(C_\frakk)\psi(x) &= -\frac{r_0}{8n}\left(\widetilde{\calD}_{\mu,\nu}+\frac{r_0d}{2}\left|d_0-\frac{d}{2}\right|\right)f(|x|).
\end{align*}
\item[\textup{(2)}] For $V$ of split rank $r_0=1$ we have with $\lambda=\frac{r_0}{r}\sigma$, $\sigma\in(0,-2\nu)$:
\begin{align*}
 \td\pi(C_\frakk)\psi(x) &= -\frac{r_0}{8n}\left(\widetilde{\calD}_{\mu,\nu+\sigma}+\frac{r_0d}{2}\left|d_0-\frac{d}{2}\right|\right)f(|x|).
\end{align*}
\end{enumerate}
\end{theorem}

\begin{corollary}
\begin{enumerate}
\item[\textup{(1)}] For $V$ of split rank $r_0>1$ let $u\in L^2(\RR_+,t^{\mu+\nu+1}\td t)$ be any eigenfunction of $\widetilde{\calD}_{\mu,\nu}$ for the eigenvalue $4j(j+\mu+1)$, $j\in\NN_0$.
\item[\textup{(2)}] For $V$ of split rank $r_0=1$ and $\lambda=\frac{r_0}{r}\sigma$, $\sigma\in(0,-2\nu)$, let $u\in L^2(\RR_+,t^{\mu+\nu+\sigma+1}\td t)$ be any eigenfunction of $\widetilde{\calD}_{\mu,\nu+\sigma}$ for the eigenvalue $4j(j+\mu+1)$, $j\in\NN_0$.
\end{enumerate}
Then
\begin{align*}
 \psi(x) &:= u(|x|), & x\in\calO
\end{align*}
defines a $\frakk$-finite vector in the $\frakk$-type $W^j$.
\end{corollary}

For a moment we now allow general real parameters $\alpha,\beta\in\RR$. To find explicit $L^2$-eigenfunctions of $\widetilde{\calD}_{\alpha,\beta}$, we recall from \cite[(3.2)]{HKMM09a} the generating functions $G^{\alpha,\beta}(t,s)$ by
\begin{align*}
 G^{\alpha,\beta}(t,s) &:= \frac{1}{(1-t)^{\frac{\alpha+\beta+2}{2}}}\widetilde{I}_{\frac{\alpha}{2}}\left(\frac{st}{1-t}\right)\widetilde{K}_{\frac{\beta}{2}}\left(\frac{s}{1-t}\right),
\end{align*}
where we have renormalized the $I$-Bessel function as
\[
\widetilde{I}_\lambda(z)
:=\left(\frac{z}{2}\right)^{-\lambda} I_\lambda(z)
= \sum_{n=0}^\infty \frac{1}{\Gamma(n+\lambda+1)n!}
  \left(\frac{z}{2}\right)^{2n}.
\]
The function $G^{\alpha,\beta}(t,s)$ is analytic at $t=0$ and defines a series $(\Lambda_j^{\alpha,\beta}(s))_{j=0,1,2,\ldots}$ of real-analytic functions on $\RR_+$ by
\begin{align*}
 G^{\alpha,\beta}(t,s) &= \sum_{j=0}^\infty{\Lambda_j^{\alpha,\beta}(s)t^j}.
\end{align*}

We then have
\begin{equation}
\Lambda_0^{\alpha,\beta}(s)
= \frac{1}{\Gamma(\frac{\alpha+2}{2})}
  \widetilde{K}_{\frac{\rho}{2}}(s).
\end{equation}

We refer to \cite{HKMM09a, HKMM09b, KM10} for basic properties of
$(\Lambda_j^{\alpha,\beta}(s))_{j=0,1,2,\dotsc}$
as `special functions'.
Among others, we recall from \cite{HKMM09a}:
\begin{theorem}
For $\alpha+\beta,\alpha-\beta>-2$ the function $\Lambda_j^{\alpha,\beta}(s)$ is non-zero, contained in $L^2(\RR_+,s^{\alpha+\beta+1}\td s)$ and an eigenfunction of $\widetilde{\calD}_{\alpha,\beta}$ for the eigenvalue $4j(j+\alpha+1)$.
\end{theorem}

We apply this theorem to our setting.
Let $\mu=\mu(V)$ and $\nu=\nu(V)$ be the parameters (see Table
\ref{tb:MuNu}) belonging to a simple real Jordan algebra $V$ for which we constructed the representation $\pi$.

\begin{corollary}\label{cor:3.14}
\begin{enumerate}
\item[\textup{(1)}] Let $V$ be of split rank $r_0>1$. Then for every $j\in\NN_0$, the function
\begin{align*}
 \psi_j(x) &:= \Lambda_j^{\mu,\nu}(|x|), & x\in\calO,
\end{align*}
is a non-zero $\frakk$-finite vector in the $\frakk$-type $W^j$.
\item[\textup{(2)}] Let $V$ be of split rank $r_0=1$ and $\lambda=\frac{r_0}{r}\sigma$, $\sigma\in(0,-2\nu)$. Then for every $j\in\NN_0$, the function
\begin{align*}
 \psi_j(x) &:= \Lambda_j^{\mu,\nu+\sigma}(|x|), & x\in\calO,
\end{align*}
is a non-zero $\frakk$-finite vector in the $\frakk$-type $W^j$ for any $j\in\NN_0$.
\end{enumerate}
\end{corollary}

The function $\psi_j(x)$ with $j=0$ in Corollary \ref{cor:3.14}
coincides,
up to a Gamma factor,
with the generating function $\psi_0$ which was introduced in
\eqref{def:Psi0} and \eqref{def:Psi0-r_0=1}.

\begin{remark}\label{rem:KBessel}
In $L^2$-model of a number of `small' unitary representations,
we can observe that functions belonging to minimal $K$-types are given
by means of the modified $K$-Bessel function $\widetilde{K}_\alpha(x)$ 
(see \cite{DS99,KO03c,xkop,Sah92}).
It is known that  Hermite/Laguerre polynomials
  appear  in the higher $K$-types of the minimal representations
for $\mathfrak{g} = \mathfrak{sp}(n, \mathbb R)$
and $\mathfrak{o}(n,2)$.
The idea behind `special functions' $\Lambda_j^{\mu,\nu}(x)$ is to find
an analog of these classical polynomials
for the  minimal representations that we have constructed.
This idea is pursued in \cite{HKMM09b, KM10}.
\end{remark}

\subsection{The unitary inversion operator $\calF_\calO$}\label{sec: inversion operator}

The proof of Theorem C parallels the argument in \cite[Chapter2]{KM07b}, where we introduce an involutive unitary operator that we call the \textit{unitary inversion operator} $\calF_\calO$ on $L^2(\calO,\td\mu_\lambda)$. This operator intertwines the Bessel operators with multiplication by coordinate functions. 
The unitary inversion operator
$\calF_\calO$ is not only a tool to prove Theorem C, but is of
interest on its own. In fact, it is the Euclidean Fourier transform up
to a phase factor if $V=\Sym(k,\RR)$; for $V=\RR^{p,q}$, the principal
object of the paper \cite{KM07a} is the unitary inversion operator
$\calF_{\calO}$ for the ``light cone'' $\calO$ in the Minkowski space
$\RR^{1,n-1}$ and that of the book \cite{KM07b} is $\calF_{\calO}$ for
the isotropic cone $\calO$ in $\RR^{p,q}$. We refer to \cite{K10} 
and \cite[Chapter 1]{KM07b} for the general program of the $L^2$-model of minimal representations
  and further perspectives on the role of the unitary 
inversion operator $\calF_{\calO}$. 
{}From the representation theoretic viewpoint, the operator
$\calF_\calO$ generates the action $\pi$ of the whole group
${G}^\vee$ 
 together with the (relatively simple) action of a maximal parabolic subgroup on $L^2(\calO,\td\mu_\lambda)$. \\

Let $V$ be a simple real Jordan algebra with $V^+$ simple.
As before, we assume $V\ncong\RR^{p,q}$ with $p+q$ odd, $p,q\geq2$.
Let $\lambda\in\calW$ such that $\calO_\lambda=\calO$. For $V$ of split rank $r_0=1$ we further assume that $\lambda\in(0,-2\nu)$. Then $\td\pi_\lambda$ integrates to an irreducible unitary representation $\pi$ of $G^\vee$ on $L^2(\calO,\td\mu_\lambda)$ as we proved in Theorem \ref{thm:IntgkModule}. Let $$j^\vee:=\exp_{G^\vee}\left(\frac{\pi}{2}({\bf e},0,-{\bf e})\right)\in G^\vee,$$ then $j^\vee$ projects to the conformal inversion $j\in G$ under the covering map $G^\vee\rightarrow G$.

Further recall the Cartan involution $\vartheta\in\Str(V)$. We define the action $\rho_\lambda$ on $L^2(\calO,\td\mu_\lambda)$ also for $\vartheta$, extending formula \eqref{eq:L2Rep2}:
\begin{align*}
 \rho_\lambda(\vartheta)\psi(x) &:= \chi(\vartheta^*)^{\frac{\lambda}{2}}\psi(\vartheta^*x) = \psi(\vartheta x), & \psi\in L^2(\calO,\td\mu_\lambda).
\end{align*}
Then we define the unitary operator $\calF_\calO$ on $L^2(\calO,\td\mu_\lambda)$ by
\begin{align}
 \calF_\calO := e^{-\pi\sqrt{-1}\frac{r_0}{2}(d_0-\frac{d}{2})_+}\rho_\lambda(\vartheta)\pi(j^\vee).\label{eq:DefFO}
\end{align}

\begin{remark}
We define an element of order two in $\Co(V)$ by
\begin{equation}\label{eqn:w0}
w_0:=j\circ\vartheta=\vartheta\circ j\in\Co(V).
\end{equation}
In general $w_0$ is not contained in the identity component $G=\Co(V)_0$. In fact, since the Cartan involution of $\Co(V)$ is given by conjugation with $w_0$, we have $w_0\in G$ if and only if $\rank(G)=\rank(K)$. In particular, $w_0\in G$ if $V$ is euclidean and $w_0\notin G$ if $V$ is complex. We can extend our unitary representation $\pi$ to a (possibly)
disconnected group generated by $G^\vee$ and $\vartheta$, and lift $w_0$ to $w_0^\vee:=j^\vee\circ\vartheta=\vartheta\circ j^\vee$, an element of order $2\ell$ with $\ell\in\NN$ denoting the smallest positive integer such that
\begin{align}
 \ell\frac{r_0}{2}\left(d_0-\frac{d}{2}\right)_+ &\in \ZZ.\label{eq:defk}
\end{align}
Here we use the notation
\begin{align*}
 x_+ &:= \frac{1}{2}(x+|x|) = \begin{cases}x & \mbox{if }x\geq0,\\0 & \mbox{if }x<0.\end{cases}\index{xplus@$x_+$}
\end{align*}
for the positive part of a real number $x\in\RR$. Note that the integer $\ell$ is either $1$, $2$ or $4$ (see Table \ref{tb:Constants}). Then we have
\[
\calF_{\calO}=e^{-\pi\sqrt{-1}\frac{r_0}{2}(d_0-\frac{d}{2})_+}\pi(w_0^\vee)
\]
with the same letter $\pi$ to denote the extension. Details of this extension can be found in \cite[Section 2.4]{Moe10}.
\end{remark}

\begin{proposition}
\begin{enumerate}
\item[\textup{(1)}] $\calF_\calO$ is an involutive unitary operator on $L^2(\calO,\td\mu_\lambda)$.
\item[\textup{(2)}] The following intertwining formulas hold:
\begin{align}
\begin{split}
 \calF_\calO\circ\vartheta x &= -\calB_\lambda\circ\calF_\calO,\\
 \calF_\calO\circ\calB_\lambda &= -\vartheta x\circ\calF_\calO.
\end{split}\label{eq:IntertwiningFormulas}
\end{align}
Moreover, any other operator on $L^2(\calO,\td\mu_\lambda)$ with these properties is a scalar multiple of $\calF_\calO$.
\end{enumerate}
\end{proposition}

\begin{proof}
See \cite[Theorem 2.5.2]{KM07b} for $G=O(p,q)$ and \cite[Theorem 2.4.1 and Corollary 3.8.4]{Moe10} for the general case.
\end{proof}

Now, we are ready to give a proof of Theorem C:

\begin{theorem}\label{thm:BesselRing}
The ring of differential operators on $\calO$ generated by the Bessel operators $\phi(\calB_\lambda)=\langle\phi,\calB_\lambda\rangle$, $\phi\in V^*$, is isomorphic to the ring of functions on $\calO$ which are restrictions of polynomials on $V$.
\end{theorem}

\begin{proof}
The map $V^*\ni\phi\mapsto{\phi(\calB_\lambda)}$ extends to a surjection of the ring of polynomials on $V$ onto the ring of differential operators on $\calO$ generated by the Bessel operators $\phi(\calB_\lambda)$, $\phi\in V^*$. Using \eqref{eq:IntertwiningFormulas}, we find that for $\phi\in\CC[V]$ in the kernel we have
\begin{align*}
 && \phi(\calB_\lambda) &= 0 && \mbox{on $\calO$}\\
 &\Leftrightarrow & \calF_\calO\circ\phi(\calB_\lambda) &= 0 && \mbox{on $\calO$}\\
 &\Leftrightarrow & \phi(-\vartheta x)\circ\calF_\calO &= 0 && \mbox{on $\calO$}\\
 &\Leftrightarrow & \phi(-\vartheta x) &= 0 && \mbox{on $\calO$}\\
 &\Leftrightarrow & \phi &= 0 && \mbox{on $(-\calO)$.}
\end{align*}
Therefore, the ring of differential operators on $\calO$ generated by the Bessel operators $\phi(\calB_\lambda)$, $\phi\in V^*$, is isomorphic to the ring of functions on $(-\calO)$ which are restrictions of polynomials. The latter is canonically isomorphic to the ring of functions on $\calO$ which are restrictions of polynomials and the proof is complete.
\end{proof}

\begin{remark}
For $V=\RR^{p,q}$ the unitary inversion operator $\calF_\calO$ was
studied in detail by the second author and G. Mano in the book
\cite{KM07b}, where the global formula of the group action of
$O(p+1,q+1)$ on $L^2(\calO)$ is obtained with an explicit integral kernel. 
The global formula of $\calF_{\calO}$ for a general
non-euclidean Jordan algebra $V$ is not known except when $V=\RR^{p,q}$.
The relation of $\calF_\calO$ to the special functions
$\Lambda_j^{\mu,\nu}(x)$ is studied in \cite[Sections 2.4 and 3.8]{Moe10}.
Further, the operator $\calF_{\calO}$ may be regarded as a `boundary
value' of a holomorphic semigroup in the case where $V$ is euclidean
(see \cite{howe,K10,KM07b}).
\end{remark} 

\setcounter{table}{1}
\begin{sidewaystable}[ht]
\begin{center}
\begin{tabular}{r||c|c|c|c|c||c|c|c|c|}
  \cline{2-10}
  & $V$ & $n$ & $r$ & $d$ & $e$ & $V^+$ & $n_0$ & $r_0$ & $d_0$\\
  \hline
  \hline
  \multicolumn{1}{|c||}{\multirow{6}*{\begin{tabular}{c}euclidean\\split\end{tabular}}} & $\RR$ & $1$ & $1$ & $0$ & $0$ & $\RR^{1,0}$ & $1$ & $1$ & $0$\\
  \multicolumn{1}{|c||}{} & $\Sym(k,\RR)$ ($k\geq2$) & $\frac{k}{2}(k+1)$ & $k$ & $1$ & $0$ & $\Sym(k,\RR)$ & $\frac{k}{2}(k+1)$ & $k$ & $1$\\
  \multicolumn{1}{|c||}{} & $\Herm(k,\CC)$ ($k\geq2$) & $k^2$ & $k$ & $2$ & $0$ & $\Herm(k,\CC)$ & $k^2$ & $k$ & $2$\\
  \multicolumn{1}{|c||}{} & $\Herm(k,\HH)$ ($k\geq2$) & $k(2k-1)$ & $k$ & $4$ & $0$ & $\Herm(k,\HH)$ & $k(2k-1)$ & $k$ & $4$\\
  \multicolumn{1}{|c||}{} & $\RR^{1,k-1}$ ($k\geq3$) & $k$ & $2$ & $k-2$ & $0$ & $\RR^{1,k-1}$ & $k$ & $2$ & $k-2$\\
  \multicolumn{1}{|c||}{} & $\Herm(3,\OO)$ & $27$ & $3$ & $8$ & $0$ & $\Herm(3,\OO)$ & $27$ & $3$ & $8$\\
  \hline
  \multicolumn{1}{|c||}{\multirow{4}*{\begin{tabular}{c}non-euclidean\\split\end{tabular}}} & $M(k,\RR)$ ($k\geq2$) & $k^2$ & $k$ & $2$ & $0$ & $\Sym(k,\RR)$ & $\frac{k}{2}(k+1)$ & $k$ & $1$\\
  \multicolumn{1}{|c||}{} & $\Skew(2k,\RR)$ ($k\geq2$) & $k(2k-1)$ & $k$ & $4$ & $0$ & $\Herm(k,\CC)$ & $k^2$ & $k$ & $2$\\
  \multicolumn{1}{|c||}{} & $\RR^{p,q}$ ($p\geq2$, $q\geq1$) & $p+q$ & $2$ & $p+q-2$ & $0$ & $\RR^{1,q}$ & $q+1$ & $2$ & $q-1$\\
  \multicolumn{1}{|c||}{} & $\Herm(3,\OO_s)$ & $27$ & $3$ & $8$ & $0$ & $\Herm(3,\HH)$ & $15$ & $3$ & $4$\\
  \hline
  \multicolumn{1}{|c||}{\multirow{5}*{\begin{tabular}{c}complex\\non-split\end{tabular}}} & $\Sym(k,\CC)$ ($k\geq2$) & $k(k+1)$ & $2k$ & $2$ & $1$ & $\Sym(k,\RR)$ & $\frac{k}{2}(k+1)$ & $k$ & $1$\\
  \multicolumn{1}{|c||}{} & $M(k,\CC)$ ($k\geq2$) & $2k^2$ & $2k$ & $4$ & $1$ & $\Herm(k,\CC)$ & $k^2$ & $k$ & $2$\\
  \multicolumn{1}{|c||}{} & $\Skew(2k,\CC)$ ($k\geq2$) & $2k(2k-1)$ & $2k$& $8$ & $1$ & $\Herm(k,\HH)$ & $k(2k-1)$ & $k$ & $4$\\
  \multicolumn{1}{|c||}{} & $\CC^k$ ($k\geq3$) & $2k$ & $4$ & $2(k-2)$ & $1$ & $\RR^{1,k-1}$ & $k$ & $2$ & $k-2$\\
  \multicolumn{1}{|c||}{} & $\Herm(3,\OO)_\CC$ & $54$ & $6$ & $16$ & $1$ & $\Herm(3,\OO)$ & $27$ & $3$ & $8$\\
  \hline
  \multicolumn{1}{|c||}{\multirow{3}*{\begin{tabular}{c}non-euclidean\\non-split\end{tabular}}} & $\Sym(2k,\CC)\cap M(k,\HH)$ ($k\geq2$) & $k(2k+1)$ & $2k$ & $4$ & $2$ & $\Herm(k,\CC)$ & $k^2$ & $k$ & $2$\\
  \multicolumn{1}{|c||}{} & $M(k,\HH)$ ($k\geq2$) & $4k^2$ & $2k$ & $8$ & $3$ & $\Herm(k,\HH)$ & $k(2k-1)$ & $k$ & $4$\\
  \multicolumn{1}{|c||}{} & $\RR^{k,0}$ ($k\geq2)$ & $k$ & $2$ & $0$ & $k-1$ & $\RR^{1,0}$ & $1$ & $1$ & $0$\\
  \hline
\end{tabular}
\caption{Simple real Jordan algebras and their structure constants \label{tb:Constants}}
\end{center}
\end{sidewaystable}

\begin{sidewaystable}[ht]
\begin{center}
\begin{tabular}{r||c|c|c|}
  \cline{2-4}
  & $V$ & $\mu$ & $\nu$\\
  \hline
  \hline
  \multicolumn{1}{|c||}{\multirow{6}*{\begin{tabular}{c}euclidean\\split\end{tabular}}} & $\RR$ & $-1$ & $-1$\\
  \multicolumn{1}{|c||}{} & $\Sym(k,\RR)$ ($k\geq2$) & $\frac{k-2}{2}$ & $-1$\\
  \multicolumn{1}{|c||}{} & $\Herm(k,\CC)$ ($k\geq2$) & $k-1$ & $-1$\\
  \multicolumn{1}{|c||}{} & $\Herm(k,\HH)$ ($k\geq2$) & $2k-1$ & $-1$\\
  \multicolumn{1}{|c||}{} & $\RR^{1,k-1}$ ($k\geq3$) & $k-3$ & $-1$\\
  \multicolumn{1}{|c||}{} & $\Herm(3,\OO)$ & $11$ & $-1$\\
  \hline
  \multicolumn{1}{|c||}{\multirow{4}*{\begin{tabular}{c}non-euclidean\\split\end{tabular}}} & $M(k,\RR)$ ($k\geq2$) & $k-2$ & $0$\\
  \multicolumn{1}{|c||}{} & $\Skew(2k,\RR)$ ($k\geq2$) & $2k-3$ & $1$\\
  \multicolumn{1}{|c||}{} & $\RR^{p,q}$ ($p\geq2$, $q\geq1$) & $\max(p,q)-2$ & $\min(p,q)-2$\\
  \multicolumn{1}{|c||}{} & $\Herm(3,\OO_s)$ & $7$ & $3$\\
  \hline
  \multicolumn{1}{|c||}{\multirow{5}*{\begin{tabular}{c}complex\\non-split\end{tabular}}} & $\Sym(k,\CC)$ ($k\geq2$) & $k-1$ & $-1$\\
  \multicolumn{1}{|c||}{} & $M(k,\CC)$ ($k\geq2$) & $2(k-1)$ & $0$\\
  \multicolumn{1}{|c||}{} & $\Skew(2k,\CC)$ ($k\geq2$) & $2(2k-2)$ & $2$\\
  \multicolumn{1}{|c||}{} & $\CC^k$ ($k\geq3$) & $k-2$ & $k-4$\\
  \multicolumn{1}{|c||}{} & $\Herm(3,\OO)_\CC$ & $16$ & $6$\\
  \hline
  \multicolumn{1}{|c||}{\multirow{3}*{\begin{tabular}{c}non-euclidean\\non-split\end{tabular}}} & $\Sym(2k,\CC)\cap M(k,\HH)$ ($k\geq2$) & $2k-1$ & $-1$\\
  \multicolumn{1}{|c||}{} & $M(k,\HH)$ ($k\geq2$) & $4k-2$ & $0$\\
  \multicolumn{1}{|c||}{} & $\RR^{k,0}$ ($k\geq2)$ & $k-2$ & $-k$\\
  \hline
\end{tabular}
\caption{Simple real Jordan algebras and the constants $\mu$ and $\nu$\label{tb:MuNu}}
\end{center}
\end{sidewaystable}

\begin{sidewaystable}[ht]
\begin{center}
\begin{tabular}{r||c|c|c|c|c|}
  \cline{2-6}
  & $V$ & $\frakg=\co(V)$ & $\frakk=\co(V)^\theta$ & $\frakl=\str(V)$ & $\frakk_\frakl=\str(V)^\theta$\\
  \hline\hline
  \multicolumn{1}{|c||}{\multirow{6}*{\begin{tabular}{c}euclidean\\split\end{tabular}}} & $\RR$ & $\sl(2,\RR)$ & $\so(2)$ & $\RR$ & $0$\\
  \multicolumn{1}{|c||}{} & $\Sym(k,\RR)$ ($k\geq2$) & $\sp(k,\RR)$ & $\su(k)\oplus\RR$ & $\sl(k,\RR)\oplus\RR$ & $\so(k)$\\
  \multicolumn{1}{|c||}{} & $\Herm(k,\CC)$ ($k\geq2$) & $\su(k,k)$ & $\su(k)\oplus\su(k)\oplus\RR$ & $\sl(k,\CC)\oplus\RR$ & $\su(k)$\\
  \multicolumn{1}{|c||}{} & $\Herm(k,\HH)$ ($k\geq2$) & $\so^*(4k)$ & $\su(2k)\oplus\RR$ & $\su^*(2k)\oplus\RR$ & $\sp(k)$\\
  \multicolumn{1}{|c||}{} & $\RR^{1,k-1}$ ($k\geq3$) & $\so(2,k)$ & $\so(k)\oplus\RR$ & $\so(1,k-1)\oplus\RR$ & $\so(k-1)$\\
  \multicolumn{1}{|c||}{} & $\Herm(3,\OO)$ & $\mathfrak{e}_{7(-25)}$ & $\mathfrak{e}_6\oplus\RR$ & $\mathfrak{e}_{6(-26)}\oplus\RR$ & $\mathfrak{f}_4$\\
  \hline
  \multicolumn{1}{|c||}{\multirow{4}*{\begin{tabular}{c}non-euclidean\\split\end{tabular}}} & $M(k,\RR)$ ($k\geq2$) & $\sl(2k,\RR)$ & $\so(2k)$ & $\sl(k,\RR)\oplus\sl(k,\RR)\oplus\RR$ & $\so(k)\oplus\so(k)$\\
  \multicolumn{1}{|c||}{} & $\Skew(2k,\RR)$ ($k\geq2$) & $\so(2k,2k)$ & $\so(2k)\oplus\so(2k)$ & $\sl(2k,\RR)\oplus\RR$ & $\so(2k)$\\
  \multicolumn{1}{|c||}{} & $\RR^{p,q}$ ($p\geq2$, $q\geq1$) & $\so(p+1,q+1)$ & $\so(p+1)\oplus\so(q+1)$ & $\so(p,q)\oplus\RR$ & $\so(p)\oplus\so(q)$\\
  \multicolumn{1}{|c||}{} & $\Herm(3,\OO_s)$ & $\mathfrak{e}_{7(7)}$ & $\su(8)$ & $\mathfrak{e}_{6(6)}\oplus\RR$ & $\sp(4)$\\
  \hline
  \multicolumn{1}{|c||}{\multirow{5}*{\begin{tabular}{c}complex\\non-split\end{tabular}}} & $\Sym(k,\CC)$ ($k\geq2$) & $\sp(k,\CC)$ & $\sp(k)$ & $\sl(k,\CC)\oplus\CC$ & $\su(k)\oplus\RR$\\
  \multicolumn{1}{|c||}{} & $M(k,\CC)$ ($k\geq2$) & $\sl(2k,\CC)$ & $\su(2k)$ & $\sl(k,\CC)\oplus\sl(k,\CC)\oplus\CC$ & $\su(k)\oplus\su(k)\oplus\RR$\\
  \multicolumn{1}{|c||}{} & $\Skew(2k,\CC)$ ($k\geq2$) & $\so(4k,\CC)$ & $\so(4k)$ & $\sl(2k,\CC)\oplus\CC$ & $\su(2k)\oplus\RR$\\
  \multicolumn{1}{|c||}{} & $\CC^k$ ($k\geq3$) & $\so(k+2,\CC)$ & $\so(k+2)$ & $\so(k,\CC)\oplus\CC$ & $\so(k)\oplus\RR$\\
  \multicolumn{1}{|c||}{} & $\Herm(3,\OO)_\CC$ & $\mathfrak{e}_7(\CC)$ & $\mathfrak{e}_7$ & $\mathfrak{e}_6(\CC)\oplus\CC$ & $\mathfrak{e}_6\oplus\RR$\\
  \hline
  \multicolumn{1}{|c||}{\multirow{3}*{\begin{tabular}{c}non-euclidean\\non-split\end{tabular}}} & $\Sym(2k,\CC)\cap M(k,\HH)$ ($k\geq2$) & $\sp(k,k)$ & $\sp(k)\oplus\sp(k)$ & $\su^*(2k)\oplus\RR$ & $\sp(k)$\\
  \multicolumn{1}{|c||}{} & $M(k,\HH)$ ($k\geq2$) & $\su^*(4k)$ & $\sp(2k)$ & $\su^*(2k)\oplus\su^*(2k)\oplus\RR$ & $\sp(k)\oplus\sp(k)$\\
  \multicolumn{1}{|c||}{} & $\RR^{k,0}$ ($k\geq2$) & $\so(k+1,1)$ & $\so(k+1)$ & $\so(k)\oplus\RR$ & $\so(k)$\\
  \hline
\end{tabular}
\caption{Simple real Jordan algebras and their corresponding Lie algebras \label{tb:Groups}}
\end{center}
\end{sidewaystable}

\addcontentsline{toc}{section}{References}
\bibliographystyle{amsplain}
\bibliography{bibdb}

\vspace{30pt}

\textsc{Joachim Hilgert\\Institut f\"ur Mathematik, Universit\"at Paderborn, Warburger Str. 100, 33098 Paderborn, Germany.}\\
\textit{E-mail address:} \texttt{hilgert@math.uni-paderborn.de}\\

\textsc{Toshiyuki Kobayashi\\IPMU and Graduate School of Mathematical Sciences, The University of Tokyo, 3-8-1 Komaba, Meguro, Tokyo, 153-8914, Japan.}\\
\textit{E-mail address:} \texttt{toshi@ms.u-tokyo.ac.jp}\\

\textsc{Jan M\"ollers\\Institut for Matematiske Fag, Aarhus Universitet, Ny Munkegade 118, Bygning 1530, Lokale 423, 8000 Aarhus C, Danmark.}\\
\textit{E-mail address:} \texttt{moellers@imf.au.dk}



\end{document}